\documentclass[hidelinks]{article}
\usepackage
[a4paper,% other options: a3paper, a5paper, etc
left=2.5cm,
right=2.5cm,
top=2.5cm,
bottom=2.5cm]
{geometry}

\usepackage{inputenc}
\usepackage{amsmath}
\usepackage{amsfonts}
\usepackage{amsthm}
\usepackage{amssymb}
\usepackage{latexsym}
\usepackage{enumitem} 
\usepackage{mathrsfs}
\usepackage{mathtools}
\usepackage[]{algorithm2e}
\usepackage{hyperref}
\usepackage{xcolor}
\usepackage{setspace}
\usepackage[date=year,doi=false,isbn=false,url=false,eprint=false,citestyle=numeric-comp,maxnames=99,giveninits]{biblatex}
\usepackage{bbm}
\usepackage{subcaption}
\usepackage{graphicx}

\renewbibmacro{in:}{}

\def\R{\mathbb{R}}

\DeclareMathOperator{\Var}{Var}

\DeclareMathOperator{\Uniform}{Uniform}
\DeclareMathOperator{\Rademacher}{Rademacher}

\DeclareMathOperator{\spn}{span}

\DeclareMathOperator{\Rem}{Rem}

\newtheorem{theorem}{Theorem}
\newtheorem{proposition}{Proposition}
\newtheorem{lemma}{Lemma}
\newtheorem{corollary}{Corollary}
\newtheorem{definition}{Definition}

\hypersetup{
    colorlinks=true,
    linkcolor=blue,
    filecolor=blue,      
    urlcolor=cyan,
    citecolor=red
}

\title{Minimax Signal Detection in Sparse Additive Models}
\author{Subhodh Kotekal and Chao Gao \\ \textit{University of Chicago}}
\date{}
\bibliography{spam_final}

\begin{document}
\maketitle

\abstract{Sparse additive models are an attractive choice in circumstances calling for modelling flexibility in the face of high dimensionality. We study the signal detection problem and establish the minimax separation rate for the detection of a sparse additive signal. Our result is nonasymptotic and applicable to the general case where the univariate component functions belong to a generic reproducing kernel Hilbert space. Unlike the estimation theory, the minimax separation rate reveals a nontrivial interaction between sparsity and the choice of function space. We also investigate adaptation to sparsity and establish an adaptive testing rate for a generic function space; adaptation is possible in some spaces while others impose an unavoidable cost. Finally, adaptation to both sparsity and smoothness is studied in the setting of Sobolev space, and we correct some existing claims in the literature.}

\section{Introduction}

In the interest of interpretability, computation, and circumventing the statistical curse of dimensionality plaguing high dimensional regression, structure is often assumed on the true regression function. Indeed, it might plausibly be argued that sparse linear regression is the distinguishing export of modern statistics. Despite its popularity, circumstances may call for more flexibility to capture nonlinear effects of the covariates. Striking a balance between flexibility and structure, Hastie and Tibshirani \cite{hastie_generalized_1986} proposed generalized additive models (GAMs) as a natural extension to the vaunted linear model. In a GAM, the regression function admits an additive decomposition of univariate (nonlinear) component functions. However, as in the linear model, the sample size must outpace the dimension for consistent estimation. Following modern statistical instinct, a sparse additive model is compelling \cite{lin_component_2006,raskutti_minimax-optimal_2012,koltchinskii_sparsity_2010,yuan_minimax_2016,meier_high-dimensional_2009,ravikumar_sparse_2009}. The regression function admits an additive decomposition of univariate functions for which only a small subset are nonzero; it is the combination of a GAM and sparsity.   

To fix notation, consider the \(p\)-dimensional Gaussian white noise model 
\begin{equation}\label{model:gwn}
    dY_x = f(x) \,dx + \frac{1}{\sqrt{n}} dB_x
\end{equation}
for \(x \in [0, 1]^p\). Though it may be more faithful to practical data analysis to consider the nonparametric regression model \(Y_i = f(X_i) + \epsilon_i\) (e.g. as in \cite{raskutti_minimax-optimal_2012}), the Gaussian white noise model is convenient as it avoids distracting technicalities while maintaining focus on the statistical essence. Indeed, the nonparametric statistics literature has a long history of studying the white noise model to understand theoretical limits, relying on well-known asymptotic equivalences \cite{brown_asymptotic_1996,reis_asymptotic_2008} which imply, under various conditions, that mathematical results obtained in one model can be ported over to the other model. As our focus is theoretical rather than practical, we follow in this tradition. The generalized additive model asserts the regression function is of the form  \(f(x) = \sum_{j=1}^{p} f_j(x_j)\) with \(f_1,...,f_p\) being univariate functions belonging to some function space \(\mathcal{H}\). Likewise, the sparse additive model asserts
\begin{equation*}
    f(x) = \sum_{j \in S} f_j(x_j)
\end{equation*}
for some unknown support set \(S\) of size \(s\) denoting the active covariates.

Most of the existing literature has addressed estimation of sparse additive models, primarily in the nonparametric regression setting and with \(\mathcal{H}\) being a reproducing kernel Hilbert space. After a series of works \cite{lin_component_2006,koltchinskii_sparsity_2010,meier_high-dimensional_2009,ravikumar_sparse_2009}, Raskutti et al. \cite{raskutti_minimax-optimal_2012} (see also \cite{koltchinskii_sparsity_2010}) established that a penalized \(M\)-estimator achieves the minimax estimation rate under various choices of the reproducing kernel Hilbert space \(\mathcal{H}\). Yuan and Zhou \cite{yuan_minimax_2016} establish minimax estimation rates under a notion of approximate sparsity. As is now seen as typical of estimation theory, the powerful framework of empirical processes is brought down to bear on their proofs. Some articles have also addressed generalizations of the sparse additive model. For example, the authors of \cite{yang_minimax-optimal_2015} consider, among other structures, an additive signal \(f(x) = \sum_{j=1}^{k} f_j(x)\) where each component function \(f_j\) is actually a multivariate function depending on at most \(s_j\) many coordinates and is \(\alpha_j\)-H\"{o}lder. The authors go on to derive minimax rates that handle heterogeneity in the smoothness indices and the sparsities of the coordinate functions; as a taste of their results, they show, in a particular regime and under some conditions, the rate \(kn^{-\frac{2\alpha}{2\alpha + s}} + ks \log\left(\frac{p}{s}\right)/n\) in the special, homogeneous case where \(s_j = s\) and \(\alpha_j = \alpha\) for all \(j\). Recently, the results of \cite{bhattacharya_deep_2023} show certain deep neural networks can achieve the minimax estimation rate for sparse \(k\)-way interaction models. The \(k\)-way interaction model is also known as nonparametric ANOVA. To give an example, the sparse \(2\)-way interaction model assumes \(f(x) = \sum_{j \in S_1} f_j(x_j) + \sum_{(k, l) \in S_2} f_{kl}(x_k, x_l)\) where the sets of active variables \(S_1\) and interactions \(S_2\) have small cardinalities. When the \(f_j\) are \(\beta_1\)-H\"{o}lder and the \(f_{kl}\) are \(\beta_2\)-H\"{o}lder, \cite{bhattacharya_deep_2023} establishes, under some conditions and up to factors logarithmic in \(n\), the rate \(s_1(n^{-\frac{2\beta_1}{2\beta_1 + 1}} + (\log p)/n) + s_2(n^{-\frac{2\beta_2}{2\beta_2 + 2}} + (\log p)/n)\).

The literature has much less to say on the problem of signal detection 
\begin{align}
    H_0 &: f \equiv 0, \label{problem:test_1}\\
    H_1 &: ||f||_2 \geq \varepsilon \text{ and } f \in \mathcal{F}_s \label{problem:test_2}
\end{align}
where \(\mathcal{F}_s\) is the class of sparse additive signals given by (\ref{param:Fs}). Adopting a minimax perspective \cite{burnasev_minimax_1979,ingster_minimax_1982,ingster_asymptotically_1987,ingster_adaptive_1990,ingster_nonparametric_2003}, the goal is to determine the smallest rate \(\varepsilon\) as a function of the sparsity level \(s\), the dimension \(p\), the sample size \(n\), and the function space \(\mathcal{H}\) such that consistent detection of the alternative against the null is possible. 

Though to a much lesser extent than the corresponding estimation theory, optimal testing rates have been established in various high dimensional settings other than sparse additive models. The most canonical setup, the Gaussian sequence model, is perhaps the most studied \cite{ingster_nonparametric_2003,baraud_non-asymptotic_2002,spokoiny_adaptive_1996,liu_minimax_2021,kotekal_minimax_2021,chhor_sparse_2022,donoho_higher_2004,collier_minimax_2017,carpentier_adaptive_2019,comminges_adaptive_2021,hall_innovated_2010}. Optimal rates have also been established in linear regression \cite{arias-castro_global_2011,ingster_detection_2010,mukherjee_minimax_2020} and other models \cite{mukherjee_testing_2022,deb_detecting_2020}. A common motif is that optimal testing rates exhibit different phenomena from optimal estimation rates. 

Returning to (\ref{problem:test_1})-(\ref{problem:test_2}), the only directly relevant works in the literature are Ingster and Lepski's article \cite{ingster_multichannel_2003} and a later article by Gayraud and Ingster \cite{gayraud_detection_2012}. Ingster and Lepski \cite{ingster_multichannel_2003} consider a sparse multichannel model which, after a transformation to sequence space, is closely related to (\ref{problem:test_1})-(\ref{problem:test_2}). Adopting an asymptotic setting and exclusively choosing \(\mathcal{H}\) to be a Sobolev space, they establish asymptotic minimax separation rates. However, their results only address the regimes \(p = O(s^2)\) and \(s = O(p^{1/2-\delta})\) for a constant \(\delta \in (0, 1/2)\). Their result does not precisely pin down the rate near the phase transition \(s \asymp \sqrt{p}\). In essence, their testing procedure in the sparse regime is to apply a Bonferroni correction to a collection of \(\chi^2\)-type tests, one test per each of the \(p\) coordinates. Thus, a gap in their rate near \(\sqrt{p}\) is unsurprising. In the dense regime, a single \(\chi^2\)-type test is used, as is typical in minimax testing literature. Ingster and Lepski \cite{ingster_multichannel_2003} also address adaptation to sparsity as well as adaptation both the sparsity and the smoothness. Ingster and Gayraud \cite{gayraud_detection_2012} consider sparse additive models rather than a sparse multichannel model but make the same choice of \(\mathcal{H}\) and work in an asymptotic setup. They establish sharp constants for the sparse case \(s = p^{1/2-\delta}\) via a Higher Criticism type testing statistic. Throughout the paper, we make comparisons of our results to primarily \cite{ingster_multichannel_2003} as it was the first article to establish rates. Our results do not address the question of sharp constants. 

Our paper's main contributions are the following. First, adopting a nonasymptotic minimax testing framework as initiated in \cite{baraud_non-asymptotic_2002}, we establish the nonasymptotic minimax separation rate for (\ref{problem:test_1})-(\ref{problem:test_2}) for any configuration of the sparsity \(s\), dimension \(p\), sample size \(n\), and a generic function space \(\mathcal{H}\). Notably, we do not restrict ourselves to Sobolev (or Besov) space as in \cite{ingster_multichannel_2003,gayraud_detection_2012}. The test procedure we analyze involves thresholded \(\chi^2\) statistics, following a strategy employed in other sparse signal detection problems \cite{collier_minimax_2017,kotekal_minimax_2021,chhor_sparse_2022,collier_estimation_2020,liu_minimax_2021}. 

Our second contribution is to establish an adaptive testing rate for a generic function space. Typically, the sparsity level is unknown, and it is of practical interest to have a methodology which can accommodate a generic \(\mathcal{H}\). Interestingly, some choices of \(\mathcal{H}\) do not involve any cost of adaptation, that is, the minimax rate can be achieved without knowing the sparsity. Our rate's incurred adaptation cost turns out to be a delicate function of \(\mathcal{H}\), thus extending Ingster and Lepski's focus on Sobolev spaces \cite{ingster_multichannel_2003}. Even in the Sobolev case, our result extends upon their article; near the regime \(s \asymp \sqrt{p}\), our test provides finer detail by incurring a cost involving only \(\log\log\log p\) instead of \(\log\log p\) as incurred by their test. In principle, our result can be used to reverse the process and find a space \(\mathcal{H}\) for which this adaptive rate incurs a given adaptation cost. 

Finally, adaptation to both sparsity and smoothness is studied in the context of Sobolev space. We identify an issue with and correct a claim made by \cite{ingster_multichannel_2003}. 

\subsection{Notation}\label{section:notation}
The following notation will be used throughout the paper. For \(p \in \mathbb{N}\), let \([p] := \{1,...,p\}\). For \(a, b \in \R\), denote \(a \vee b := \max\{a, b\}\) and \(a \wedge b = \min\{a , b\}\). Denote \(a \lesssim b\) to mean there exists a universal constant \(C > 0\) such that \(a \leq C b\). The expression \(a \gtrsim b\) means \(b \lesssim a\). Further, \(a \asymp b\) means \(a \lesssim b\) and \(b \lesssim a\). The symbol \(\langle \cdot, \cdot \rangle\) denotes the usual inner product in Euclidean space and \(\langle \cdot, \cdot\rangle_F\) denotes the Frobenius inner product. The total variation distance between two probability measures \(P\) and \(Q\) on a measurable space \((\mathcal{X}, \mathcal{A})\) is defined as \(d_{TV}(P, Q) := \sup_{A \in \mathcal{A}} |P(A) - Q(A)|\). If \(Q\) is absolutely continuous with respect to \(P\), the \(\chi^2\)-divergence is defined as \(\chi^2(Q\,||\,P) := \int_{\mathcal{X}} \left(\frac{dQ}{dP} - 1\right)^2 \, dP\). For a finite set \(S\), the notation \(|S|\) denotes the cardinality of \(S\). Throughout, iterated logarithms will be used (e.g. expressions like \(\log\log\log p\) and \(\log\log(np)\)). Without explicitly stating so, we will take such an expression to be equal to some universal constant if otherwise it would be less than one. For example, \(\log\log\log p\) should be understood to be equal to a universal constant when \(p < e^{e^e}\). 

\subsection{Setup}

\subsubsection{Reproducing Kernel Hilbert Space (RKHS)}
Following the literature (e.g. \cite{wainwright_high-dimensional_2019}), \(\mathcal{H}\) will denote a reproducing kernel Hilbert space (RKHS). Before discussing main results, basic properties of RKHSs are reviewed \cite{wainwright_high-dimensional_2019}. Suppose \(\mathcal{H} \subset L^2([0,1])\) is a reproducing kernel Hilbert space (RKHS) with associated inner product \(\langle \cdot, \cdot \rangle_{\mathcal{H}}\). There exists a symmetric function \(K : [0,1] \times [0, 1] \to \R_{+}\) called a \emph{kernel} such that for any \(x \in [0, 1]\) we have (1) the function \(K(\cdot, x) \in \mathcal{H}\) and (2) for all \(f \in \mathcal{H}\) we have \(f(x) = \langle f, K(\cdot, x)\rangle_{\mathcal{H}}\). Mercer's theorem (Theorem 12.20 in \cite{wainwright_high-dimensional_2019}) guarantees that the kernel \(K\) admits an expansion in terms of eigenfunctions \(\left\{ \psi_k \right\}_{k=1}^{\infty}\), namely \(K(x, x') = \sum_{k=1}^{\infty} \mu_k \psi_k(x)\psi_k(x')\). To give examples, the kernel \(K(x, x') = 1 + \min\{x, x'\}\) defines the first-order Sobolev space with eigenvalue decay \(\mu_k \asymp k^{-2}\), and the kernel \(K(x, x') = \exp\left(-\frac{(x-x')^2}{2}\right)\) exhibits eigenvalue decay \(\mu_k \asymp e^{-ck \log k}\) (see \cite{wainwright_high-dimensional_2019} for a more detailed review).

Without loss of generality, we order the eigenvalues \(\mu_1 \geq \mu_2 \geq ... \geq 0\). The eigenfunctions \(\left\{\psi_k\right\}_{k=1}^{\infty}\) are orthonormal in \(L^2([0,1])\) under the usual \(L^2\) inner product \(\langle \cdot, \cdot \rangle_{L^2}\) and the inner product \(\langle \cdot, \cdot \rangle_{\mathcal{H}}\) satisfies \(\langle f, g \rangle_{\mathcal{H}} = \sum_{k=1}^{\infty} \mu_k^{-1} \langle f, \psi_k \rangle_{L^2} \langle g, \psi_k \rangle_{L^2}\) for \(f, g \in \mathcal{H}\). 

\subsubsection{Parameter space}\label{section:parameter}
The parameter space which will be used throughout the paper is defined in this section. Suppose \(\mathcal{H}\) is an RKHS. Recall we are interested in sparse additive signals \(f(x) = \sum_{j \in S} f_j(x_j)\) for some sparsity pattern \(S \subset [p]\). Following \cite{gayraud_detection_2012,raskutti_minimax-optimal_2012}, for convenience we assume \(\int_{0}^{1} f_j(t) \, dt = 0\) for all \(j\). This assumption is mild and can be relaxed; further discussion can be found in Section \ref{section:uncentered}. Letting \(\mathbf{1} \in L^2([0,1])\) denote the constant function equal to one, consider that \(\mathcal{H}_0 := \mathcal{H} \cap \spn\left\{\mathbf{1}\right\}^\perp\) is a closed subspace of \(\mathcal{H}\). Hence, \(\mathcal{H}_0\) is also an RKHS. We will put aside \(\mathcal{H}\) (along with its eigenfunctions and eigenvalues) and only work with \(\mathcal{H}_0\). Let \(\{\psi_k\}_{k=1}^{\infty}\) and \(\left\{\mu_k\right\}_{k=1}^{\infty}\) denote its associated eigenfunctions and eigenvalues respectively. Following \cite{baraud_non-asymptotic_2002}, we assume \(\mu_1 = 1\) also for convenience; this can easily be relaxed. 

For each subset \(S \subset [p]\), define 
\begin{equation*}
    \mathcal{H}_S := \left\{f(x) = \sum_{j \in S} f_j(x_j) : f_j \in \mathcal{H}_0 \text{ and } ||f_j||_{\mathcal{H}_0} \leq 1 \text{ for all } j \in S\right\}. 
\end{equation*}
The condition on the RKHS norm enforces regularity. Define the parameter space 
\begin{equation}\label{param:Fs}
    \mathcal{F}_s := \bigcup_{\substack{S \subset [p], \\ |S| \leq s}} \mathcal{H}_S
\end{equation}
for each \(1 \leq s \leq p\). 

Following \cite{ingster_multichannel_2003,gayraud_detection_2012}, it is convenient to transform (\ref{model:gwn}) from function space to sequence space. The tensor product \(L^2([0,1])^{\otimes p}\) admits the orthonormal basis \(\left\{ \Phi_\ell \right\}_{\ell \in \mathbb{N}^p}\) with \(\Phi_\ell(t) = \prod_{j=1}^{p} \psi_{\ell_j}(t_j)\) for \(t \in [0, 1]^p\). For ease of notation, denote \(\phi_{j,k}(t) = \psi_k(t_j)\) for \(k \in \mathbb{N}\) and \(j \in [p]\). Define the random variables
\begin{align}
    X_{k,j} &:= \int_{[0, 1]^p} \phi_{j,k}(x) \, dY_x \sim N(\theta_{k,j}, n^{-1}) \label{data:X}
\end{align}
where \(\theta_{k,j} = \int_{0}^{1} \psi_k(x) f_j(x) \, dx\). The assumption \(\int_{0}^{1} f_j(t) \, dt = 0\) for all \(j\) is used here. Note by orthogonality that \(\{X_{k,j}\}_{k \in \mathbb{N}, j \in [p]}\) is a collection of independent random variables. The notation \(\Theta \in \R^{\mathbb{N} \times p}\) will frequently be used to denote the full matrix of coefficients. The notation \(\Theta_j \in \R^{\mathbb{N}}\) will also be used to denote the \(j\)th column of \(\Theta\). For \(f \in \mathcal{F}_s\), the corresponding set of coefficients is 
\begin{equation}\label{param:Ts}
    \mathscr{T}_s := \left\{ \Theta \in \R^{\mathbb{N} \times p} : \sum_{j=1}^{p} \mathbbm{1}_{\{\Theta_j \neq 0\}} \leq s \text{ and } \sum_{k=1}^{\infty} \frac{\theta_{k,j}^2}{\mu_{k}} \leq 1 \text{ for all } j \right\}.
\end{equation}

Note the parameter spaces \(\mathcal{F}_s\) and \(\mathscr{T}_s\) are in correspondence, and we will frequently write \(f\) and \(\Theta\) freely in the same context without comment. The understanding is \(f\) is a function and \(\Theta\) is its corresponding basis coefficients. The notation \(E_f, P_f, E_\Theta\), and \(P_\Theta\) will be used to denote expectations and probability measures with respect to the denoted parameters. 

\subsubsection{Problem}
As described earlier, given an observation from (\ref{model:gwn}) the signal detection problem (\ref{problem:test_1})-(\ref{problem:test_2}) is of interest. The goal is to characterize the nonasymptotic minimax separation rate \(\varepsilon_{\mathcal{H}}^* = \varepsilon_{\mathcal{H}}^*(p, s, n)\).
\begin{definition}\label{def:minimax_separation_rate}
    We say \(\varepsilon^*_{\mathcal{H}} = \varepsilon_{\mathcal{H}}^*(p, s, n)\) is the nonasymptotic minimax separation rate for the problem (\ref{problem:test_1})-(\ref{problem:test_2}) if 
    \begin{enumerate}[label=(\roman*)]
        \item for all \(\eta \in (0, 1)\), there exists \(C_\eta > 0\) depending only on \(\eta\) such that for all \(C > C_\eta\),
        \begin{equation*}
            \inf_{\varphi}\left\{ P_0\left\{\varphi \neq 0\right\} + \sup_{\substack{f \in \mathcal{F}_s, \\ ||f||_2 \geq C \varepsilon_{\mathcal{H}}^*(p, s, n)}} P_f\left\{\varphi \neq 1\right\}\right\} \leq \eta, 
        \end{equation*}
        \item for all \(\eta \in (0, 1)\), there exists \(c_\eta > 0\) depending only on \(\eta\) such that for all \(0 < c < c_\eta\),
        \begin{equation*}
            \inf_{\varphi}\left\{ P_0\left\{\varphi \neq 0\right\} + \sup_{\substack{f \in \mathcal{F}_s, \\ ||f||_2 \geq c\varepsilon_{\mathcal{H}}^*(p, s, n)}} P_f\left\{\varphi \neq 1\right\}\right\} \geq 1-\eta. 
        \end{equation*}
    \end{enumerate}
\end{definition}

The nonasymptotic minimax testing framework was initiated by Baraud \cite{baraud_non-asymptotic_2002} and characterizes (up to universal factors) the fundamental statistical limit of the detection problem. The framework is nonasymptotic in the sense that the conditions for the minimax separation rate hold for any configuration of the problem parameters.

Since \(||f||_2^2 = \sum_{j=1}^{p} \sum_{k=1}^{\infty} \theta_{k,j}^2\), the problem can be equivalently formulated in sequence space as 
\begin{align}
    H_0 &: \Theta \equiv 0, \label{problem:seq_test_1}\\
    H_1 &: ||\Theta||_F \geq \varepsilon \text{ and } \Theta \in \mathscr{T}_s. \label{problem:seq_test_2}
\end{align}

This testing problem, with the parameter space (\ref{param:Ts}), is interesting its own right outside the sparse additive model and RKHS context. Indeed, the testing problem (\ref{problem:seq_test_1})-(\ref{problem:seq_test_2}) is essentially a group-sparse extension of the detection problem in ellipses considered in Baraud's seminal article \cite{baraud_non-asymptotic_2002}. In fact, this interpretation was actually our initial motivation to study the detection problem. The connection to sparse additive models was a later consideration; similar to the way in which the later article \cite{gayraud_detection_2012} considers sparse additive models when building upon the earlier, fundamental work \cite{ingster_multichannel_2003} dealing with a sparse multichannel (essentially group-sparse) model. Taking the perspective of a sequence problem has a long history in nonparametric regression \cite{ingster_nonparametric_2003,ingster_multichannel_2003,baraud_non-asymptotic_2002,ermakov_minimax_1990,johnstone_gaussian_nodate,wei_local_2020,reis_asymptotic_2008,brown_asymptotic_1996} due to not only its fundamental connections but also its advantage in distilling the problem to its essence and dispelling technical distractions. Our results can be exclusively read (and readers are encouraged to do so) in the context of the sequence problem (\ref{problem:seq_test_1})-(\ref{problem:seq_test_2}).

\section{Minimax rates}
We begin by describing some high-level and natural intuition before informally stating our main result in Section \ref{section:naive_ansatz}. Section \ref{section:preliminaries} contains the development of some key quantities. In Sections \ref{section:lower_bound} and \ref{section:upper_bound}, we formally state minimax lower and upper bounds respectively. In Section \ref{section:special_cases}, some special cases illustrating the general result are discussed.

\subsection{A naive ansatz}\label{section:naive_ansatz}

A first instinct is to look to previous results for context in an attempt to make a conjecture regarding the optimal testing rate. To illustrate how this line of thinking might proceed, consider the classical Gaussian white noise model on the unit interval in one dimension,
\begin{equation*}
    dY_x = f(x)\, dx + \frac{1}{\sqrt{n}} dB_x
\end{equation*}
for \(x \in [0, 1]\). Assume \(f\) lives inside the unit ball of a reproducing kernel Hilbert space \(\mathcal{H}\) and thus admits a decomposition in the associated orthonormal basis with a coefficient vector \(\theta \in \ell^2(\mathbb{N})\) living in an ellipsoid determined by the kernel's eigenvalues \(\mu_1 \geq \mu_2 \geq ... \geq 0\). The optimal rate of estimation is given by Birg\'{e} and Massart \cite{birge_gaussian_2001}
\begin{equation*}
    \epsilon_{\text{est}}^2 \asymp \max_{\nu \in \mathbb{N}}\left\{ \mu_{\nu} \wedge \frac{\nu}{n} \right\}.
\end{equation*}
Baraud \cite{baraud_non-asymptotic_2002} established that the minimax separation rate for the signal detection problem 
\begin{align*}
    H_0 &: f \equiv 0, \\
    H_1 &: ||f||_2 \geq \varepsilon \text{ and } ||f||_{\mathcal{H}} \leq 1
\end{align*}
is given by 
\begin{equation*}
    \epsilon_{\text{test}}^2 \asymp \max_{\nu \in \mathbb{N}} \left\{\mu_{\nu} \wedge \frac{\sqrt{\nu}}{n}\right\}. 
\end{equation*}
In both estimation and testing, the maximizer \(\nu^*\) can be conceptualized as the correct truncation order. Specifically for testing, Baraud's procedure \cite{baraud_non-asymptotic_2002}, working in sequence space, rejects \(H_0\) when \(\sum_{k=1}^{\nu^*} X_k^2 - \frac{\nu^*}{n} \gtrsim \frac{\sqrt{\nu^*}}{n}\). The data for \(k > \nu^*\) are not used, and it is in this sense \(\nu^*\) is understood as a truncation level. To illustrate these results, the rates for Sobolev spaces with smoothness \(\alpha\) are \(\epsilon_{\text{est}}^2 \asymp n^{-\frac{2\alpha}{2\alpha+1}}\) and \(\epsilon_{\text{test}}^2 \asymp n^{-\frac{4\alpha}{4\alpha+1}}\) since \(\mu_\nu \asymp \nu^{-2\alpha}\). By now, these nonasymptotic results of \cite{birge_gaussian_2001,baraud_non-asymptotic_2002} are well known and canonical.  

Moving to the setting of sparse additive signals in the model (\ref{model:gwn}), Raskutti et al. \cite{raskutti_minimax-optimal_2012} derive the nonasymptotic minimax rate of estimation
\begin{equation*}
    \frac{s \log p}{n} + s \epsilon_{\text{est}}^2.
\end{equation*}
While their upper bound holds for any choice of \(\mathcal{H}\) (satisfying some mild conditions) and sparsity level \(s\), they only obtain a matching lower bound when \(s = o(p)\) and when the unit ball of \(\mathcal{H}\) has logarithmic or polynomial scaling metric entropy. This rate obtained by Raskutti et al. \cite{raskutti_minimax-optimal_2012} is pleasing. It is quite intuitive to see the term \(\frac{s \log p}{n}\) due to the sparsity in the parameter space. The term \(s \epsilon_{\text{est}}^2\) is the natural rate for estimating \(s\) many univariate functions in \(\mathcal{H}\) as if sparsity pattern were known. Notably, there is no interaction between the choice of \(\mathcal{H}\) and the sparsity structure. The sparsity term \(\frac{s \log p}{n}\) is independent of \(\mathcal{H}\) and the estimation term \(s \epsilon_{\text{est}}^2\) is dimension free.

One might intuit that this lack of interaction is a general phenomenon. Instinct may suggest that signal detection in sparse additive models is also just \(s\) many instances of a univariate nonparametric detection problem plus the problem of a detecting a signal which is nonzero on an unknown sparsity pattern. Framing it as two distinct problems, one might conjecture the optimal testing rate should be \(s \epsilon_{\text{test}}^2\) plus the \(s\)-sparse detection rate. Collier et al. \cite{collier_minimax_2017} provide a natural candidate for the \(s\)-sparse testing rate, namely the rate \(\frac{s \log\left(\frac{p}{s^2}\right)}{n}\) for \(s < \sqrt{p}\) and \(\frac{\sqrt{p}}{n}\) for \(s \geq \sqrt{p}\). 

However, this is not the case as a quick glance at \cite{ingster_multichannel_2003} falsifies the conjecture for the case of Sobolev \(\mathcal{H}\). Though quick to dispel hopeful thinking, \cite{ingster_multichannel_2003} expresses little about the interaction between sparsity and \(\mathcal{H}\). Our result explicitly captures this nontrivial interaction for a generic \(\mathcal{H}\). We show the minimax separation rate is given by
\begin{align}\label{rate:main}
    \varepsilon_{\mathcal{H}}^*(p, s, n)^2 &\asymp 
    s \wedge \begin{cases}
        \frac{s \log\left(\frac{p}{s^2}\right)}{n} + s \cdot \max_{\nu \in \mathbb{N}} \left\{\mu_{\nu} \wedge \frac{\sqrt{\nu \log\left(1 + \frac{p}{s^2}\right)}}{n} \right\} &\textit{if } s < \sqrt{p}, \\
        s \cdot \max_{\nu \in \mathbb{N}} \left\{\mu_{\nu} \wedge \frac{\sqrt{\nu \log\left(1 + \frac{p}{s^2}\right)}}{n} \right\} &\textit{if } s \geq \sqrt{p}. 
    \end{cases} 
\end{align}

The rate bears some resemblance to the sparse testing rate of Collier et al. \cite{collier_minimax_2017} and the nonparametric testing rate of Baraud \cite{baraud_non-asymptotic_2002}, but the combination of the two is not \textit{a priori} straightforward. At this point in discussion, not enough has been developed to explain how the form of the rate arises. Various features of the rate will be commented on later on in the paper. 

\subsection{Preliminaries}\label{section:preliminaries}
In this section, some key pieces are defined. 
\begin{definition}
    Define \(\Gamma_{\mathcal{H}} = \Gamma_{\mathcal{H}}(p, s, n)\) to be the quantity 
    \begin{equation*}
        \Gamma_{\mathcal{H}} := \max_{\nu \in \mathbb{N}} \left\{ \mu_{\nu} \wedge \frac{\sqrt{\nu \log\left(1 + \frac{p}{s^2}\right)}}{n}\right\}.
    \end{equation*} 
\end{definition}
Note, since \(\mu_1 = 1\), it follows \(\Gamma_{\mathcal{H}} \gtrsim \frac{1}{n}\). It is readily seen from (\ref{rate:main}) there are two broad regimes to consider. When \(n \lesssim \log(1 + p/s^2)\), we have \(\varepsilon_{\mathcal{H}}^*(p, s, n)^2 \asymp s\). In the regime \(n \gtrsim \log(1 + p/s^2)\), the rate is more complicated. The first regime is really a trivial regime since any signal \(f \in \mathcal{F}_s\) must satisfy \(||f||_2^2 \leq s\) by virtue of \(\mu_1 = 1\). Therefore, the degenerate test which always accepts \(H_0\) vacuously detects sparse additive signals with \(||f||_2^2 \geq 2s\). Hence, the upper bound \(\varepsilon_{\mathcal{H}}^*(p, s, n)^2 \lesssim s\) is trivially achieved. It turns out a matching lower bound can be proved which establishes that the regime \(n \lesssim \log(1 + p/s^2)\) is fundamentally trivial; see Section \ref{section:lower_bound} for a formal statement.

More generally, the form (\ref{rate:main}) is useful when discussing lower bounds. A different form is more convenient when discussing upper bounds, and it is a form which is familiar in the context of \cite{collier_minimax_2017,collier_optimal_2018}. 

\begin{definition}
    Define \(\nu_{\mathcal{H}}\) to be the smallest positive integer \(\nu\) such that 
    \begin{equation*}
        \mu_{\nu} \leq \frac{\sqrt{\nu \log\left(1+\frac{p}{s^2}\right)}}{n}.
    \end{equation*}
\end{definition}

As the next lemma shows, \(\nu_{\mathcal{H}}\) is essentially the solution to the maximization problem in the definition of \(\Gamma_{\mathcal{H}}\). Drawing an analogy to the result of Baraud \cite{baraud_non-asymptotic_2002} described in Section \ref{section:naive_ansatz}, \(\nu_{\mathcal{H}}\) can be conceptualized as the correct order of truncation accounting for the dimension and sparsity. 

\begin{lemma}\label{lemma:fixed_point_equiv}
    If \(\log\left(1 + \frac{p}{s^2}\right) \leq \frac{n}{2}\), then \(\Gamma_{\mathcal{H}} \leq \frac{\sqrt{\nu_{\mathcal{H}} \log\left(1 + \frac{p}{s^2}\right)}}{n}\leq \sqrt{2} \Gamma_{\mathcal{H}}\). 
\end{lemma}

With Lemma \ref{lemma:fixed_point_equiv}, the testing rate can be expressed as  
\begin{equation}\label{rate:main_expanded}
    \varepsilon^*_{\mathcal{H}}(p, s, n)^2 \asymp 
    \begin{cases}
        \frac{s}{n}\log\left(1 + \frac{p}{s^2}\right) + \frac{s}{n} \sqrt{\nu_{\mathcal{H}} \log\left(1 + \frac{p}{s^2}\right)} &\textit{if } s < \sqrt{p}, \\
        \frac{\sqrt{p \nu_{\mathcal{H}}}}{n} &\textit{if } s \geq \sqrt{p}. 
    \end{cases}
\end{equation}

The condition \(\log\left(1 + \frac{p}{s^2}\right) \lesssim n\) in Lemma \ref{lemma:fixed_point_equiv} is natural in light of the triviality which occurs when \(n \lesssim \log\left(1 + \frac{p}{s^2}\right)\).

\subsection{Lower bound}\label{section:lower_bound}
In this section, a formal statement of a lower bound on the minimax separation rate is given. Define 
\begin{equation}\label{rate:psi_lbound}
    \psi(p, s, n)^2 := 
    \begin{cases}
        \frac{s \log\left(1+\frac{p}{s^2}\right)}{n} \vee s\Gamma_{\mathcal{H}} &\textit{if } s < \sqrt{p},\\
        s\Gamma_{\mathcal{H}} &\textit{if } s \geq \sqrt{p}.
    \end{cases}
\end{equation}

First, it is shown that the testing problem is trivial if \(\log\left(\frac{p}{s^2}\right) \gtrsim n\).

\begin{proposition}[Triviality]\label{prop:lbound_trivial}
    Suppose \(1 \leq s \leq p\). Suppose \(\kappa > 0\) and \(\log\left(1 + \frac{p}{s^2}\right) \geq \kappa n \). If \(\eta \in (0, 1)\), then for any \(0 < c < 1 \wedge \sqrt{\kappa} \wedge \sqrt{\kappa \log\left(1 + 4\eta^2\right)}\) we have 
    \begin{equation*}
        \inf_{\varphi}\left\{ P_0\left\{ \varphi \neq 0 \right\} + \sup_{\substack{f \in \mathcal{F}_s, \\ ||f||_2 \geq c\sqrt{s} }} P_{f}\left\{\varphi \neq 1\right\} \right\} \geq 1 - \eta. 
    \end{equation*}
\end{proposition}

Proposition \ref{prop:lbound_trivial} asserts that in order to achieve small testing error, a necessary condition is that \(||f||_2 \geq C \sqrt{s}\) for a sufficiently large \(C > 0\). To see why Proposition \ref{prop:lbound_trivial} is a statement about triviality, observe that \(\mathcal{F}_s \subset \left\{f \in L^2([0, 1]^p) : ||f||_2 \leq \sqrt{s}\right\}\). To reiterate plainly, all potential \(f \in \mathcal{F}_s\) in the model (\ref{model:gwn}) live inside the ball of radius \(\sqrt{s}\). There are essentially no functions \(f \in \mathcal{F}_s\) that have detectable norm when \(\log\left(\frac{p}{s^2}\right) \gtrsim n\), and so the problem is essentially trivial. 

The lower bound construction is straightforward. Working in sequence space, consider an alternative hypothesis with a prior \(\pi\) in which a draw \(\Theta \sim \pi\) is obtained by drawing a size \(s\) subset \(S \subset [p]\) uniformly at random and setting \(\theta_{k,j} = \rho\) if \(k = 1\) and \(j \in S\) or setting \(\theta_{k,j} = 0\) otherwise. The value of \(\rho\) determines the separation between the null and alternative hypotheses since \(||\Theta||_F^2 = s\rho^2\). However, \(\rho\) must respect some constraints to ensure \(\pi\) is supported on the parameter space and that it is impossible to distinguish the null and alternative hypotheses with vanishing error. Observe this construction places us in a one-dimensional sparse sequence model (since \(\theta_{k,j} = 0\) for \(k \geq 2\)), which is precisely the setting of \cite{collier_minimax_2017}. From their results, it is seen we must have \(\rho^2 \lesssim \log\left(1 + \frac{p}{s^2}\right)/n\). To ensure \(\pi\) is supported on the parameter space, we must have \(\sum_{k=1}^{\infty} \theta_{k,j}^2/\mu_k \leq 1\) for all \(j \in [p]\). Since \(\mu_1 = 1\), it follows \(\sum_{k=1}^{\infty} \theta_{k,j}^2/\mu_k = \theta_{1,j}^2 \leq \rho^2\), and so the constraint \(\rho \lesssim 1\) must be enforced. When \(\log\left(1 + p/s^2\right) \gtrsim n\), only the second condition \(\rho \lesssim 1\) is binding, and so the largest separation we can achieve is \(||\Theta||_F^2 \asymp s\). Hence, the problem is trivial in this regime.

To ensure non-triviality, it will be assumed \(\log\left(1 + \frac{p}{s^2}\right) \leq \frac{n}{2}\). The choice of the factor \(1/2\) is only for convenience and is not essential. In fact, the condition \(\log\left(1 + \frac{p}{s^2}\right) < n\) would always suffice for our purposes, and the condition \(\log\left(1 + \frac{p}{s^2}\right) \leq n\) would also suffice for \(n > 1\). The following theorem establishes \(\varepsilon_{\mathcal{H}}^*(p, s, n)^2 \gtrsim \psi(p, s, n)^2\). 
\begin{theorem}\label{thm:lbound_nontrivial}
    Suppose \(1 \leq s \leq p\) and \(\log\left(1 + \frac{p}{s^2}\right) \leq \frac{n}{2}\). If \(\eta \in (0, 1)\), then there exists \(c_\eta > 0\) depending only on \(\eta\) such that for any \(0 < c < c_\eta\) we have 
    \begin{equation*}
        \inf_{\varphi}\left\{ P_0\left\{ \varphi \neq 0 \right\} + \sup_{\substack{f \in \mathcal{F}_s, \\ ||f||_2 \geq c \psi(p, s, n) }} P_{f}\left\{\varphi \neq 1\right\} \right\} \geq 1 - \eta
    \end{equation*}
    where \(\psi\) is given by (\ref{rate:psi_lbound}).
\end{theorem}
The lower bound is proved via Le Cam's two point method (also known as the method of ``fuzzy'' hypotheses \cite{tsybakov_introduction_2009}) which is standard in the minimax hypothesis testing literature \cite{ingster_nonparametric_2003}. The prior distribution employed in the argument is a natural construction. Namely, the active set (i.e. the size \(s\) set of nonzero coordinate functions) is selected uniformly at random and the corresponding nonzero coordinate functions are drawn from the usual univariate nonparametric prior employed in the literature \cite{baraud_non-asymptotic_2002, ingster_nonparametric_2003}. 

\subsection{Upper bound}\label{section:upper_bound}
In this section, testing procedures are constructed and formal statements establishing rate-optimality are made. The form of the rate in (\ref{rate:main_expanded}) is a more convenient target, and should be kept in mind when reading the statements of our upper bounds. 

\subsubsection{Hard thresholding in the sparse regime}
In this section, the sparse regime \(s < \sqrt{p}\) is discussed. For any \(d \in \mathbb{N}\) and \(j \in [p]\), define \(E_j(d) = n \sum_{k \leq d} X_{k, j}^2\) where the data \(\{X_{k,j}\}_{k \in \mathbb{N}, j \in [p]}\) is defined via transformation to sequence space (\ref{data:X}). For any \(r \geq 0\), define 
\begin{equation}\label{def:T_r}
    T_r(d) := \sum_{j = 1}^{p} \left(E_{j}(d) - \alpha_r(d)\right) \mathbbm{1}_{\{E_j(d) \geq d + r^2\}}
\end{equation}
where 
\begin{equation}\label{def:alpha}
    \alpha_r(d) := \frac{E\left(||g||^2 \mathbbm{1}_{\{||g||^2 \geq d + r^2\}}\right)}{P\{||g||^2 \geq d + r^2\}}
\end{equation}
where \(g \sim N(0, I_d)\). Note \(\alpha_r(d)\) is a conditional expectation under the null hypothesis. The random variable \(T_r(d)\) will be used as a test statistic. Such a statistic was first defined by Collier et al. \cite{collier_minimax_2017}, and similar statistics have been successfully employed in other signal detection problems \cite{collier_estimation_2020,liu_minimax_2021,kotekal_minimax_2021,chhor_sparse_2022}. However, all previous works in the literature have only used this statistic with \(d = 1\). For our problem, it is necessary to take growing \(d\). Consequently, a more refined analysis is necessary, and essentially new phenomena appear in the upper bound as we discuss later. As noted in Section \ref{section:preliminaries}, the quantity \(\nu_{\mathcal{H}}\) can be conceptualized of as the correct truncation order, that is, it turns out the correct choice is \(d \asymp \nu_{\mathcal{H}}\). The choice of \(r\) is more complicated, and in fact there are two separate regimes depending on the size of \(d\). The regime in which \(d \gtrsim \log\left(1 + \frac{p}{s^2}\right)\) is referred to as the ``bulk'' regime. The complementary regime \(d \lesssim \log\left(1 + \frac{p}{s^2}\right)\) is referred to as the ``tail'' regime.    

\begin{proposition}[Bulk]\label{prop:test_sparse_bulk}
    Set \(d = \nu_{\mathcal{H}} \vee \lceil D \rceil \) where \(D\) is the universal constant from Lemma \ref{lemma:null_bulk}. There exist universal positive constants \(K_1,K_2,\) and \(K_3\) such that the following holds. Suppose \(1 \leq s < \sqrt{p}\) and \(\log\left(1 + \frac{p}{s^2}\right) \leq K_3^2 d\). Set 
    \begin{equation*}
        \tau(p, s, n)^2 = \frac{s\sqrt{\nu_{\mathcal{H}} \log\left(1 + \frac{p}{s^2}\right)}}{n}.
    \end{equation*}
    If \(\eta \in (0, 1)\), then there exists \(C_\eta > 0\) depending only on \(\eta\) such that for all \(C > C_\eta\), we have 
    \begin{equation*}
        P_0\left\{T_r(d) > C K_1 n\tau(p, s, n)^2\right\} + \sup_{\substack{f \in \mathcal{F}_s, \\ ||f||_2 \geq C \tau(p, s, n)}} P_f\left\{ T_r(d) \leq C K_1 n\tau(p, s, n)^2 \right\} \leq \eta
    \end{equation*}
    where \(r = K_2 \left(d \log\left(1 + \frac{p}{s^2}\right)\right)^{1/4}\). Here, \(T_r(d)\) is given by (\ref{def:T_r}). 
\end{proposition}

\begin{proposition}[Tail]\label{prop:test_sparse_tail}
    Set \(d = \nu_{\mathcal{H}} \vee \lceil D \rceil \) where \(D\) is the universal constant from Lemma \ref{lemma:null_tail}. Let \(K_3\) denote the universal positive constant from Proposition \ref{prop:test_sparse_bulk}. There exist universal positive constants \(K_1\) and \(K_2\) such that the following holds. Suppose \(1 \leq s < \sqrt{p}\) and \(\log\left(1 + \frac{p}{s^2}\right) \geq K_3^2 d\). Set 
    \begin{equation*}
        \tau(p, s, n)^2 = \frac{s \log\left(1 + \frac{p}{s^2}\right)}{n}.    
    \end{equation*}
    If \(\eta \in (0, 1)\), then there exists \(C_\eta > 0\) depending only on \(\eta\) such that for all \(C > C_\eta\), we have
    \begin{equation*}
        P_0\left\{T_r(d) > C K_1 n\tau(p, s, n)^2\right\} + \sup_{\substack{f \in \mathcal{F}_s, \\ ||f||_2 \geq C \tau(p, s, n)}} P_f\left\{ T_r(d) \leq C K_1 n\tau(p, s, n)^2 \right\} \leq \eta
    \end{equation*}
    where \(r = K_2 \sqrt{\log\left(1 + \frac{p}{s^2}\right)}\). Here, \(T_r(d)\) is given by (\ref{def:T_r}). 
\end{proposition}
Propositions \ref{prop:test_sparse_tail} and \ref{prop:test_sparse_bulk} thus imply that, for \(s < \sqrt{p}\), the minimax separation rate is upper bounded by \(s\log\left(1 + \frac{p}{s^2}\right)/n + s \sqrt{\nu_{\mathcal{H}}\log\left(1 + \frac{p}{s^2}\right)}/n\). By Lemma \ref{lemma:fixed_point_equiv}, it follows that  
\begin{equation*}
    \varepsilon^*_{\mathcal{H}}(p, s, n) \lesssim \frac{s\log\left(1 + \frac{p}{s^2}\right)}{n} + s \Gamma_{\mathcal{H}}
\end{equation*}
under the condition that \(\log\left(1 + \frac{p}{s^2}\right) \leq \frac{n}{2}\). As established by Proposition \ref{prop:lbound_trivial} in Section \ref{section:lower_bound}, a condition like this is essential to avoid triviality. 

Propositions \ref{prop:test_sparse_tail} and \ref{prop:test_sparse_bulk} reveal an interesting phase transition in the minimax separation rate at the point 
\begin{equation*}
    \nu_{\mathcal{H}} \asymp \log\left(1 + \frac{p}{s^2}\right).
\end{equation*}
This phase transition phenomena is driven by the tail behavior of \(\chi^2_d\). Consider that under the null distribution \(E_j(d) \sim \chi^2_d\) for all \(j\), and so the statistic \(T_r(d)\) is the sum of \(p\) independent and thresholded \(\chi^2_d\) random variables. By a well-known lemma of Laurent and Massart \cite{laurent_adaptive_2000} (also from Bernstein's inequality up to constants), for any \(u > 0\) we have 
\begin{equation}\label{eqn:laurent_massart_tail}
    P\left\{\chi^2_d - d \geq 2 \sqrt{d u} + 2u\right\} \leq e^{-u}.
\end{equation}
Roughly speaking, \(\chi^2_d - d\) exhibits subgaussian-type deviation in the ``bulk'' \(u \lesssim \sqrt{d}\) and subexponential-type deviation in the ``tail'' \(u \gtrsim \sqrt{d}\). Consequently, \(T_r(d)\) should be intuited as a sum of thresholded subgaussian random variables when \(r \lesssim \sqrt{d}\), and as a sum of thresholded subexponential random variables when \(r \gtrsim \sqrt{d}\). 

Examining (\ref{eqn:laurent_massart_tail}), in the ``tail'' \(u \gtrsim \sqrt{d}\) it follows \(2\sqrt{du} + 2u \asymp u\), which no longer exhibits dependence on \(d\). Analogously, in Proposition \ref{prop:test_sparse_tail} the threshold is taken as \(r \gtrsim \sqrt{d}\) and so the resulting rate exhibits no dependence on \(d\) and consequently no dependence on \(\mathcal{H}\). On the other hand, in the ``bulk'' \(u \lesssim \sqrt{d}\) it follows \(2 \sqrt{du} + 2u \asymp \sqrt{du}\). Analogously, in Proposition \ref{prop:test_sparse_bulk} the threshold is taken as \(r \lesssim \sqrt{d}\) and so the resulting rate indeed depends on \(d\) and thus on \(\mathcal{H}\). 

\subsubsection{\texorpdfstring{\(\chi^2\) tests in the dense regime}{Chi-squared tests in the dense regime}}

The situation is less complicated in the dense regime \(s \geq \sqrt{p}\) as it suffices to use the \(\chi^2\) testing statistic \(T := n \sum_{j=1}^{p} \sum_{k \leq \nu_{\mathcal{H}}} X_{k,j}^2\). 

\begin{proposition}\label{prop:dense_alpha}
    Suppose \(s \geq \sqrt{p}\). Set
    \begin{equation*}
        \tau(p, s, n)^2 = \frac{\sqrt{p\nu_{\mathcal{H}}}}{n}.
    \end{equation*}
    If \(\eta \in (0, 1)\), then there exists \(C_\eta > 0\) depending only on \(\eta\) such that for all \(C > C_\eta\), we have
    \begin{equation*}
        P_0\left\{ T > p \nu_{\mathcal{H}} + \frac{2}{\sqrt{\eta}} \sqrt{p \nu_{\mathcal{H}}}\right\} + \sup_{\substack{f \in \mathcal{F}_s, \\ ||f||_2 \geq C \tau(p, s, n)}}P_f\left\{ T \leq p \nu_{\mathcal{H}} + \frac{2}{\sqrt{\eta}} \sqrt{p \nu_{\mathcal{H}}} \right\} \leq \eta.
    \end{equation*} 
\end{proposition}

As in the sparse case, (\ref{rate:main_expanded}) Lemma \ref{lemma:fixed_point_equiv} asserts \(\tau(p, s, n)^2 \asymp s\Gamma_{\mathcal{H}}\) provided that the condition \(\log\left(1 + \frac{p}{s^2}\right) \leq n/2\) is satisfied. 

\subsection{Special cases}\label{section:special_cases}

Having formally stated lower and upper bounds for the minimax separation rate, it is informative to explore a number of special cases and witness a variety of phenomena. Throughout the illustrations it will be assumed \(\log\left(1 + \frac{p}{s^2}\right) \leq n/2\) as discussed earlier. 

\subsubsection{Sobolev}\label{section:special_cases_sobolev}
Taking the case \(\mu_k \asymp k^{-2\alpha}\) as emblematic of Sobolev space with smoothness \(\alpha > 0\), we obtain the minimax separation rate
\begin{equation*}
    \varepsilon_{\mathcal{H}}^*(p, s, n)^2 \asymp
    \begin{cases}
        \frac{s \log\left(\frac{p}{s^2}\right)}{n} + s \left( \frac{n}{\sqrt{\log\left(\frac{p}{s^2}\right)}} \right)^{-\frac{4\alpha}{4\alpha+1}} &\textit{if } s < \sqrt{p}, \\
        s \left(\frac{n s}{\sqrt{p}}\right)^{-\frac{4\alpha}{4\alpha+1}} &\textit{if } s \geq \sqrt{p}. 
    \end{cases}
\end{equation*}
It is useful to compare to the rates obtained by Ingster and Lepski (Theorem 2 in \cite{ingster_multichannel_2003}) (see also \cite{gayraud_detection_2012}) although their choice of parameter is space is slightly different from that considered in this paper. In the sparse case \(s = O\left(p^{1/2-\delta}\right)\) for a constant \(\delta \in (0, 1/2]\), their rate is 
\begin{equation*}
    \begin{cases}
        s \left(\frac{n}{\sqrt{\log p}}\right)^{-\frac{4\alpha}{4\alpha+1}} &\textit{if } \log p \leq n^{\frac{1}{2\alpha+1}}, \\
        s \frac{\log p}{n} &\textit{if } \log p > n^{\frac{1}{2\alpha+1}}.
    \end{cases}
\end{equation*}
In the dense regime \(p = O(s^2)\), their rate (Theorem 1 in \cite{ingster_multichannel_2003}) is \(s \left(\frac{ns}{\sqrt{p}}\right)^{-\frac{4\alpha}{4\alpha+1}}\). Quick algebra verifies that our rate indeed matches Ingster and Lepski's rate in these sparsity regimes. 

In the sparse regime, the strange looking phase transition at \(\log p \asymp n^{\frac{1}{2\alpha+1}}\) in their rate now has a coherent explanation in view of our result. The situation \(\log p \lesssim n^{\frac{1}{2\alpha+1}}\) corresponds to the ``bulk'' regime in which case \(T_r(d)\) from (\ref{def:T_r}) behaves like a sum of thresholded subgaussian random variables. On the other side where \(\log p \gtrsim n^{\frac{1}{2\alpha+1}}\), subexponential behavior is exhibited by \(T_r(d)\). In fact, our result gives more detail beyond \(s \lesssim p^{1/2 - \delta}\). Assume only \(s < \sqrt{p}\) (for example, \(s = \frac{\sqrt{p}}{\log p}\) is allowed now). Then the phase transition between the ``bulk'' and ``tail'' regimes actually occurs at \(\log\left(1 + \frac{p}{s^2}\right) \asymp n^{\frac{1}{2\alpha+1}}\).

\subsubsection{Finite dimension}
Consider a finite dimensional situation, that is \(1 = \mu_1 = \mu_2 = ... = \mu_m > \mu_{m+1} = \mu_{m+2} = ... = 0\) for some positive integer \(m\). Function spaces exhibiting this kind of structure include linear functions, finite polynomials, and generally RKHSs based on kernels with finite rank. If \(m < \frac{n^2}{\log\left(1+\frac{p}{s^2}\right)}\), the minimax separation rate is 
\begin{equation*}
    \varepsilon_{\mathcal{H}}^*(p, s, n)^2 \asymp 
    \begin{cases}
        \frac{s\log\left(\frac{p}{s^2}\right)}{n} + \frac{s\sqrt{m \log\left(\frac{p}{s^2}\right)}}{n} &\textit{if } s < \sqrt{p}, \\
        \frac{\sqrt{p m}}{n} &\textit{if } s \geq \sqrt{p}. 
    \end{cases}
\end{equation*}
In the sparse regime, the phase transition between the bulk and tail regimes occurs at \(\log\left(1 + \frac{p}{s^2}\right) \asymp m\).

\subsubsection{Exponential decay}
As another example, consider exponential decay of the eigenvalues \(\mu_k = c_1 e^{-c_2 k^\gamma}\) where \(c_1,c_2\) are universal constants and \(\gamma > 0\). Such decay is a feature of RKHSs based on Gaussian kernels. The minimax separation rate is 
\begin{equation*}
    \varepsilon_{\mathcal{H}}^*(p, s, n)^2 \asymp 
    \begin{cases}
        \frac{s\log\left(\frac{p}{s^2}\right)}{n} + s \log^{\frac{1}{2\gamma}}\left(\frac{n}{\sqrt{\log\left(\frac{p}{s^2}\right)}}\right) \cdot \frac{\sqrt{\log\left(\frac{p}{s^2}\right)}}{n} &\textit{if } s < \sqrt{p} \\
        \frac{\sqrt{p}}{n} \log^{\frac{1}{2\gamma}}\left(\frac{ns}{\sqrt{p}}\right) &\textit{if } s \geq \sqrt{p}.
    \end{cases}
\end{equation*}
In the sparse regime, the phase transition between the bulk and tail regimes occurs at \(\log^{\gamma}\left(1 + \frac{p}{s^2}\right) \asymp \log n\). The minimax separation rate is quite close to the finite dimensional rate, which is sensible as RKHSs based on Gaussian kernels are known to be fairly ``small'' nonparametric function spaces \cite{wainwright_high-dimensional_2019}.

\section{Adaptation}

Thus far the sparsity parameter \(s\) has been assumed known, and the tests constructed make use of this information. In practice, the statistician is typically ignorant of the sparsity level and so it is of interest to understand whether adaptive tests can be furnished. In this section, we will establish an adaptive testing rate which accommodates a generic \(\mathcal{H}\), and it turns out to exhibit a cost for adaptation which depends delicately on the function space.  

To the best of our knowledge, Spokoiny's article was the first to demonstrate an unavoidable cost for adaptation in a signal detection problem \cite{spokoiny_adaptive_1996}. Later work established unavoidable costs for adaptive testing across a variety of situations (see \cite{ingster_nonparametric_2003} and references therein). This early work largely focused on adapting to an unknown smoothness parameter in a univariate nonparametric regression setting. More recently adaptation to sparsity in high dimensional models has been studied. In many problems, one can adapt to sparsity without cost in the rate (nor in the constant for some problems) \cite{donoho_higher_2004,liu_minimax_2021,kotekal_minimax_2021,ingster_detection_2010,arias-castro_global_2011,hall_innovated_2010}. In the context of sparse additive models in Sobolev space, Ingster and Lepski \cite{ingster_multichannel_2003} (see also \cite{gayraud_detection_2012}) consider adaptation, and we discuss their results in Section \ref{section:special_cases_adapt}. 

\subsection{Preliminaries}\label{section:preliminaries_adapt}
To formally state our result, some slight generalizations of the concepts found in Section \ref{section:preliminaries} are needed.

\begin{definition}
    For \(a > 0\), define \(\nu_{\mathcal{H}}(s, a)\) to be the smallest positive integer \(\nu\) satisfying 
    \begin{equation*}
        \mu_{\nu} \leq \frac{\sqrt{\nu\log\left(1 + \frac{pa}{s^2}\right)}}{n}.
    \end{equation*}
\end{definition} 

\begin{definition}
    For \(a > 0\), define \(\Gamma_{\mathcal{H}}(s, a)\) to be 
    \begin{equation*}
        \Gamma_{\mathcal{H}}(s, a) := \max_{\nu \in \mathbb{N}} \left\{ \mu_{\nu} \wedge \frac{\sqrt{\nu \log\left(1 + \frac{pa}{s^2}\right)}}{n} \right\}.
    \end{equation*}
\end{definition}

Note \(\nu_{\mathcal{H}}(s, a)\) is decreasing in \(a\) and \(\Gamma_{\mathcal{H}}(s, a)\) is increasing in \(a\). As discussed in Section \ref{section:preliminaries}, two different forms are useful when discussing the separation rate, and Lemma \ref{lemma:fixed_point_equiv} facilitated that discussion. The following lemma is a slight generalization and has the same purpose.  

\begin{lemma}\label{lemma:Gamma_nu_equiv_adapt}
    Suppose \(a > 0\). If \(\log\left(1 + \frac{pa}{s^2}\right) \leq \frac{n}{2}\), then 
    \begin{equation*}
        \Gamma_{\mathcal{H}}(s, a) \leq \frac{\sqrt{\nu_{\mathcal{H}}(s, a) \log\left(1 + \frac{pa}{s^2}\right)}}{n} \leq \sqrt{2} \Gamma_{\mathcal{H}}(s, a).
    \end{equation*}
\end{lemma}

At a high level, the central issue with adaptation lies in the selection of an estimator's or a test's hyperparameter (such as the bandwidth or penalty). Typically, there is some optimal choice but it requires knowledge of an unknown parameter (e.g. smoothness or sparsity) in order to pick it. In the current problem, the optimal choice of \(\nu\) is unknown since the sparsity level \(s\) is unknown. 

In adaptive testing, the typical strategy is to fix a grid of different values, construct a test for each potential value of the hyperparameter, and for detection use the test which takes the maximum over this collection of tests. Typically, a geometric grid is chosen, and the logarithm of the grid's cardinality reflects a cost paid by the testing procedure \cite{spokoiny_adaptive_1996,liu_minimax_2021,ingster_multichannel_2003}. It turns out this high level intuition holds for signal detection in sparse additive models, but the details of how to select the grid are not direct since a generic \(\mathcal{H}\) must be accommodated. For \(a > 0\), define 
\begin{equation*}
    \mathscr{V}_a := \left\{ 2^k : k \in \mathbb{N} \cup \{0\} \text{ and } 2^{k-1} < \nu_{\mathcal{H}}(s, a) \leq 2^k \text{ for some } s \in [p]\right\}. 
\end{equation*}
Define 
\begin{equation}\label{def:script_A}
    \mathscr{A}_{\mathcal{H}} := \sup\left\{a \geq 1 : \log(e|\mathscr{V}_a|) \geq a \right\}
\end{equation}
and define 
\begin{equation}\label{def:script_V}
    \mathscr{V}_{\mathcal{H}} := \mathscr{V}_{\mathscr{A}_{\mathcal{H}}}.
\end{equation}
The grid \(\mathscr{V}_{\mathcal{H}}\) will be used. It is readily seen that \(\mathscr{A}_{\mathcal{H}}\) is finite as the crude bound \(\mathscr{A}_{\mathcal{H}} \leq \log(ep)\) is immediate. The following lemma shows that \(\mathscr{A}_{\mathcal{H}}\) is, in essence, a fixed point. 

\begin{lemma}\label{lemma:script_AV_equiv}
    If \(\log\left(1 + p\mathscr{A}_{\mathcal{H}}\right) \leq \frac{n}{2}\), then \(\mathscr{A}_{\mathcal{H}} \leq \log(e|\mathscr{V}_{\mathcal{H}}|) \leq 2\mathscr{A}_{\mathcal{H}}\). 
\end{lemma}  

It turns out an adaptive testing procedure can be constructed which pays a price determined by \(\mathscr{A}_{\mathcal{H}}\) (equivalently \(\log(e|\mathscr{V}_{\mathcal{H}}|)\) up to constant factors). To elaborate, assume \(\log\left(1 + p\mathscr{A}_{\mathcal{H}}\right) \leq n/2\). We will construct an adaptive test which achieves 
\begin{equation*}
    \begin{cases}
        \frac{s}{n}\log\left(1 + \frac{p\mathscr{A}_{\mathcal{H}}}{s^2}\right) + s \Gamma_{\mathcal{H}}(s, \mathscr{A}_{\mathcal{H}}) &\textit{if } s < \sqrt{p\mathscr{A}_{\mathcal{H}}}, \\
        s \Gamma_{\mathcal{H}}(s, \mathscr{A}_{\mathcal{H}}) &\textit{if } s \geq \sqrt{p\mathscr{A}_{\mathcal{H}}}. 
    \end{cases}
\end{equation*}
The form of this separation rate is very similar to the minimax rate, except that the phase transition has shifted to \(s \asymp \sqrt{p\mathscr{A}_{\mathcal{H}}}\) and a cost involving \(\mathscr{A}_{\mathcal{H}}\) is incurred. The value of \(\mathscr{A}_{\mathcal{H}}\) can vary quite a bit with the choice of \(\mathcal{H}\), and a few examples are illustrated in Section \ref{section:special_cases_adapt}. 

Our choice of the geometric grid yielding the fixed point characterization of Lemma \ref{lemma:script_AV_equiv} may not be immediately intuitive. It can be understood at a high level by noticing there are two competing factors at play. First, fix \(a \geq 1\) and note it can be conceptualized to represent a target cost from scanning over some grid. To achieve that target cost, we should use the truncation levels \(\nu_{\mathcal{H}}(s, a)\) for each sparsity \(s\). Now, consider that we will scan over the geometric grid \(\mathscr{V}_{a}\) obtained from the truncation levels. If that geometric grid has cardinality which would incur a cost less than \(a\), then it would appear strange that we are scanning over a smaller grid to achieve a target cost associated to a larger grid. It is intuitive that we might try to aim for a lower target cost \(a' < a\) instead. But this changes our grid to \(\mathscr{V}_{a'}\) which now has a different cardinality, and so the same logic applies to \(a'\). A similar line of reasoning applies if the grid had happened to be too large. Therefore, we intuitively want to select \(a\) such that \(\mathscr{V}_a\) has the proper cardinality to push the cost \(a\) upon us. Hence, we are seeking a kind of fixed point in this sense.  

\subsection{Lower bound}\label{section:adapt_lower_bound}
The formulation of a lower bound is more delicate than one might initially anticipate. To illustrate the subtlety, suppose we have a candidate adaptive rate \(s \mapsto \epsilon(s)\). It is tempting to consider the following lower bound formulation; for any \(\eta \in (0, 1)\), there exists \(c_\eta > 0\) such that for any \(0 < c < c_\eta\),
\begin{equation*}
    \inf_{\varphi} \left\{P_0\left\{\varphi \neq 0\right\} + \max_{1 \leq s \leq p} \sup_{\substack{f \in \mathcal{F}_s, \\ ||f||_2 \geq c\epsilon(s)}} P_f\left\{\varphi \neq 1\right\} \right\} \geq 1-\eta. 
\end{equation*}
While extremely natural, this criterion has an issue. In particular, if there exists \(\bar{s}\) such that \(\epsilon(\bar{s}) \asymp \varepsilon^*(\bar{s})\) where \(s \mapsto \varepsilon^*(s)\) denotes the minimax rate, then the above criterion is trivially satisfied. One simply lower bounds the maximum over all \(s\) with the choice \(s = \bar{s}\) and then invokes the minimax lower bound. To exaggerate, the candidate rate \(\epsilon(1) = \varepsilon^*(1)\) and \(\epsilon(s) = \infty\) for \(s \geq 2\) satisfies the above criterion. Furthermore, from an upper bound perspective there is a trivial test which achieves this candidate rate. The absurdity from this lower bound criterion would then force us to conclude this candidate rate gives an adaptive testing rate. 

To avoid such absurdities, we use a different lower bound criterion. The key issue in the formulation is that, in the domain of maximization for the Type II error, one must not include any \(s\) for which the candidate rate is of the same order as the minimax rate. 

\begin{definition}\label{def:adapt_lower_bound}
    Suppose \(s \mapsto \varepsilon_{\text{adapt}}(p, s, n)\) is a candidate adaptive rate and \(\left\{\mathcal{S}_{p,n}\right\}_{p,n \in \mathbb{N}}\) are a collection of sets satisfying \(\mathcal{S}_{p,n} \subset [p]\) for all \(p\) and \(n\). We say \(s \mapsto \varepsilon_{\text{adapt}}(p, s, n)\) satisfies the adaptive lower bound criterion with respect to \(\left\{\mathcal{S}_{p,n}\right\}_{p,n \in \mathbb{N}}\) if the following two conditions hold. 
    \begin{enumerate}[label=(\roman*), itemsep=12pt]
        \item For any \(\eta \in (0, 1)\), there exists \(c_\eta > 0\) depending only on \(\eta\) such that for all \(0 < c < c_\eta\), 
        \begin{equation*}
            \inf_{\varphi}\left\{P_0\left\{\varphi \neq 0\right\} + \max_{s \in \mathcal{S}_{p,n}} \sup_{\substack{f \in \mathcal{F}_s, \\ ||f||_2 \geq c\varepsilon_{\text{adapt}}(s)}} P_f\left\{\varphi \neq 1\right\} \right\} \geq 1-\eta
        \end{equation*}
        for all \(p\) and \(n\). 
        \item Either 
        \begin{equation}
            \max_{p, n \in \mathbb{N}} \min_{s \in \mathcal{S}_{p,n}} \frac{\varepsilon_{\text{adapt}}(p, s, n)}{\varepsilon^*(p, s, n)} = \infty \label{eqn:nontrivial_candidate}
        \end{equation}
        or 
        \begin{equation}
            \max_{p, n \in \mathbb{N}} \max_{s \in [p]} \frac{\varepsilon_{\text{adapt}}(p, s, n)}{\varepsilon^*(p, s, n)} < \infty \label{eqn:trivial_candidate}
        \end{equation}
        where \(s \mapsto \varepsilon^*(p, s, n)\) denotes the minimax separation rate. 
    \end{enumerate}
\end{definition}

Intuitively, this criterion is nontrivial only when there are no sparsity levels in the chosen reference sets for which the candidate rate is the same order as the minimax rate. More explicitly, there is no sequence \(\{s_{p,n}\}_{p,n \in \mathbb{N}}\) with \(s_{p,n} \in \mathcal{S}_{p,n}\) along which the candidate and minimax rates match (up to constants). This criterion avoids trivialities such as the one described earlier. In Definition \ref{def:adapt_lower_bound}, formalizing the notion of ``order'' requires speaking in terms of sequences, though it may appear unfamiliar and clunky at first glance. Definition \ref{def:adapt_lower_bound} is not new, but rather a direct port to the testing context of the notion of an \emph{adaptive rate of convergence on the scale of classes} from \cite{tsybakov_pointwise_1998} in the estimation setting (see also \cite{collier_optimal_2018} for an application of this notion to linear functional estimation in the sparse Gaussian sequence model).

The condition (\ref{eqn:trivial_candidate}) simply requires the candidate rate to match the minimax rate if (\ref{eqn:nontrivial_candidate}) does not hold. In this way, the definition allows for the possibility of the minimax rate to be an adaptive rate (in those cases where no cost for adaptation needs to be paid).

We now state a lower bound with respect to this criterion. Define the candidate rate
\begin{equation}\label{rate:psi_adapt}
    \psi_{\text{adapt}}(p, s, n)^2 := 
    \begin{cases}
        \frac{s}{n}\log\left(1 + \frac{p\mathscr{A}_{\mathcal{H}}}{s^2}\right) \vee s \Gamma_{\mathcal{H}}(s, \mathscr{A}_{\mathcal{H}}) &\textit{if } s < \sqrt{p \mathscr{A}_{\mathcal{H}}}, \\
        s \Gamma_{\mathcal{H}}(s, \mathscr{A}_{\mathcal{H}}) &\textit{if } s \geq \sqrt{p \mathscr{A}_{\mathcal{H}}}. 
    \end{cases}
\end{equation}
Here, \(\mathscr{A}_{\mathcal{H}}\) is defined in (\ref{def:script_A}). Examining this candidate, it is possible one may run into absurdities for \(s < (p\mathscr{A}_{\mathcal{H}})^{1/2-\delta}\) since \(\log\left(1 + \frac{p\mathscr{A}_{\mathcal{H}}}{s^2}\right) \asymp \log p\), meaning the candidate rate can match the minimax rate. Note here we have used that \(\mathscr{A}_{\mathcal{H}}\) grows at most logarithmically in \(p\). Therefore, we will take \(\mathcal{S} = \left\{ s \in [p] : s \geq \sqrt{p\mathscr{A}_{\mathcal{H}}}\right\}\) (dropping the subscripts and writing \(\mathcal{S} = \mathcal{S}_{p,n}\) for notational ease). 

It follows from the fact that \(\Gamma_{\mathcal{H}}(s, a)\) is increasing in \(a\) that either (\ref{eqn:nontrivial_candidate}) or (\ref{eqn:trivial_candidate}) should typically hold. To see that \(\Gamma_{\mathcal{H}}(s, a)\) is indeed increasing in \(a\), let us fix \(a \leq a'\). Consider that for every \(\nu \in \mathbb{N}\) we have 
\begin{equation*}
    \mu_{\nu} \wedge \frac{\sqrt{\nu \log\left(1 + \frac{pa}{s^2}\right)}}{n} \leq \mu_{\nu} \wedge \frac{\sqrt{\nu \log\left(1 + \frac{pa'}{s^2}\right)}}{n}.   
\end{equation*}
Taking maximum over \(\nu\) on both sides yields \(\Gamma_{\mathcal{H}}(s, a) \leq \Gamma_{\mathcal{H}}(s, a')\), i.e. \(\Gamma_{\mathcal{H}}\) is increasing in \(a\). Now, consider \(a = 1\) and \(a' = \mathscr{A}_{\mathcal{H}}\) in order to compare the candidate rate \(\psi_{\text{adapt}}\) to the minimax rate \(\varepsilon^*\). Since \(s \geq \sqrt{p\mathscr{A}_{\mathcal{H}}}\), we have \(\log\left(1 + \frac{p}{s^2}\right) \asymp \frac{p}{s^2}\) and \(\log\left(1 + \frac{p\mathscr{A}_{\mathcal{H}}}{s^2}\right) \asymp \frac{p\mathscr{A}_{\mathcal{H}}}{s^2}\). With the above display, it is intuitive that either (\ref{eqn:nontrivial_candidate}) or (\ref{eqn:trivial_candidate}) should hold in typical nonpathological cases; these conditions can be easily checked on a problem by problem basis (such as those in Section \ref{section:special_cases_adapt}).

In preparation for the statement of the lower bound, the following set is needed 
\begin{equation}
    \widetilde{\mathscr{V}}_{\mathcal{H}} := \left\{2^k : k \in \mathbb{N} \cup \left\{0\right\} \text{ and } 2^{k-1} < \nu_{\mathcal{H}}(s, \mathscr{A}_{\mathcal{H}}) \leq 2^k \text{ for some } s \in \mathcal{S} \right\}. \label{def:tilde_script_V}
\end{equation}
With \(\widetilde{\mathscr{V}}_{\mathcal{H}}\) defined, we are ready to state the lower bound. The following result establishes condition (i) of Definition \ref{def:adapt_lower_bound} for \(\psi_{\text{adapt}}\) given by (\ref{rate:psi_adapt}) with respect to our choice \(\mathcal{S} = \left\{s \in [p] : s \geq \sqrt{p\mathscr{A}_{\mathcal{H}}}\right\}\).

\begin{theorem}\label{thm:adapt_lbound}
    Suppose \(\mathscr{A}_{\mathcal{H}} \asymp \log\left(e|\widetilde{\mathscr{V}}_{\mathcal{H}}|\right)\). Assume further \(\log\left(1 + p \mathscr{A}_{\mathcal{H}}\right) \leq \frac{n}{2}\). If \(\eta \in \left(0 , 1\right)\), then there exists \(c_\eta > 0\) depending only on \(\eta\) such that for all \(0 < c < c_\eta\), we have 
    \begin{align*}
        \inf_{\varphi} \left\{ P_0\left\{ \varphi \neq 0 \right\} + \max_{s \in \mathcal{S}} \sup_{\substack{f \in \mathcal{F}_s, \\ ||f||_2 \geq c \psi_{\text{adapt}}}} P_f\left\{\varphi \neq 1\right\} \right\} \geq 1-\eta
    \end{align*} 
    where \(\psi_{\text{adapt}}\) is given by (\ref{rate:psi_adapt}) and \(\mathcal{S} = \left\{s \in [p] : s \geq \sqrt{p\mathscr{A}_{\mathcal{H}}}\right\}\). 
\end{theorem}

The condition \(\mathscr{A}_{\mathcal{H}} \asymp \log\left(e|\widetilde{\mathscr{V}}_{\mathcal{H}}|\right)\) is somewhat mild; for example it is satisfied when the eigenvalues satisfy a polynomial decay such as \(\mu_k = c_\alpha k^{-2\alpha}\) for some \(\alpha > 0\) and positive universal constant \(c_\alpha\). It is also satisfied when there is exponential decay such as \(\mu_k = c_1 e^{-c_2 j^\gamma}\) for some \(\gamma > 0\) and positive universal constants \(c_1\) and \(c_2\). Sections \ref{section:special_cases_adapt} contains further discussion of these two cases. 

The proof strategy is, at a high level, the same as that in Section \ref{section:lower_bound} with the added complication that the prior should also randomize over the sparsity level in order to capture the added difficulty of not knowing the sparsity level. The random sparsity level is drawn by drawing a random truncation point \(\nu\) and extracting the induced sparsity. The rest of the prior construction is similar to that in Section \ref{section:lower_bound}, but the analysis is complicated by the random sparsity level. 

Finally, note Theorem \ref{thm:adapt_lbound} is a statement about a cost paid for adaptation for sparsity levels in \(\mathcal{S} = \left\{s \in [p] : s \geq \sqrt{p\mathscr{A}_{\mathcal{H}}} \right\}\). Nothing is asserted about smaller sparsities. It would be interesting to show whether or not the candidate \(\psi_{\text{adapt}}\) also captures the cost for sparsity levels \(s\) which satisfy \((p\mathscr{A}_{\mathcal{H}})^{1/2-\delta} < s < \sqrt{p\mathscr{A}_{\mathcal{H}}} \) for all \(\delta \in (0, 1/2]\) (e.g. \(s \asymp \sqrt{p\mathscr{A}_{\mathcal{H}}}/\log(p\mathscr{A}_{\mathcal{H}})\)); we leave it open for future work.

\subsection{Upper bound}\label{section:adapt_upper_bound}

In this section, an adaptive test is constructed. Recall \(\mathscr{A}_{\mathcal{H}}\) and \(\mathscr{V}_{\mathcal{H}}\) given in (\ref{def:script_A}) and (\ref{def:script_V}) respectively. Define 
\begin{equation}\label{rate:tau_adapt_ubound}
    \tau_{\text{adapt}}^2(p, s, n) =
    \begin{cases}
        \left(\frac{s}{n}\log\left(1 + \frac{p \mathscr{A}_{\mathcal{H}}}{s^2}\right) \right) \vee \left( \frac{s}{n} \sqrt{\nu_{\mathcal{H}}(s, \mathscr{A}_{\mathcal{H}}) \log\left(1 + \frac{p \mathscr{A}_{\mathcal{H}}}{s^2}\right)}\right) &\textit{if } s < \sqrt{p \mathscr{A}_{\mathcal{H}}}, \\
        \frac{s}{n}\sqrt{\nu_{\mathcal{H}}(s, \mathscr{A}_{\mathcal{H}}) \log\left(1 + \frac{p\mathscr{A}_{\mathcal{H}}}{s^2}\right)} &\textit{if } s \geq \sqrt{p \mathscr{A}_{\mathcal{H}}}. 
    \end{cases}
\end{equation}

In the test, the following collection of geometric grids is used. Define
\begin{equation}\label{def:script_S}
    \mathscr{S} := \left\{1, 2, 4, ... , 2^{\left\lceil\log_2 \sqrt{p \mathscr{A}_{\mathcal{H}}}\right\rceil-1}\right\} \cup \{p\}.
\end{equation}
The choice to take \(\mathscr{S}\) as a geometric grid is not statistically critical, but doing so is computationally advantageous. 

\begin{theorem}\label{thm:adapt_ubound}
    There exist universal positive constants \(D, K_2, K_2',\) and \(K_3\) such that the following holds. Let \(\mathscr{A}_{\mathcal{H}}\), \(\mathscr{V}_{\mathcal{H}}\), and \(\mathscr{S}\) be given by (\ref{def:script_A}), (\ref{def:script_V}), and (\ref{def:script_S}) respectively. For \(\nu \in \mathscr{V}_{\mathcal{H}}\) and \(s \in [p]\), set \(d_\nu = \nu \vee \lceil D \rceil\) and set
    \begin{align*}
        r_{\nu, s} &:= K_2\left(d_\nu \log\left(1 + \frac{p \mathscr{A}_{\mathcal{H}}}{s^2}\right)\right)^{1/4}, \\
        r_s' &:= K_2'\sqrt{\log\left(1 + \frac{p \mathscr{A}_{\mathcal{H}}}{s^2}\right)}. 
    \end{align*}
    If \(\eta \in (0, 1)\), then there exists \(C_\eta > 0\) depending only on \(\eta\) such that for all \(C > C_\eta\), we have 
    \begin{equation*}
        P_0\left\{ \max_{\nu \in \mathscr{V}_{\mathcal{H}}} \max_{s \in \mathscr{S}} \varphi_{\nu, s} \neq 0 \right\} + \max_{1 \leq s^* \leq p} \sup_{\substack{f \in \mathcal{F}_{s^*}, \\ ||f||_2 \geq C \tau_{\text{adapt}}(p, s^*, n)}} P_f\left\{ \max_{\nu \in \mathscr{V}_{\mathcal{H}}} \max_{s \in \mathscr{S}} \varphi_{\nu, s} \neq 1\right\} \leq \eta
    \end{equation*}
    where \(\tau_{\text{adapt}}\) is given by (\ref{rate:tau_adapt_ubound}). Here, \(\varphi_{\nu, s}\) is given by 
    \begin{equation*}
        \varphi_{\nu, s} := 
        \begin{cases}
            \mathbbm{1}_{\left\{T_{r_{\nu, s}}(d_\nu) > C s\sqrt{\nu \log\left(1 + \frac{p\mathscr{A}_{\mathcal{H}}}{s^2}\right)}\right\}} &\textit{if } s < \sqrt{p \mathscr{A}_{\mathcal{H}}} \textit{ and } \sqrt{\log\left(1 + \frac{p\mathscr{A}_{\mathcal{H}}}{s^2}\right)} \leq K_3 \sqrt{d_\nu}, \\
            \mathbbm{1}_{\left\{T_{r_s'}(d_\nu) > C s \log\left(1 + \frac{p\mathscr{A}_{\mathcal{H}}}{s^2}\right)\right\}} &\textit{if } s < \sqrt{p \mathscr{A}_{\mathcal{H}}} \textit{ and } \sqrt{\log\left(1 + \frac{p\mathscr{A}_{\mathcal{H}}}{s^2}\right)} > K_3 \sqrt{d_\nu}, \\
            \mathbbm{1}_{\left\{ n\sum_{j=1}^{p}\sum_{k \leq \nu} X_{k,j}^2 > \nu p + C\sqrt{\nu p \mathscr{A}_{\mathcal{H}}}\right\}} &\textit{if } s \geq \sqrt{p\mathscr{A}_{\mathcal{H}}}
        \end{cases}
    \end{equation*}
    and the statistic \(T_r(d)\) is given by (\ref{def:T_r}).
\end{theorem}

As mentioned earlier, the adaptive test involves scanning over the potential values of \(\nu\) in the geometric grid \(\mathscr{V}_{\mathcal{H}}\). Consequently, a cost involving \(\mathscr{A}_{\mathcal{H}}\) is paid. Note for \(s < \sqrt{p\mathscr{A}_{\mathcal{H}}}\) we also need to scan over \(s \in \mathscr{S}\) in order to set the thresholds in the statistics \(T_r(d)\) properly. It turns out this extra scan does not incur a cost; cost is driven by needing to scan over \(\nu\). 

To understand how scanning over \(\nu\) can incur a cost, it is most intuitive to consider the need to control the Type I error when scanning over the \(\chi^2\) statistics with various degrees of freedom. Roughly speaking, it is necessary to pick the threshold values \(t_\nu\) such that the Type I error \(P\left(\bigcup_{\nu \in \mathscr{V}_{\mathcal{H}}} \left\{ \chi^2_{\nu p} - \nu p > t_\nu \right\}\right)\) is smaller than some prescribed error level. By union bound and (\ref{eqn:laurent_massart_tail}), consider the tail bound \( P\left(\bigcup_{\nu \in \mathscr{V}_{\mathcal{H}}} \left\{ \chi^2_{\nu p} - \nu p > 2\sqrt{\nu p u} + 2u \right\}\right) \leq \sum_{\nu \in \mathscr{V}_{\mathcal{H}}} P\left\{ \chi^2_{\nu p} - \nu p > 2\sqrt{\nu p u} + 2u \right\} \leq |\mathscr{V}_{\mathcal{H}}| e^{-u}\). To ensure the Type I error is small, this tail bound suggests the choice \(t_\nu \asymp \sqrt{\nu p \log |\mathscr{V}_{\mathcal{H}}|} + 2 \log|\mathscr{V}_{\mathcal{H}}| \asymp \sqrt{\nu p \log |\mathscr{V}_{\mathcal{H}}|}\) where the latter equivalence follows from the fact \(|\mathscr{V}_{\mathcal{H}}|\) grows at most logarithmically in \(p\). Since the threshold \(t_\nu\) is inflated, a signal must also have larger norm in order to be detected. The logarithmic inflation factor is a consequence of the \(\chi^2\) tail.

\subsection{Special cases}\label{section:special_cases_adapt}
To illustrate how a price for adaptation depends delicately on the choice of function space \(\mathcal{H}\), we consider a a variety of cases. Throughout, the assumption \(\log(1+p\mathscr{A}_{\mathcal{H}}) \leq n/2\) as well as the other assumptions necessary for our result are made.   

\subsubsection{Sobolev}\label{section:special_cases_adapt_sobolev}
Taking \(\mu_k \asymp k^{-2\alpha}\) as emblematic of Sobolev space with smoothness \(\alpha\), it can be shown that \(\mathscr{A}_{\mathcal{H}} \asymp \log(e|\widetilde{\mathscr{V}}_{\mathcal{H}}|) \asymp \log \log p\) (recall (\ref{def:script_A}) and (\ref{def:tilde_script_V})). Our upper and lower bounds assert an adaptive testing rate is given by
\begin{equation*}
    \varepsilon_{\text{adapt}}(p, s, n)^2 \asymp
    \begin{cases}
        \frac{s \log\left(\frac{p\log\log p}{s^2}\right)}{n} + s \left( \frac{n}{\sqrt{\log\left(\frac{p\log\log p}{s^2}\right)}}\right)^{-\frac{4\alpha}{4\alpha+1}} &\textit{if } s < \sqrt{p\log\log p}, \\
        s \left(\frac{ns}{\sqrt{p\log\log p}}\right)^{-\frac{4\alpha}{4\alpha+1}} &\textit{if } s \geq \sqrt{p\log\log p}. 
    \end{cases}
\end{equation*}
Ingster and Lepski \cite{ingster_multichannel_2003} consider adaptation in the sparse regime \(s = O(p^{1/2-\delta})\) and the dense regime \(s = p^{1/2+\delta}\) for \(\delta \in (0, 1/2)\) separately. In the sparse regime, they construct an adaptive test which is able to achieve the minimax rate; no cost is paid. In the dense rate, not only do they give a test which achieves \(s \left(\frac{ns}{\sqrt{p\log\log p}}\right)^{-\frac{4\alpha}{4\alpha+1}}\), they also supply a lower bound (now, in our lower bound notation, \(\mathcal{S} = [p]\) can be taken) showing it cannot be improved. As seen in the above display, our test enjoys optimality in these regimes.

The reader should take care when comparing \cite{ingster_multichannel_2003} to our result. In our setting, the unknown sparsity level can vary throughout the entire range \(s \in [p]\). In contrast, \cite{ingster_multichannel_2003} consider two separate cases. Either the unknown sparsity \(s\) is constrained to \(s \leq p^{1/2-\delta}\), or it is constrained to \(s \geq p^{1/2+\delta}\). To elaborate, the precise value of \(s\) is unknown but \cite{ingster_multichannel_2003} assumes the statistician has knowledge of which of the two cases is in force. In contrast, we make no such assumption; nothing at all is known about the sparsity level. In our view, this situation is more faithful to the problem facing the practitioner. Despite constructed for a seemingly more difficult setting, it turns out our test is optimal under Ingster and Lepski's setting.

Our result provides finer detail around \(\sqrt{p}\) missed by \cite{ingster_multichannel_2003}. In Theorem 3 of their article, Ingster and Lepski propose an adaptive test procedure which is applicable in the regime \(s \asymp \sqrt{p}\). It requires signal strength of squared order at least  
\begin{equation*}
    \sqrt{p}\left(\frac{n}{\sqrt{\log \log p}}\right)^{-\frac{4\alpha}{4\alpha+1}}.
\end{equation*}
This is suboptimal since our test achieves  
\begin{equation*}
    \begin{cases}
        \sqrt{p}\left(\frac{n}{\sqrt{\log\log\log p}}\right)^{-\frac{4\alpha}{4\alpha+1}} &\textit{if } \log\log\log p \leq n^{\frac{1}{2\alpha+1}}, \\
        \frac{\sqrt{p} \log\log\log p}{n} &\textit{if } \log\log\log p > n^{\frac{1}{2\alpha+1}},
    \end{cases}
\end{equation*}
which is faster.

\subsubsection{Finite dimension}
Consider a finite dimensional structure \(1 = \mu_1 = \mu_2 = ... = \mu_m > \mu_{m+1} = ... = 0\) where \(m\) is a positive integer. If \(m < \frac{n^2}{\log\left(1 + p\right)}\), it can be shown \(\mathscr{A}_{\mathcal{H}} \asymp 1\), and so the minimax separation rate can be achieved by our adaptive test. In other words, no cost is paid for adaptation. 

\subsubsection{Exponential decay}
Consider exponential decay of the eigenvalues \(\mu_k = c_1 e^{-c_2 k^\gamma}\) where \(\gamma > 0\). It can be shown that \(\mathscr{A}_{\mathcal{H}} \asymp \log(e|\widetilde{\mathscr{V}}_{\mathcal{H}}|) \asymp \log\log\log p\), and so an adaptive rate is
\begin{equation*}
    \varepsilon_{\text{adapt}}(p, s, n)^2 \asymp
    \begin{cases}
        \frac{s \log\left(\frac{p\log\log\log p}{s^2}\right)}{n} + s \log^{\frac{1}{2\gamma}}\left(\frac{n}{\sqrt{\log\left(\frac{p\log\log\log p}{s^2}\right)}}\right) \cdot \frac{\sqrt{\log\left(\frac{p\log\log\log p}{s^2}\right)}}{n} &\textit{if } s < \sqrt{p\log\log\log p}, \\
        \frac{\sqrt{p\log\log\log p}}{n}\log^{\frac{1}{2\gamma}}\left(\frac{ns}{\sqrt{p\log\log\log p}}\right) &\textit{if } s \geq \sqrt{p\log\log\log p}. 
    \end{cases}
\end{equation*}
The cost for adaptation here grows very slowly, and is perhaps another indication of the relative ``small'' size of RKHSs based on Gaussian kernels (as noted in Section \ref{section:special_cases}).

\section{Adaptation to both sparsity and smoothness}

So far we have assumed that the space \(\mathcal{H}\) is known. However, it is likely that \(\mathcal{H}\) is one constituent out of a collection of spaces indexed by some hyperparameter, such as a smoothness level. Typically, the true value of this hyperparameter is unknown in addition to the sparsity being unknown, and it is of interest to understand how the testing rate changes. To avoid excessive abstractness, we follow \cite{ingster_multichannel_2003,gayraud_detection_2012} and adopt a working model where \(\mathscr{T}_s\) has \(\mu_k\) of the form \(\mu_k = k^{-2\alpha}\) emblematic of Sobolev space with smoothness \(\alpha\). To emphasize the dependence on \(\alpha\), let us write \(\mathscr{T}(s, \alpha)\) and its corresponding space \(\mathcal{F}(s, \alpha)\). Reiterating, we are interested in adapting to both the sparsity level \(s\) and the smoothness level \(\alpha\).

Ingster and Lepski \cite{ingster_multichannel_2003} study adaptation to both sparsity and smoothness. In particular, they study adaptation for an unknown \(\alpha \in [\alpha_0, \alpha_1]\) in a closed interval where the endpoints \(\alpha_0 < \alpha_1\) are known. As argued in \cite{ingster_multichannel_2003}, since \(\alpha_1\) is known, the sequence space basis in Section \ref{section:parameter} for any \(\alpha \in [\alpha_0, \alpha_1]\) can be taken to be \(\alpha_1\)-regular, and thus is known to the statistician. Furthermore, they separately study the dense regime \(s \geq p^{1/2+\delta}\) and the sparse regime \(s < p^{1/2-\delta}\) for some constant \(\delta \in (0, 1/2)\). We adopt the same setup as them.

Ingster and Lepski \cite{ingster_multichannel_2003} claim the following adaptive rate. In the dense case \(s \geq p^{1/2 + \delta}\), they make the assumption \(\frac{p}{s^2} \log\left(\frac{n^2s^2/p}{\log(n^2s^2/p)} \right) \to 0\) and state (Theorem 4 in \cite{ingster_multichannel_2003})
\begin{equation*}
    s \left(\frac{ns}{\sqrt{p\log\left(\frac{ns}{\sqrt{p}}\right)}}\right)^{-\frac{4\alpha}{4\alpha+1}}. 
\end{equation*}

\noindent In contrast\footnote{The same error affects the proofs of the lower bounds in Theorems 4 and 5 in \cite{ingster_multichannel_2003}. The prior they define is not supported on the parameter space. Under their prior (see (75) in \cite{ingster_multichannel_2003}), the various coordinate functions \(f_j\) end up having different smoothness levels \(\alpha_j\) instead of sharing a single \(\alpha\).}, we will show the adaptive rate in the dense case \(s \geq p^{1/2 + \delta}\) is
\begin{equation*}
    \varepsilon_{\text{dense}}^2(p, s, n, \alpha) \asymp s \left(\frac{ns}{\sqrt{p\log\log(np)}}\right)^{-\frac{4\alpha}{4\alpha+1}}. 
\end{equation*}
The careful reader will note that in the special case \(p = 1\), our answer recovers the adaptive separation rate \(\left(n/\sqrt{\log\log n}\right)^{-\frac{4\alpha}{4\alpha+1}}\) proved by Spokoiny \cite{spokoiny_adaptive_1996}. In the sparse case \(s < p^{1/2 - \delta}\), their claimed rate (Theorem 5 in \cite{ingster_multichannel_2003}) is
\begin{equation}\label{rate:IL_sparse}
    \frac{s \log p}{n} + s \left(\frac{n}{\sqrt{\log(p\log n)}}\right)^{-\frac{4\alpha}{4\alpha+1}}.
\end{equation}
As a partial correction, we prove the following lower bound
\begin{equation}\label{rate:revised_lowerbound}
    \frac{s \log(p \log\log n)}{n} + s \left(\frac{n}{\sqrt{\log(p \log\log n)}}\right)^{-\frac{4\alpha}{4\alpha+1}}. 
\end{equation}
In the regime \(\log p \gtrsim \log\log n\), (\ref{rate:IL_sparse}) and (\ref{rate:revised_lowerbound}) match each other and match the minimax rate (see Section \ref{section:special_cases_sobolev}). However, in the case \(\log p \lesssim \log\log n\), there may be a difference between (\ref{rate:IL_sparse}) and (\ref{rate:revised_lowerbound}). 

\subsection{Dense}
Fix \(\delta \in (0, 1/2)\) and \(\alpha_0 < \alpha_1\). Recall in the dense regime that \(s \geq p^{1/2+\delta}\). Define 
\begin{equation}\label{rate:sobolev_adapt_dense}
    \tau_{\text{dense}}^2(p, s, n, \alpha) := s \left(\frac{ns}{\sqrt{p \log\log(np)}}\right)^{-\frac{4\alpha}{4\alpha+1}}.
\end{equation}
As mentioned earlier, we will be establishing this rate in the dense regime. For use in a testing procedure, define the geometric grid 
\begin{equation}\label{def:V_dense}
    \mathcal{V}_{\text{test}} := \left\{1, 2, 4, 8, ... 2^{\left\lceil \log_2 \left(\left(\frac{np}{\sqrt{p\log\log(np)}}\right)^{\frac{2}{4\alpha_0+1}}\right)\right\rceil} \right\}.
\end{equation}
Note the statistician does not need to know \(s\) nor \(\alpha\) to construct \(\mathcal{V}_{\text{test}}\). Further note \(\log|\mathcal{V}_{\text{test}}| \asymp \log\log(np)\). 

\subsubsection{Upper bound}
    The test procedure employed follows the typical strategy in adaptive testing, namely constructing individual tests for potential values of \(\nu\) in the grid \(\mathcal{V}_{\text{test}}\), and then taking maximum over the individual tests to perform signal detection. The cost of adaptation involves the logarithm of the cardinality of \(\mathcal{V}_{\text{test}}\) (i.e. \(\log\log(np)\)). Since the dense regime is in force, the individual tests are \(\chi^2\) tests. 

    \begin{theorem}\label{thm:sobolev_adapt_ubound_dense}
        Fix \(\delta \in (0, 1/2)\) and \(\alpha_0 < \alpha_1\). If \(\eta \in (0, 1)\), then there exists \(C_\eta > 0\) depending only on \(\eta\) such that for all \(C > C_\eta\), we have 
        \begin{equation*}
            P_0\left\{\max_{\nu \in \mathcal{V}_{\text{test}}} \varphi_\nu = 1\right\} + \max_{s \geq p^{1/2+\delta}} \sup_{\alpha \in [\alpha_0, \alpha_1]} \sup_{\substack{f \in \mathcal{F}(s, \alpha), \\ ||f||_2 \geq C \tau_{\text{dense}}(p, s, n, \alpha)}} P_f\left\{ \max_{\nu \in \mathcal{V}_{\text{test}}} \varphi_\nu = 0 \right\} \leq \eta
        \end{equation*}
        where 
        \begin{equation*}
            \varphi_\nu = \mathbbm{1}_{\left\{n \sum_{j=1}^{p} \sum_{k \leq \nu} X_{k,j}^2 \geq \nu p + K_\eta \left(\sqrt{\nu p \log\log(np)} + \log\log(np)\right)\right\}}.
        \end{equation*}
        with \(K_\eta\) being a constant depending only on \(\eta\). Here \(\tau_{\text{dense}}\) is given by (\ref{rate:sobolev_adapt_dense}) and \(\mathcal{V}_{\text{test}}\) is given by (\ref{def:V_dense}).
    \end{theorem}
\subsubsection{Lower bound}
The following theorem states the lower bound. Note that it satisfies Definition \ref{def:adapt_lower_bound} with the straightforward modification to incorporate adaptation over \(\alpha\). In particular, the potential difficulty with absurdities outlined in Section \ref{section:adapt_lower_bound} do not arise since the candidate rate
\begin{equation*}
    \tau^2_{\text{dense}}(p, s, n, \alpha) = s \left(\frac{ns}{\sqrt{p \log\log(np)}}\right)^{-\frac{4\alpha}{4\alpha+1}}
\end{equation*}
is never of the same order as the minimax rate \(\varepsilon^*(p,s,n,\alpha)^2 = s \left(\frac{ns}{\sqrt{p}}\right)^{-\frac{4\alpha}{4\alpha+1}}\).

\begin{theorem}\label{thm:sobolev_adapt_lbound_dense}
    Fix \(\delta \in (0, 1/2)\) and \(\alpha_0 < \alpha_1\). If \(\eta \in (0, 1)\), then there exists \(c_\eta > 0\) depending only on \(\eta\) such that for all \(0 < c < c_\eta\), we have 
    \begin{equation*}
        \inf_{\varphi}\left\{ P_0\left\{\varphi \neq 0\right\} + \max_{s \geq p^{1/2+\delta}} \sup_{\alpha \in [\alpha_0, \alpha_1]} \sup_{\substack{f \in \mathcal{F}(s, \alpha), \\ ||f||_2 \geq c\tau_{\text{dense}}(p, s, n, \alpha)}} P_f\left\{\varphi \neq 1\right\} \right\} \geq 1-\eta
    \end{equation*}
    where \(\tau_{\text{dense}}\) is given by (\ref{rate:sobolev_adapt_dense}). 
\end{theorem}

\subsection{Sparse}\label{section:both_adapt_sparse}
Fix \(\delta \in (0, 1/2)\) and \(\alpha_0 < \alpha_1\). Recall that in the sparse regime we have \(s < p^{1/2-\delta}\). Define 
\begin{equation}\label{rate:sobolev_adapt_sparse}
    \tau_{\text{sparse}}^2(p, s, n, \alpha) := \frac{s\log\left(p \log\log n\right)}{n} + s \left(\frac{n}{\sqrt{\log(p\log\log n)}}\right)^{-\frac{4\alpha}{4\alpha+1}}.
\end{equation}

Ingster and Lepski \cite{ingster_multichannel_2003} give a testing procedure achieving (\ref{rate:IL_sparse}) and thus establish the upper bound. As noted earlier, when \(\log p \gtrsim \log\log n\), not only do (\ref{rate:sobolev_adapt_sparse}) and (\ref{rate:IL_sparse}) match each other but they also match the minimax rate (see Section \ref{section:special_cases_sobolev}). In the regime \(\log p \lesssim \log\log n\), it is not clear from a lower bound perspective whether a cost for adaptation is unavoidable. We complement their upper bound by providing a lower bound which demonstrates that it is indeed necessary to pay a cost for adaptation. However, the cost we identify may not be sharp. To elaborate, in the case where \(p\) is a large universal constant but much smaller than \(n\), the sparse regime rate (\ref{rate:sobolev_adapt_sparse}) involves \(\sqrt{\log\log\log n}\). From Spokoiny's article \cite{spokoiny_adaptive_1996}, \(\sqrt{\log\log n}\) is instead expected. We leave it for future work to pin down the sharp cost.  

\begin{theorem}\label{thm:sobolev_adapt_lbound_sparse}
    Fix \(\delta \in (0, 1/2)\) and \(\alpha_0 < \alpha_1\). If \(\eta \in (0, 1)\), then there exists \(c_\eta > 0\) depending only on \(\eta\) such that for all \(0 < c < c_\eta\), we have 
    \begin{align*}
        \inf_{\varphi}\left\{ P_0\left\{\varphi \neq 0\right\} + \max_{s < p^{1/2-\delta}} \sup_{\alpha \in [\alpha_0, \alpha_1]} \sup_{\substack{f \in \mathcal{F}(s, \alpha), \\ ||f||_2 \geq c\tau_{\text{sparse}}(p, s, n, \alpha)}} P_f\left\{\varphi \neq 1\right\} \right\} \geq 1-\eta
    \end{align*}
    where \(\tau_{\text{sparse}}\) is given by (\ref{rate:sobolev_adapt_sparse}). 
\end{theorem}

As discussed with respect to Theorem \ref{thm:sobolev_adapt_lbound_dense}, the lower bound in Theorem \ref{thm:sobolev_adapt_lbound_sparse} satisfies Definition \ref{def:adapt_lower_bound} with the straightforward modification to incorporate adaptation over \(\alpha\). 

\section{Discussion}

\subsection{Relaxing the centered assumption}\label{section:uncentered}

In Section \ref{section:parameter}, in the definition of the parameter space (\ref{param:Fs}) the signal \(f\) was constrained to be centered. It was claimed this constraint is mild and can be relaxed; this section elaborates on this point. 

Suppose \(f\) is an additive function (possibly uncentered) with each \(f_j \in \mathcal{H}\). Let \(\bar{f} = \int_{[0,1]^p} f(x) \, dx\) and note we can write \(f = \bar{f}\mathbf{1}_p + (f - \bar{f}\mathbf{1}_p)\) where \(\mathbf{1}_p \in L^2([0,1]^p)\) is the constant function equal to one. Since \(f\) is an additive function, for any \(x \in [0, 1]^p\) we have \(f(x) = \bar{f} + g(x)\) where \(g(x) = \sum_{j=1}^{p} g_j(x_j)\) with \(g_j(x_j) = f_j(x_j) - \int_{0}^{1} f_j(t) \, dt\). In other words, we have \(f = \bar{f}\mathbf{1}_p + g\) where \(g\) itself is an additive function with each \(g_j \in \mathcal{H}_0\).

Asserting \(f\) is a sparse additive function implies \(g\) is a (centered) sparse additive function. Following \cite{raskutti_minimax-optimal_2012,gayraud_detection_2012}, we impose constraints (i.e. RKHS norm constraints) on the centered component functions \(g_j\). Therefore, the following uncentered signal detection problem can be considered,
\begin{align}
    H_0 &: f \equiv 0, \label{problem:uncentered_1}\\
    H_1 &: ||f||_2 \geq \varepsilon \text{, } f \text{ is } s\text{-sparse additive, and } g \in \mathcal{F}_s \label{problem:uncentered_2}
\end{align}
where \(\mathcal{F}_s\) is given by (\ref{param:Fs}). Let \(\varepsilon^*\) denote the minimax separation rate. Consider \(||f||_2^2 = \bar{f}^2 + ||g||_2^2\) by orthogonality. Thus, there are two related detection problems, Problem I
\begin{align*}
    H_0 &: f \equiv 0, \\
    H_1 &: |\bar{f}| \geq \varepsilon_1 \text{ and } f \text{ is } s\text{-sparse additive}
\end{align*}
and Problem II
\begin{align*} 
    H_0 &: f \equiv 0, \\
    H_1 &: ||g||_2 \geq \varepsilon_2 \text{ and } g \in \mathcal{F}_s. 
\end{align*}
Let \(\varepsilon_1^*\) and \(\varepsilon_2^*\) denote the minimax separation rates of the respective problems. We claim \((\varepsilon^*)^2 \asymp (\varepsilon_1^*)^2 + (\varepsilon_2^*)^2\). From \(||f||^2 = \bar{f}^2 + ||g||_2^2 \geq \bar{f}^2 \vee ||g||_2^2\) it is directly seen that \((\varepsilon^*)^2 \gtrsim (\varepsilon_1^*)^2 \vee (\varepsilon_2^*)^2 \asymp (\varepsilon_1^*)^2 + (\varepsilon_2^*)^2\), and so the lower bound is proved. To show the upper bound, consider that if \(||f||_2^2 \geq C \left((\varepsilon_1^*)^2 + (\varepsilon_2^*)^2\right)\), then by triangle inequality it follows that either \(|\bar{f}|^2 \geq C(\varepsilon_1^*)^2\) or \(||g||_2^2 \geq C(\varepsilon_2^*)^2\). Thus, there is enough signal to detect and the test which takes maximum over the subproblem minimax tests is optimal. Hence, the upper bound \((\varepsilon^*)^2 \lesssim (\varepsilon_1^*)^2 + (\varepsilon_2^*)^2\) is now also proved. 

All that remains is to determine the minimax rates \(\varepsilon_1^*\) and \(\varepsilon_2^*\). Let us index starting from zero and take \(\{\psi_{k}\}_{k=0}^{\infty}\) and \(\{\mu_k\}_{k=0}^{\infty}\) to be the associated eigenfunctions and eigenvalues of the RKHS \(\mathcal{H}_0\). Recall \(\{\psi_{k}\}_{k=0}^{\infty}\) forms an orthonormal basis for \(L^2([0, 1])\) under the usual \(L^2\) inner product. Since by definition \(\mathcal{H}_0\) is orthogonal to the span of the constant function in \(L^2([0, 1])\), without loss of generality we can take \(\psi_0\) to be the constant function equal to one, \(\mu_0 = 0\), and \(\{\psi_k\}_{k=1}^{\infty}\) to be the remaining eigenfunctions orthogonal (in \(L^2\)) to constant functions.   

The data \(\{X_{k, j}\}_{k \in \mathbb{N}, j \in [p]}\) (defined in (\ref{data:X})) is sufficient for Problem II. Therefore, the minimax rate \(\varepsilon_2^*\) is exactly given by our result established in this paper. The minimax rate \(\varepsilon_1^*\) can be upper bounded in the following manner. Consider that \(\langle \mathbf{1}_p, dY_x\rangle_{L^2([0, 1]^p)} \sim N(\bar{f}, \frac{1}{n})\) since \(||\mathbf{1}_p||_2^2 = 1\). Therefore, the parametric rate \(\varepsilon_1^* \lesssim \frac{1}{\sqrt{n}}\) can be achieved. Hence, \((\varepsilon^*)^2 \asymp (\varepsilon_1^*)^2 + (\varepsilon_2^*)^2 \asymp (\varepsilon_2^*)^2\). In other words, centering is immaterial to the fundamental limits of the signal detection problem. 

\subsection{Sharp constant}
As we noted earlier, our results do not say anything about the sharp constant in the detection boundary. The problem of obtaining a sharp characterization of the constants in the detection boundary is interesting and likely delicate. In an asymptotic setup and in the Sobolev case, Gayraud and Ingster \cite{gayraud_detection_2012} were able to derive the sharp constant in the sparse regime \(s = p^{1-\beta}\) for fixed \(\beta \in (1/2, 1)\) under the condition \(\log p = o(n^{\frac{1}{2\alpha+1}})\). Gayraud and Ingster discuss that this condition is actually essential, in that the detection boundary no longer exhibits a dependence on \(\beta\) when the condition is violated. This condition has a nice formulation in our notation, namely it is equivalent to the condition 
\begin{equation*}
    \frac{\log(p/s^2)}{n} = o\left(\Gamma_{\mathcal{H}}\right).
\end{equation*}
This correspondence in the Sobolev case suggests this condition may actually be essential for a generic \(\mathcal{H}\). It would be interesting to understand if that is true, and to derive the sharp constant when it holds if it so.

To be clear, it may be the case (perhaps likely) our proposed procedure is suboptimal in terms of the constant. Indeed, the existing literature on sparse signal detection, both in sparse sequence model \cite{donoho_higher_2004,hall_innovated_2010} and sparse linear regression \cite{ingster_detection_2010,arias-castro_global_2011} rely on Higher Criticism type tests to achieve the optimal constant. Gayraud and Ingster \cite{gayraud_detection_2012} themselves use Higher Criticism. For a generic space \(\mathcal{H}\), our procedure should not be understood as the only test which is rate-optimal. In the sparsity regime \(s = p^{1/2-\delta}\), we suspect an analogous Higher Criticism type statistic which accounts for the eigenstructure of the kernel might not only achieve the optimal rate, but also the sharp constant.

\subsection{Future directions}
There are a number of avenues for future work. First, we only considered one space \(\mathcal{H}\) in which all component functions \(f_j\) live in. In some scenarios, it may be desirable to consider a different space \(\mathcal{H}_j\) for each component. Raskutti et al. \cite{raskutti_minimax-optimal_2012} obtained the minimax estimation rate when considering multiple kernels (under some conditions on the kernels). We imagine our broad approach in this work could be extended to determine the minimax separation rate allowing multiple kernels. Instead of a common \(\nu_{\mathcal{H}}\), it is likely different quantities \(\nu_{\mathcal{H}_j}\) will be needed per coordinate and the test statistics could be modified in the natural way. The theory developed here could be used directly, and it seems plausible the minimax separation rate could be established in a straightforward manner.

Another avenue of research involves considering ``soft" sparsity in the form of \(\ell_q\) constraints for \(0 < q < 1\). Yuan and Zhou \cite{yuan_minimax_2016} developed minimax estimation rates for RKHSs exhibiting polynomial decay of its eigenvalues (e.g. Sobolev space). In terms of signal detection, its plausible that the quadratic functional estimator under the Gaussian sequence model with \(\ell_q\) bound on the mean could be extended and used as a test statistic \cite{collier_minimax_2017}. The hard sparsity and soft sparsity settings studied in \cite{collier_minimax_2017} are handled quite similarly. It is possible not much additional theory needs to be developed in order to obtain minimax separation rates under soft sparsity. 

Since hypothesis testing is quite closely related to functional estimation in many problems, it is natural to ask about functional estimation in the context of sparse additive models. For example, it would be of interest to estimate the quadratic functional \(||f||_2^2\) or the norm \(||f||_2\). It is well known in the nonparametric literature that estimating \(L_r\) norms for odd \(r\) yields drastically different rates from testing and from even \(r\) \cite{lepski_estimation_1999,han_estimation_2020}. A compelling direction is to investigate the same problem in the sparse additive model setting. 

Additionally, it would be interesting to consider the nonparametric regression model \(Y_i = f(X_i) + Z_i\) where the design distribution \(X_i \overset{iid}{\sim} P_X\) exhibits some dependence between the coordinates. The correspondence between the white noise model (\ref{model:gwn}) and the nonparametric regression model we relied on requires that the design distribution be close to uniform on \([0, 1]^p\). However, in practical situations it is typically the case that the coordinates of \(X_i\) exhibit dependence, and it would be interesting to understand how the fundamental limits of testing are affected.

Finally, in a somewhat different direction than that discussed so far, it is of interest to study the signal detection problem under group sparsity in other models. As we had encouraged, our results can be interpreted exclusively in terms of sequence space, that is the problem (\ref{problem:seq_test_1})-(\ref{problem:seq_test_2}) with parameter space (\ref{param:Ts}). From this perspective, the group sparse structure is immediately apparent. Estimation has been extensively studied in group sparse settings, especially in linear regression (see \cite{wainwright_high-dimensional_2019} and references therein). Hypothesis testing has not witnessed nearly the same level of research activity, and so there is ample opportunity. We imagine some features of the rates established in this paper are general features of the group sparse structure, and it would be intriguing to discover the commonalities.

\section{Acknowledgments}
SK is grateful to Oleg Lepski for a kind correspondence. The research of SK is supported in part by NSF Grants DMS-1547396 and ECCS-2216912. The research of CG is supported in part by NSF Grant ECCS-2216912, NSF Career Award DMS-1847590, and an Alfred Sloan fellowship.

\section{Proofs}

\subsection{Minimax upper bounds}

\subsubsection{Sparse}
\begin{proof}[Proof of Proposition \ref{prop:test_sparse_bulk}]
    Fix \(\eta \in (0, 1)\). We will set \(C_\eta\) at the end of the proof, so for now let \(C > C_\eta\). The universal constant \(K_1\) will be selected in the course of the proof. Let \(L^*\) denote the universal constant from Lemma \ref{lemma:bulk_alpha}. Set \(K_2 := 1 \vee \frac{1}{(\log 2)^{1/4}} \vee c^{-1/4} \vee \left(\frac{L^{*}}{c}\right)^{1/4}\) and \(K_3 := \frac{L^{*}}{K_2^2}\) where \(c = c^* \wedge c^{**}\) where \(c^*\) and \(c^{**}\) are the universal constants in the exponential terms of Lemmas \ref{lemma:null_bulk} and \ref{lemma:bulk_thold_var} respectively. 

    We first bound the Type I error. Since \(\sqrt{\log\left(1 + \frac{p}{s^2}\right)} \leq K_3 \sqrt{d}\), we have \(1 \leq K_2^2 \sqrt{\log\left(1+\frac{p}{s^2}\right)} \leq K_2^2 K_3 \sqrt{d} \leq L^*\sqrt{d}\). Therefore, we can apply Lemma \ref{lemma:null_bulk} to obtain that 
    \begin{equation*}
        P_0\left\{T_r(d) > C^*\left(\sqrt{xpr^4 e^{-\frac{c^*r^4}{d}}} + \frac{d}{r^2} x\right)\right\} \leq e^{-x}
    \end{equation*}
    for any \(x > 0\). Here, \(C^{*}\) is a universal constant. Taking \(x = C\) and noting that \(C > 1\) provided we select \(C_\eta \geq 1\), we see that 
    \begin{align*}
        C^*\left(\sqrt{x p r^4 e^{-\frac{c^*r^4}{d}}} + \frac{d}{r^2} x\right) &\leq C^*C\left( K_2^2 \sqrt{p d \log\left(1 + \frac{p}{s^2}\right) e^{-c^* K_2^4 \log\left(1 + \frac{p}{s^2}\right)}} + \frac{d}{r^2}\right) \\
        &\leq C^* C \left(K_2^2 \sqrt{d \log\left(1 + \frac{p}{s^2}\right)} \sqrt{p \cdot \frac{s^2}{s^2 + p}} + \frac{d}{r^2}\right) \\
        &\leq C^* C \left(K_2^2 s \sqrt{d \log\left(1 + \frac{p}{s^2}\right)} + \frac{d}{r^2}\right) \\
        &\leq 2C^* K_2^2 C s \sqrt{d \log\left(1 + \frac{p}{s^2}\right)} \\
        &= C K_1 n\tau(p, s, n)^2
    \end{align*}
    where we have used that \(c^* K_2^4 \geq 1\) and where we have selected \(K_1 = 2C^* K_2^2\). Note we have also used that \(\frac{d}{r^2} \leq \sqrt{d} \leq K_2^2 s \sqrt{d\log\left(1 + \frac{p}{s^2}\right)}\). Thus, with these choices of \(K_1, K_2,\) and \(x\), we have 
    \begin{equation*}
        P_0\left\{T_r(d) > CK_1 n \tau(p, s, n)^2\right\} \leq e^{-x} = e^{-C} \leq e^{-C_\eta} \leq \frac{\eta}{2}
    \end{equation*}
    provided we select \(C_\eta \geq \log\left(\frac{2}{\eta}\right) \vee 1\). 

    We now examine the Type II error. To bound the Type II error, we will use Chebyshev's inequality. In particular, consider that for any \(f \in \mathcal{F}_s\) with \(||f||_2 \geq C\tau(p, s, n)\), we have 
    \begin{align}
        P_f\left\{T_r(d) \leq C K_1 n \tau(p, s, n)^2\right\} &= P_f\left\{E_f(T_r(d)) - CK_1 n \tau(p, s, n)^2 \leq E_f(T_r(d)) - T_r(d)\right\} \nonumber \\
        &\leq \frac{\Var_f\left(T_r(d)\right)}{\left(E_f(T_r(d)) - CK_1n\tau(p, s, n)^2\right)^2} \label{eqn:bulk_test_chebyshev}
    \end{align}
    provided that \(E_f(T_r(d)) \geq CK_1n\tau(p, s, n)^2\). To ensure this application of Chebyshev's inequality is valid, we must compute suitable bounds for the expectation and variance of \(T_r(d)\), which the following lemmas provide.

    \begin{lemma}\label{lemma:test_bulk_exp_lbound}
        If \(C_\eta\) is larger than some sufficiently large universal constant, then \(E_f(T_r(d)) \geq 2 CK_1 n \tau(p, s, n)^2\). 
    \end{lemma}
    \begin{lemma}\label{lemma:test_bulk_var_ubound}
        If \(C_\eta\) is larger than some sufficiently large universal constant, then \(\Var_f(T_r(d)) \leq C^\dag\left(s^2r^4 + E_f(T_r(d))\right)\) where \(C^\dag > 0\) is a universal constant. 
    \end{lemma}

    These bounds are proved later on. Let us now describe why they allow us to bound the Type II error. From (\ref{eqn:bulk_test_chebyshev}) as well as Lemmas \ref{lemma:test_bulk_exp_lbound} and \ref{lemma:test_bulk_var_ubound}, we have 
    \begin{align*}
        P_f\left\{T_r(d) \leq C K_1 n\tau(p, s, n)^2\right\} &\leq \frac{\Var_f(T_r(d))}{\left( E_f(T_r(d)) - CK_1 n \tau(p, s, n)^2 \right)^2} \\
        &\leq \frac{C^\dag s^2 r^4 + C^\dag E_f(T_r(d)) }{\left( E_f(T_r(d)) - CK_1 n \tau(p, s, n)^2 \right)^2} \\
        &\leq \frac{C^\dag \left\lceil D \right\rceil K_2^4 s^2 \nu_{\mathcal{H}}\log\left(1 + \frac{p}{s^2}\right)}{C^2 K_1^2 n^2 \tau(p, s, n)^4} + \frac{C^\dag E_f(T_r(d))}{\left( E_f(T_r(d)) - CK_1 n \tau(p, s, n)^2 \right)^2} \\
        &= \frac{C^\dag \left\lceil D \right\rceil K_2^4}{C^2 K_1^2} + \frac{C^\dag E_f(T_r(d))}{\left( E_f(T_r(d)) - CK_1 n \tau(p, s, n)^2 \right)^2} \\
        &\leq \frac{C^\dag \left\lceil D \right\rceil K_2^4}{C^2 K_1^2} + \frac{C^\dag E_f(T_r(d))}{\frac{1}{4}\left( E_f(T_r(d))\right)^2} \\
        &\leq \frac{C^\dag \left\lceil D \right\rceil K_2^4}{C^2 K_1^2} + \frac{4C^\dag }{E_f(T_r(d))} \\
        &\leq \frac{C^\dag \left\lceil D \right\rceil K_2^4}{C^2 K_1^2} + \frac{4C^\dag}{2CK_1n\tau(p, s, n)^2} \\
        &\leq \frac{C^\dag \left\lceil D \right\rceil K_2^4}{C^2 K_1^2} + \frac{2C^\dag}{CK_1\sqrt{\log 2}}
    \end{align*}
    provided we pick \(C_\eta\) larger than a sufficiently large universal constant as required by Lemmas \ref{lemma:test_bulk_exp_lbound} and \ref{lemma:test_bulk_var_ubound}. With this bound in hand and since \(C \geq C_\eta\), we can now pick \(C_\eta\) sufficiently large depending only on \(\eta\) to obtain 
    \begin{equation*}
        \sup_{\substack{f \in \mathcal{F}_s, \\ ||f||_2 \geq C \tau(p, s, n)}}P_f\left\{T_r(d) \leq C K_1 n\tau(p, s, n)\right\} \leq \frac{\eta}{2}.
    \end{equation*}
    To summarize, we can pick \(C_\eta\) depending only on \(\eta\) to be sufficiently large and also satisfying the condition \(C_\eta \geq \log\left(\frac{2}{\eta}\right) \vee 1\) and those of Lemmas \ref{lemma:test_bulk_exp_lbound}, \ref{lemma:test_bulk_var_ubound} to ensure the testing risk is bounded by \(\eta\), i.e. 
    \begin{equation*}
        P_0\left\{T_r(d) > CK_1 n \tau(p, s, n)^2\right\} + \sup_{\substack{f \in \mathcal{F}_s, \\ ||f||_2 \geq C\tau(p, s, n)}} P_f\left\{T_r(d) \leq C K_1 n\tau(p, s, n)^2\right\} \leq \eta, 
    \end{equation*}
    as desired.
\end{proof}

    \noindent It remains to prove Lemmas \ref{lemma:test_bulk_exp_lbound} and \ref{lemma:test_bulk_var_ubound}. Recall we work in the environment of the proof of Proposition \ref{prop:test_sparse_bulk}. 

    \begin{proof}[Proof of Lemma \ref{lemma:test_bulk_exp_lbound}]
        To prove the lower bound on the expectation, first recall
    \begin{equation*}
        E_f(T_r(d)) = \sum_{j=1}^{p} E_f\left( \left(E_j(d) - \alpha_r(d)\right)\mathbbm{1}_{\{E_j(d) \geq d + r^2\}}\right).
    \end{equation*}
    Under \(P_f\), we have \(E_j(d) \sim \chi^2_d(m_j^2)\) where \(m_j^2 = n \sum_{k \leq d} \theta_{k, j}^2\). Here, the collection \(\{\theta_{k, j}\}\) denotes the basis coefficients of \(f\). Intuitively, there are two reasons why the expectation might be small. First, we are thresholding and so we are intuitively removing those coordinates with small, but nonetheless nonzero, means. Furthermore, we are truncating at level \(d\) and not considering higher-order basis coefficients; this also incurs a loss in signal.
    
    Let us first focus on the effect from thresholding. Let \(\widetilde{C}\) denote the universal constant from Lemma \ref{lemma:bulk_thold_exp}. Applying Lemma \ref{lemma:bulk_thold_exp} yields
    \begin{align*}
        E_f(T_r(d)) &= \sum_{j=1}^{p} E_f\left( \left(E_j(d) - \alpha_r(d)\right)\mathbbm{1}_{\{E_j(d) \geq d + r^2\}}\right) \\
        &= \sum_{j \in S_f} E_f\left( \left(E_j(d) - \alpha_r(d)\right)\mathbbm{1}_{\{E_j(d) \geq d + r^2\}}\right) \\
        &\geq \sum_{j \in S_f : m_j^2 \geq \widetilde{C}r^2} \frac{m_j^2}{2} \\
        &= \sum_{j \in S_f} \frac{m_j^2}{2} - \sum_{j \in S_f : m_j^2 < \widetilde{C}r^2} \frac{m_j^2}{2} \\
        &\geq \left(\sum_{j \in S_f} \frac{m_j^2}{2}\right) - \frac{\widetilde{C} sr^2}{2}
    \end{align*}
    where \(S_f \subset [p]\) denotes the subset of active variables \(j\) such that \(\Theta_j \neq 0\). Note that such \(S_f\) exists and \(|S_f| \leq s\) since \(f \in \mathcal{F}_s\). We have thus bounded the amount of signal lost from thresholding. 
    
    Let us now examine the effect of truncation. Consider that 
    \begin{align*}
            ||f||_2^2 &\leq \sum_{j \in S_f} \sum_{k=1}^{\infty} \theta_{k,j}^2 \\
            &= \sum_{j \in S_f} \left(\sum_{k \leq d} \theta_{k,j}^2 + \sum_{k > d} \theta_{k,j}^2\right) \\
            &\leq \sum_{j \in S_f} \left(\sum_{k \leq d} \theta_{k,j}^2 + \mu_{d+1} \sum_{k > d} \frac{\theta_{k, j}^2}{\mu_k}\right) \\
            &\leq \left(\sum_{j \in S_f} \sum_{k \leq d} \theta_{k,j}^2\right) + s \mu_{d+1}.
    \end{align*}
    Here, we have used \(\sum_{k=1}^{\infty} \mu_k^{-1} \theta_{k,j}^2 \leq 1\) for all \(j\). Therefore, we have shown \(\sum_{j \in S_f} \sum_{k \leq d} \theta_{k,j}^2 \geq ||f||_2^2 - s\mu_{d+1}\), and thus have quantified the loss due to truncation.
    
    We are now in position to put together the two pieces. Consider \(||f||_2^2 \geq C^2\tau(p, s, n)^2\). Consequently, 
    \begin{equation*}
        \sum_{j \in S_f} \frac{m_j^2}{2} \geq \frac{C^2 n \tau(p, s, n)^2 - ns\mu_{d+1}}{2} \geq \frac{C^2 n \tau(p, s, n)^2 - ns\mu_{\nu_{\mathcal{H}}+1}}{2}.
    \end{equation*}
    We have used the decreasing order of the kernel's eigenvalues, i.e. that \(d \geq \nu_{\mathcal{H}}\) implies \(\mu_{\nu_{\mathcal{H}}+1} \geq \mu_{d+1}\). Further, consider that by definition of \(\nu_{\mathcal{H}}\) and \(\tau(p, s, n)^2\), we have 
    \begin{equation*}
        \mu_{\nu_{\mathcal{H}}+1} \leq \mu_{\nu_{\mathcal{H}}} \leq \frac{\sqrt{\nu_{\mathcal{H}} \log\left(1 + \frac{p}{s^2}\right)}}{n} = \frac{\tau(p, s, n)^2}{s}.
    \end{equation*}
    With this in hand, it follows that \(\sum_{j \in S_f} \frac{m_j^2}{2} \geq \left(\frac{C^2-1}{2}\right)n \tau(p, s, n)^2\). To summarize, we have shown 
    \begin{align}
        E_f(T_r(d)) &\geq \left(\sum_{j \in S_f} \frac{m_j^2}{2}\right) - \frac{\widetilde{C}sr^2}{2} \nonumber \\
        &\geq \left(\sum_{j \in S_f} \frac{m_j^2}{2}\right) - \frac{\widetilde{C} K_2^2\sqrt{\lceil D \rceil} n \tau(p, s, n)^2}{2} \nonumber \\
        &\geq \left(\sum_{j \in S_f} \frac{m_j^2}{2}\right) \left(1 - \frac{\widetilde{C}K_2^2 \sqrt{\lceil D \rceil}}{C^2 - 1}\right). \label{eqn:bulk_test_exp1}
    \end{align}
    Here, we have used \(d \leq \lceil D \rceil \nu_{\mathcal{H}}\). We can also conclude 
    \begin{equation*}
        E_f(T_r(d)) \geq \left(\frac{C^2 - 1 - K_2^2 \widetilde{C}\sqrt{\lceil D \rceil}}{2}\right)n\tau(p, s, n)^2. 
    \end{equation*}
    In view of the above bound and since \(C > C_\eta\), it suffices to pick \(C_\eta\) large enough to satisfy \(C_\eta^2 - 4K_1 C_\eta \geq 1 + K_2^2 \widetilde{C}\) to ensure \(E_f(T_r(d)) \geq 2 CK_1 n \tau(p, s, n)^2\). The proof of the lemma is complete.
\end{proof}

\begin{proof}[Proof of Lemma \ref{lemma:test_bulk_var_ubound}]
    To bound the variance of \(T_r(d)\), recall that \(\sqrt{\log\left(1 + \frac{p}{s^2}\right)} \leq K_3 \sqrt{d}\), and so \(1 \leq K_2^2 \sqrt{\log\left(1 + \frac{p}{s^2}\right)} \leq K_2^2 K_3 \sqrt{d} \leq L^* \sqrt{d}\). Therefore, we can apply Lemma \ref{lemma:bulk_thold_var}. By Lemma \ref{lemma:bulk_thold_var}, we have 
    \begin{align*}
        \Var_f\left(T_r(d)\right) &= \sum_{j=1}^{p} \Var_f\left((E_j(d) - \alpha_r(d))\mathbbm{1}_{\{E_j(d) \geq d + r^2\}}\right) \\
        &\leq C^\dag pr^4 \exp\left(-\frac{c^{**}r^4}{d}\right) + C^\dag sr^4 + C^\dag \sum_{j \in S_f : m_j^2 > 4r^2} m_j^2 \\
        &\leq C^\dag pr^4 \exp\left(-\frac{c^{**}r^4}{d}\right) + C^\dag sr^4 + C^\dag \sum_{j \in S_f} m_j^2
    \end{align*}
    where \(C^\dag\) is a positive universal constant whose value can change from instance to instance. Recall \(c^{**}\) is defined at the beginning of the proof of Proposition \ref{prop:test_sparse_bulk}. Since \(r^4 = K_2^4 d \log\left(1 + \frac{p}{s^2}\right)\) and \(c^{**}K_2^4 \geq 1\), we have 
    \begin{align*}
        pr^4 \exp\left(-\frac{c^{**}r^4}{d} \right) &\leq pr^4 \exp\left(- \log\left(1 + \frac{p}{s^2}\right)\right) \leq pr^4 \cdot \frac{s^2}{s^2 + p} \leq s^2 r^4.
    \end{align*}
    Therefore, 
    \begin{equation*}
        \Var_f(T_r(d)) \leq C^\dag s^2 r^4 + C^\dag s r^4 + C^\dag \sum_{j \in S_f} m_j^2 \leq 2C^\dag s^2 r^4 + C^\dag \sum_{j \in S_f} m_j^2.
    \end{equation*}
    Taking \(C_\eta\) larger than a sufficiently large universal constant, noting \(C \geq C_\eta\), and invoking (\ref{eqn:bulk_test_exp1}), we have the desired result. 
\end{proof}

\begin{proof}[Proof of Proposition \ref{prop:test_sparse_tail}]
    Our proof is largely the same as in the proof of Proposition \ref{prop:test_sparse_bulk}, except we invoke results about the ``tail" rather than the ``bulk''. Fix \(\eta \in (0, 1)\). We will make a choice of \(C_\eta\) at the end of the proof, so for now let \(C > C_\eta\). We select \(K_1\) in the course of the proof, but we will select \(K_2\) now. Set \(K_2 := \frac{1}{\sqrt{\log 2}} \vee c^{-1/2} \vee \left(c K_3^2\right)^{-1/4}\) with \(c = c^* \wedge c^{**}\) where \(c^{*}\) and \(c^{**}\) are the universal constants in the exponential terms of Lemmas \ref{lemma:null_tail} and \ref{lemma:tail_thold_var} respectively.

    We first bound the Type I error. Since \(\log\left(1 + \frac{p}{s^2}\right) \geq K_3^2 d\), we have \(r^2 \gtrsim d\). Since \(d \geq D\), we have by Lemma \ref{lemma:null_tail} that for any \(x > 0\), 
    \begin{equation*}
        P_0\left\{ T_r(d) > C^* \left( \sqrt{xpr^4 e^{-c^*r^2}} + x \right) \right\} \leq e^{-x}
    \end{equation*}
    where \(C^*\) is a universal positive constant. Taking \(x = C\) and noting \(C > 1\) provided we have chosen \(C_\eta \geq 1\), we see that 
    \begin{align*}
        C^*\left(\sqrt{xpr^4 e^{-c^*r^2}} + x\right) &\leq C^* C\left(K_2^2 \log\left(1 + \frac{p}{s^2}\right) \sqrt{p e^{-c^* K_2^2 \log\left(1 + \frac{p}{s^2}\right)}} + 1\right) \\
        &\leq C^* C \left(K_2^2 \log\left(1 + \frac{p}{s^2}\right) \sqrt{p \cdot \frac{s^2}{s^2 + p}} + 1\right) \\
        &\leq 2C^* C K_2^2 s \log\left(1 + \frac{p}{s^2}\right) \\
        &= CK_1 n \tau(p, s, n)^2 
    \end{align*}
    where we have used that \(c^* K_2^2 \geq 1\) and we have set \(K_1 := 2C^* K_2^2\). Thus, with these choices of \(K_1, K_2,\) and \(x\) we have 
    \begin{equation*}
        P_0\left\{ T_r(d) > CK_1 n \tau(p, s, n)^2 \right\} \leq e^{-x} = e^{-C} \leq e^{-C_\eta} \leq \frac{\eta}{2}
    \end{equation*}
    provided we select \(C_\eta \geq \log\left(\frac{2}{\eta}\right) \vee 1\). 

    We now examine the Type II error. To bound the Type II error, we will use Chebyshev's inequality. In particular, consider that for any \(f \in \mathcal{F}_s\) with \(||f||_2 \geq C \tau(p, s, n)\), we have 
    \begin{align}
        P_f\left\{ T_r(d) \leq C K_1 n\tau(p, s, n)^2 \right\} &= P_f\left\{ E_f(T_r(d)) - CK_1 n\tau(p, s, n)^2 \leq E_f(T_r(d)) - T_r(d)\right\} \nonumber \\
        &\leq \frac{\Var_f\left(T_r(d)\right)}{\left(E_f(T_r(d)) - CK_1n\tau(p, s, n)^2\right)^2} \label{eqn:tail_test_chebyshev}
    \end{align}
    provided that \(E_f(T_r(d)) \geq CK_1n\tau(p, s, n)^2\). To ensure this application of Chebyshev's inequality is valid and to bound the Type II error, we will need a lower bound on the expectation of \(T_r(d)\). We will also need an upper bound on the variance of \(T_r(d)\) in order to bound the Type II error. The following lemmas provide us with the requisite bounds; they are analogous to Lemmas \ref{lemma:test_bulk_exp_lbound} and \ref{lemma:test_bulk_var_ubound} but are now in the context of the tail regime.
    
    \begin{lemma}\label{lemma:test_tail_exp_lbound}
        If \(C_\eta\) is larger than some sufficiently larger universal constant, then \(E_f(T_r(d)) \geq 2CK_1 n \tau(p, s, n)^2\).     
    \end{lemma}
    \begin{lemma}\label{lemma:test_tail_var_lbound}
        If \(C_\eta\) is larger than some sufficiently large universal constant, then \(\Var_f(T_r(d)) \leq C^{\dag}(s^2r^4 + E_f(T_r(d)))\) where \(C^\dag> 0\) is a universal constant.
    \end{lemma}

    With these bounds in hand, the argument in the proof of Proposition \ref{prop:test_sparse_bulk} can be essentially repeated to establish that 
    \begin{equation*}
        P_0\left\{T_r(d) > CK_1n\tau(p, s, n)^2 \right\} + \sup_{\substack{f \in \mathcal{F}_s, \\ ||f||_2 \geq C \tau(p, s, n)}}  P_f\left\{T_r(d) \leq CK_1 n \tau(p, s, n)^2\right\} \leq \eta
    \end{equation*}
    provided \(C_\eta \geq \log\left(\frac{2}{\eta}\right) \vee 1\) and \(C_\eta\) sufficiently large to satisfy Lemmas \ref{lemma:test_tail_exp_lbound} and \ref{lemma:test_tail_var_lbound}. We omit the details for brevity.
\end{proof}

\noindent It remains to prove Lemmas \ref{lemma:test_tail_exp_lbound} and \ref{lemma:test_tail_var_lbound}. 
\begin{proof}[Proof of Lemma \ref{lemma:test_tail_exp_lbound}]
    The proof is similar in style to the proof of Lemma \ref{lemma:test_bulk_exp_lbound}, except now results for the tail regime are invoked. Letting \(\widetilde{C}\) denote the universal constant from Lemma \ref{lemma:tail_thold_exp}, applying Lemma \ref{lemma:tail_thold_exp}, and arguing similarly to the proof of Lemma \ref{lemma:test_bulk_exp_lbound} , we obtain 
    \begin{equation*}
        E_f\left( T_r(d) \right) \geq \left(\sum_{j \in S_f} \frac{m_j^2}{2}\right) - \frac{\widetilde{C}s r^2}{2}
    \end{equation*}
    where \(S_f \subset [p]\) denotes the subset of active variables \(j\) such that \(\Theta_j \neq 0\). Note that such \(S_f\) exists and \(|S_f| \leq s\) since \(f \in \mathcal{F}_s\). Further arguing like the proof of Lemma \ref{lemma:test_bulk_exp_lbound} and using \(\sqrt{\log\left(1 + \frac{p}{s^2}\right)} \geq K_3 \sqrt{d}\), we have  

    \begin{equation}\label{eqn:tail_signal}
        \sum_{j \in S_f} \frac{m_j^2}{2} \geq \left(\frac{C^2 - \frac{1}{K_3}}{2}\right) n \tau(p, s, n)^2.
    \end{equation}
    To summarize, we have shown 
    \begin{align}\label{eqn:tail_test_exp_1}
        E_f\left(T_r(d)\right) &\geq \left(\sum_{j \in S_f} \frac{m_j^2}{2}\right) - \frac{\widetilde{C}s r^2}{2} \nonumber \\
        &\geq \left(\sum_{j \in S_f} \frac{m_j^2}{2}\right) - \frac{\widetilde{C} K_2^2 n\tau(p, s, n)^2}{2} \nonumber \\
        &\geq \left(\sum_{j \in S_f} \frac{m_j^2}{2}\right) \left(1 - \frac{\widetilde{C}K_2^2}{C^2 - \frac{1}{K_3}}\right).
    \end{align}
    Note that we also can conclude
    \begin{equation*}
        E_f\left( T_r(d) \right) \geq \left(\frac{C^2 - \frac{1}{K_3} - K_2^2 \widetilde{C}}{2} \right) n\tau(p, s, n)^2. 
    \end{equation*}
    Since \(C > C_\eta\), it suffices to pick \(C_\eta\) large enough to satisfy \(C_\eta^2 - 4K_1 C_\eta \geq \frac{1}{K_3} + K_2^2\widetilde{C}\) to ensure \(E_f(T_r(d)) \geq 2CK_1 n \tau(p, s, n)^2\). The proof of the lemma is complete.
\end{proof}

\begin{proof}[Proof of Lemma \ref{lemma:test_tail_var_lbound}]
    Recall that \(\log\left(1 + \frac{p}{s^2}\right) \geq K_3^2 d\). By definition of \(r^2\) we have \(r^2 \geq K_2^2K_3^2 d\). In other words \(r^2 \gtrsim d\) and so we can apply Lemma \ref{lemma:tail_thold_var}. By Lemma \ref{lemma:tail_thold_var}, we have 
    \begin{align*}
        \Var_f\left(T_r(d)\right) &= \sum_{j=1}^{p} \Var_f\left(\left(E_j(d) - \alpha_r(d)\right)\mathbbm{1}_{\{E_j(d) \geq d + r^2\}}\right) \\
        &\leq C^{\dag} p r^4\exp\left(-c^{**} \min\left(\frac{r^4}{d}, r^2\right) \right) + C^\dag s r^4 + C^\dag \sum_{j \in S_f : m_j^2 > 4r^2} m_j^2 \\
        &\leq C^{\dag} p r^4\exp\left(-c^{**} \min\left(\frac{r^4}{d}, r^2\right) \right) + C^\dag s r^4 + C^\dag \sum_{j \in S_f} m_j^2.
    \end{align*}
    where \(C^{\dag}\) is a positive universal constant. Recall we had defined \(c^{**}\) at the beginning of the proof of Proposition \ref{prop:test_sparse_tail}. Since \(r^2 \geq K_2^2 K_3^2 d\), we have \(\frac{r^4}{d} \geq r^2 K_2^2 K_3^2\). Therefore, 
    \begin{align*}
        \exp\left(-c^{**} \min\left(\frac{r^4}{d}, r^2\right) \right) \leq \exp\left(-c^{**} \left(K_2^2K_3^2 \wedge 1\right) r^2 \right) \leq \exp\left(- \log\left(1 + \frac{p}{s^2}\right)\right) \leq \frac{s^2}{s^2+p}
    \end{align*}
    where we have used that \(c^{**} \left(K_2^2K_3^2 \wedge 1\right) K_2^2 \geq 1\) by definition of \(K_2\). Therefore, 
    \begin{equation*}
        \Var_f\left(T_r(d)\right) \leq 2C^\dag s^2 r^4 + C^\dag  \sum_{j \in S_f} m_j^2.
    \end{equation*}
    Taking \(C_\eta\) larger than a sufficiently large universal constant and invoking (\ref{eqn:tail_test_exp_1}) yields the desired result.
\end{proof}

\subsubsection{Dense}

\begin{proof}[Proof of Proposition \ref{prop:dense_alpha}]
    For ease of notation, set \(L = \frac{2}{\sqrt{\eta}}\). Define \(C_\eta := \left(\sqrt{2\sqrt{2}L + 1}\right) \vee \left( \sqrt{1 + \frac{4}{\sqrt{\eta}} + L} \right) \vee \left(\sqrt{\frac{64\sqrt{2}}{\eta} + 1} \right)\). Let \(C > C_\eta\). We first bound the Type I error. Consider that under \(P_0\) we have \(T \sim \chi^2_{p\nu_\mathcal{H}}\). Observe that \(E_0\left(T\right) = p\nu_{\mathcal{H}}\) and \(\Var_0\left(T\right) = 2p\nu_{\mathcal{H}}\). Consequently, we have by Chebyshev's inequality
    \begin{equation*}
        P_0\left\{ T > p \nu_{\mathcal{H}} + L \sqrt{p \nu_{\mathcal{H}}}\right\} \leq \frac{\Var_0(T)}{L^2 p \nu_{\mathcal{H}}} = \frac{2}{L^2} \leq \frac{\eta}{2}. 
    \end{equation*}
    We now examine the Type II error. Under \(P_f\), we have \(T \sim \chi^2_{p\nu_{\mathcal{H}}}\left(n \sum_{j = 1}^{p} \sum_{k \leq \nu_{\mathcal{H}}} \theta_{k,j}^2 \right)\) where \(\{\theta_{k,j}\}\) denote the basis coefficients associated to \(f\). Note \(E_f(T) = p\nu_{\mathcal{H}} + n \sum_{j=1}^{p} \sum_{k \leq \nu_H} \theta_{k,j}^2\) and \(\Var_f(T) = 2p\nu_{\mathcal{H}} + 4n\sum_{j= 1}^{p} \sum_{k \leq \nu_\mathcal{H}} \theta_{k,j}^2\). In order to proceed with the argument, we need to obtain a lower bound estimate for the signal strength. Letting \(S\) denote the set of \(j\) for which \(\Theta_j\) are nonzero, consider that for \(||f||_2^2 \geq C^2\tau(p, s, n)^2\) we have
    \begin{align*}
        C^2 \tau(p, s, n)^2 &\leq ||f||_2^2 \\
        &= \sum_{j=1}^{p} \sum_{k = 1}^{\infty} \theta_{k,j}^2 \\
        &\leq \sum_{j \in S} \left(\sum_{k \leq \nu_\mathcal{H}} \theta_{k,j}^2 + \sum_{k > \nu_{\mathcal{H}}} \theta_{k,j}^2\right) \\
        &\leq \sum_{j \in S} \left( \sum_{k \leq \nu_{\mathcal{H}}} \theta_{k,j}^2 + \mu_{\nu_{\mathcal{H}}} \sum_{k > \nu_{\mathcal{H}}} \frac{\theta_{k,j}^2}{\mu_k}\right) \\
        &\leq \left(\sum_{j \in S} \sum_{k \leq \nu_{\mathcal{H}}} \theta_{k,j}^2\right) + s\mu_{\nu_{\mathcal{H}}}.
    \end{align*}
    We have used \(\sum_{k=1}^{\infty} \frac{\theta_{k,j}^2}{\mu_k} \leq 1\) for all \(j\) in the final line. Therefore by definition of \(\nu_{\mathcal{H}}\) we have
    \begin{align*}
        n \sum_{j \in S} \sum_{k \leq \nu_{\mathcal{H}}} \theta_{k,j}^2 \geq nC^2\tau(p, s, n)^2 - n s \mu_{\nu_{\mathcal{H}}} \geq \left(C_\eta^2 - 1\right) \sqrt{\nu_{\mathcal{H}} p}
    \end{align*}
    where we have used that \(ns\mu_{\nu_{\mathcal{H}}} \leq s \sqrt{\nu_{\mathcal{H}} \log\left(1 + \frac{p}{s^2}\right)} \leq \sqrt{p \nu_{\mathcal{H}}}\).

    We now continue with bounding the Type II error. By Chebyshev's inequality, we have 
    \begin{align*}
        &\sup_{\substack{f \in \mathcal{F}_s, \\ ||f||_2 \geq C_\eta \tau(p, s, n)}} P_f\left\{ T \leq p \nu_{\mathcal{H}} + L \sqrt{p \nu_{\mathcal{H}}} \right\} \\
        &= \sup_{\substack{f \in \mathcal{F}_s, \\ ||f||_2 \geq C_\eta \tau(p, s, n)}} P_f\left\{ E_f(T) - p\nu_{\mathcal{H}} -  L \sqrt{p \nu_{\mathcal{H}}} \leq E_f(T) - T \right\}\\
        &\leq \sup_{\substack{f \in \mathcal{F}_s, \\ ||f||_2 \geq C_\eta \tau(p, s, n)}} \frac{\Var_f\left( T \right)}{\left(E_f(T) - p\nu_{\mathcal{H}} - L\sqrt{p\nu_{\mathcal{H}}}\right)^2} \\
        &\leq \sup_{\substack{f \in \mathcal{F}_s, \\ ||f||_2 \geq C_\eta \tau(p, s, n)}} \frac{2p\nu_\mathcal{H} + 4n\sum_{j=1}^{p} \sum_{k \leq \nu_\mathcal{H}} \theta_{k,j}^2}{\left( n\sum_{j=1}^{p} \sum_{k \leq \nu_{\mathcal{H}}} \theta_{k,j}^2 - L\sqrt{p\nu_{\mathcal{H}}}\right)^2} \\
        &= \sup_{\substack{f \in \mathcal{F}_s, \\ ||f||_2 \geq C_\eta \tau(p, s, n)}} \frac{2p\nu_\mathcal{H}}{\left(C_\eta^2 - 1 - L\right)^2\nu_{\mathcal{H}}p} + \frac{4n\sum_{j=1}^{p} \sum_{k \leq \nu_\mathcal{H}} \theta_{k,j}^2}{\left( n\sum_{j=1}^{p} \sum_{k \leq \nu_{\mathcal{H}}} \theta_{k,j}^2 - L\sqrt{p\nu_{\mathcal{H}}}\right)^2} \\
        &\leq \sup_{\substack{f \in \mathcal{F}_s, \\ ||f||_2 \geq C_\eta \tau(p, s, n)}} \frac{2}{\left(C_\eta^2 - 1 - L\right)^2} + \frac{4n\sum_{j=1}^{p} \sum_{k \leq \nu_\mathcal{H}} \theta_{k,j}^2}{\left( \frac{1}{2} n\sum_{j=1}^{p} \sum_{k \leq \nu_{\mathcal{H}}} \theta_{k,j}^2\right)^2} \\
        &= \sup_{\substack{f \in \mathcal{F}_s, \\ ||f||_2 \geq C_\eta \tau(p, s, n)}} \frac{2}{\left(C_\eta^2 - 1 - L\right)^2} + \frac{16}{n\sum_{j=1}^{p} \sum_{k \leq \nu_{\mathcal{H}}} \theta_{k,j}^2} \\
        &\leq \frac{2}{\left(C_\eta^2 - 1 - L\right)^2} + \frac{16}{(C_\eta^2-1)\sqrt{\nu_H p}} \\
        &\leq \frac{2}{\left(C_\eta^2 - 1 - L\right)^2} + \frac{16}{C_\eta^2-1} \\
        &\leq \frac{\eta}{4} + \frac{\eta}{4} \\
        &\leq \frac{\eta}{2}. 
    \end{align*}
    Therefore, the sum of Type I and Type II errors is bounded by \(\eta\) as desired.
\end{proof}

\subsection{Minimax lower bounds}

\begin{proof}[Proof of Proposition \ref{prop:lbound_trivial}]
    We break up the analysis into two cases. 

    \textbf{Case 1:} Suppose \(s < \sqrt{p}\). We will construct a prior distribution \(\pi\) on \(\mathscr{T}_s\) and use Le Cam's two point method to furnish a lower bound. Define \(c_\eta := 1 \wedge \sqrt{\kappa} \wedge \sqrt{\kappa \log\left(1 + 4\eta^2\right)}\). Let \(0 < c < c_\eta\). Let \(\pi\) be the prior in which a draw \(\Theta \sim \pi\) is constructed by uniformly drawing \(S \subset [p]\) of size \(s\) and setting 
    \begin{equation*}
        \Theta_j = 
        \begin{cases}
            c e_1 &\textit{if } j \in S, \\
            0 &\textit{otherwise}
        \end{cases}
    \end{equation*}
    where \(e_1 \in \R^{\mathbb{N}}\) is given by \(e_1 = (1,0,0,...)\). Note that \(\Theta \sim \pi\) implies \(||\Theta||_F^2 = c^2 s\) and \(\Theta \in \mathscr{T}_s\). By Neyman-Pearson lemma and the inequality \(1-d_{TV}(Q, P) \geq 1-\sqrt{\chi^2(Q||P)}/2\), we have 
    \begin{equation}\label{eqn:lbound_trivial_chisquare}
        \inf_{\varphi}\left\{ P_0\left\{ \varphi \neq 0 \right\} + \sup_{\substack{\Theta \in \mathscr{T}_s, \\ ||\Theta||_F \geq c \sqrt{s}}} P_{\Theta}\left\{\varphi \neq 1\right\} \right\} \geq 1 - \frac{1}{2}\sqrt{\chi^2(P_{\pi}||P_0)}
    \end{equation}
    where \(P_\pi\) denotes the mixture \(\int P_\Theta \, \pi(d\Theta)\) induced by \(\pi\). By the Ingster-Suslina method (Proposition \ref{prop:Ingster_Suslina}) and Corollary \ref{corollary:hypergeometric}, we have 
    \begin{align*}
        \chi^2(P_{\pi}||P_0) &= E\left(\exp\left(n \langle \Theta, \widetilde{\Theta}\rangle_F \right)\right) - 1\\
        &= E\left(\exp\left( c^2 n |S \cap \widetilde{S}| \right)\right) - 1\\
        &\leq \left(1 - \frac{s}{p} + \frac{s}{p} e^{c^2n}\right)^s - 1.
    \end{align*}
    where \(\Theta, \widetilde{\Theta} \overset{iid}{\sim} \pi\) and \(S, \widetilde{S}\) are the corresponding random sets. Now, using that \(\log\left(1 + \frac{p}{s^2}\right) \geq \kappa n\), we have 
    \begin{align*}
        &\leq \left(1 - \frac{s}{p} + \frac{s}{p} e^{c^2 \kappa^{-1} \log\left(1 + \frac{p}{s^2}\right)}\right)^s - 1 \\
        &= \left(1 - \frac{s}{p} + \frac{s}{p}\left(1 + \frac{p}{s^2}\right)^{c^2\kappa^{-1}}\right)^s - 1\\
        &\leq \left(1 - \frac{s}{p} + \frac{s}{p} + \frac{c^2\kappa^{-1}}{s}\right)^{s} - 1 \\ 
        &\leq \exp\left(\frac{c^2}{\kappa}\right) - 1 \\
        &\leq 4\eta^2
    \end{align*}
    We have also used that \(c^2\kappa^{-1} < c_\eta^2\kappa^{-1} \leq 1\) and the inequality \((1+x)^y \leq 1+xy\) for \(x > 0\) and \(y \in (0, 1)\). Plugging into (\ref{eqn:lbound_trivial_chisquare}) yields the desired result.

    \textbf{Case 2:} Suppose \(s \geq \sqrt{p}\). We can repeat the analysis done in Case 1 with the modification of replacing every instance of \(s\) with \(\lceil \sqrt{p} \rceil\). 
\end{proof}

\begin{proof}[Proof of Theorem \ref{thm:lbound_nontrivial}]
    We will construct a prior distribution \(\pi\) on \(\mathscr{T}_s\) and use Le Cam's two point method to furnish a lower bound. Set \(c_\eta := 1 \wedge 2^{1/4} \wedge \left(\log\left(1 + 4\eta^2\right)\right)^{1/4}\). Let \(0 < c < c_\eta\). We break up the analysis into two cases. 
    
    \textbf{Case 1:} Suppose we are in the regime where \(\psi(p, s, n)^2 = s\Gamma_{\mathcal{H}}\). Set \(\rho := \sqrt{\frac{\Gamma_{\mathcal{H}}}{\nu_{\mathcal{H}}}}\). Let \(\pi\) be the prior in which a draw \(\Theta \sim \pi\) is obtained by uniformly drawing \(S \subset [p]\) of size \(s\) and drawing 
    \begin{equation*}
        \theta_{i, k} \sim
        \begin{cases}
            \Uniform\{-c\rho, c\rho\} &\textit{if } i \in S \text{ and } k < \nu_{\mathcal{H}}, \\
            \delta_0 &\textit{otherwise}
        \end{cases}
    \end{equation*}
    where \(\delta_0\) denotes the probability measure placing full mass at zero. Note that \(\Theta \sim \pi\) implies \(||\Theta||_F^2 = c^2 \rho^2 s (\nu_{\mathcal{H}}-1) = c^2 s \Gamma_{\mathcal{H}} \cdot \frac{\nu_{\mathcal{H}}-1}{\nu_{\mathcal{H}}} \geq \frac{c^2}{2} s \Gamma_{\mathcal{H}}\). Here, we have used that \(\nu_{\mathcal{H}} \geq 2\) since \(\log\left(1 + \frac{p}{s^2}\right) \leq \frac{n}{2}\). Furthermore, consider that for \(i \in S\), we have 
    \begin{align*}
        \sum_{\ell=1}^{\infty} \frac{\theta_{i,\ell}^2}{\mu_\ell} = c^2\rho^2 \sum_{\ell < \nu_{\mathcal{H}}} \frac{1}{\mu_\ell} = c^2 \frac{1}{\nu_{\mathcal{H}}} \sum_{\ell < \nu_{\mathcal{H}}} \frac{\Gamma_{\mathcal{H}}}{\mu_\ell} \leq c^2 \frac{1}{\nu_{\mathcal{H}}} \sum_{\ell < \nu_{\mathcal{H}}} \frac{\mu_{\nu_{\mathcal{H}}-1}}{\mu_\ell} \leq c^2 \leq c_\eta^2 \leq 1.
    \end{align*}
    Here, we have used the ordering of the eigenvalues \(\mu_1 \geq \mu_2 \geq ... \geq 0\) and \(\Gamma_{\mathcal{H}} \leq \mu_{\nu_{\mathcal{H}}-1}\) by Lemma \ref{lemma:exact_Gamma}. Of course, for \(i \not \in S\) we have \(\Theta_i = 0\). Hence we have \(\Theta \in \mathscr{T}_s\). We then have by the Neyman-Pearson lemma and the inequality \(1-d_{TV}(Q, P) \geq 1-\sqrt{\chi^2(Q||P)}/2\) that 
    \begin{equation}\label{eqn:minimax_lbound_chisquare}
        \inf_{\varphi}\left\{ P_0\left\{ \varphi \neq 0 \right\} + \sup_{\substack{\Theta \in \mathscr{T}_s, \\ ||\Theta||_F \geq c \psi(p, s, n)}} P_{\Theta}\left\{\varphi \neq 1 \right\} \right\} \geq 1 - \frac{1}{2}\sqrt{\chi^2(P_{\pi}||P_0)}. 
    \end{equation}
    For the following calculations, let \(\Theta, \Theta' \overset{iid}{\sim} \pi\) and \(S, S'\) are the corresponding random sets. Also, let \(\left\{r_{i\ell}, r'_{i\ell}\right\}_{1 \leq i \leq p, \ell \in \mathbb{N}}\) denote an iid collection of \(\Rademacher\left(1/2\right)\) random variables. By the Ingster-Suslina method (Proposition \ref{prop:Ingster_Suslina}), independence of \(\{S, S'\}\) with \(\left\{r_{i\ell}, r'_{i\ell}\right\}_{1 \leq i \leq p, \ell \in \mathbb{N}}\), and Corollary \ref{corollary:hypergeometric}, we have
    \begin{align*}
        \chi^2(P_{\pi}||P_0) &= E\left(\exp\left(n \langle \Theta, \Theta'\rangle_F \right)\right) - 1\\
        &= E\left(\exp\left(n c^2\rho^2 \sum_{i \in S \cap S'} \sum_{\ell \leq \nu_{\mathcal{H}}} r_{i\ell} r'_{i\ell} \right)\right) - 1\\
        &= E\left(\prod_{i \in S \cap S'} E\left(\exp\left(nc^2 \rho^2 \sum_{\ell \leq \nu_{\mathcal{H}}} r_{i\ell} r'_{i\ell}\right)\right)\right) - 1 \\
        &= E\left(\cosh\left(nc^2\rho^2 \right)^{\nu_{\mathcal{H}}|S \cap S'|}\right) - 1\\
        &\leq E\left(\exp\left(\frac{c^4 n^2 \rho^4 \nu_{\mathcal{H}}}{2} |S \cap S'| \right) \right) - 1 \\
        &\leq \left(1 - \frac{s}{p} + \frac{s}{p} e^{\frac{c^4 n^2 \rho^4 \nu_{\mathcal{H}}}{2}}\right)^s - 1.
    \end{align*}
    Consider by Lemma \ref{lemma:fixed_point_equiv} that \(c^4 n^2 \rho^4 \nu_{\mathcal{H}} = \frac{c^4 n^2 \Gamma_{\mathcal{H}}^2}{\nu_{\mathcal{H}}} \leq c^4 \log\left(1 + \frac{p}{s^2}\right)\). Therefore, 
    \begin{align*}
        \left(1 - \frac{s}{p} + \frac{s}{p} e^{\frac{c^4 n^2 \rho^4 \nu_{\mathcal{H}}}{2}}\right)^s - 1 \leq \left(1 - \frac{s}{p} + \frac{s}{p} e^{\frac{c^4}{2} \log\left(1 + \frac{p}{s^2}\right)}\right)^s - 1 \leq \left(1 + \frac{c^4}{2s}\right)^s - 1 \leq e^{c_\eta^4/2} - 1 = 4\eta^2.
    \end{align*}
    We have used \(\frac{c^4}{2} < \frac{c_\eta^4}{2} \leq 1\) and the inequality \((1+x)^y \leq 1+xy\) for \(x > 0\) and \(y \in (0, 1)\). Using \(\chi^2(P_{\pi}||P_0) \leq 4\eta^2\) with (\ref{eqn:minimax_lbound_chisquare}) yields the desired result.

    \textbf{Case 2:} Suppose \(s < \sqrt{p}\) and \(\psi(p, s, n)^2 = \frac{s}{n} \log\left(1 + \frac{p}{s^2}\right)\). Set \(\rho = 1\). Let \(\pi\) be the prior in which a draw \(\Theta \sim \pi\) is obtained by uniformly drawing \(S \subset [p]\) of size \(s\) and drawing 
    \begin{equation*}
        \theta_{i,k} \sim 
        \begin{cases}
            \Uniform\{-c\rho, c\rho\} &\textit{if } i \in S \text{ and } k =1, \\
            \delta_0 &\textit{otherwise}.
        \end{cases}
    \end{equation*}
    The desired result can be proved by arguing in a manner similar to that as in the proof of Theorem \ref{thm:lbound_nontrivial} (see also \cite{collier_minimax_2017}). Details are omitted for the sake of brevity.
\end{proof}

\subsection{Adaptive lower bound}

\begin{proof}[Proof of Theorem \ref{thm:adapt_lbound}]
    Recall the definitions of \(\mathscr{A}_{\mathcal{H}}\) in (\ref{def:script_A}) and \(\widetilde{\mathscr{V}}_{\mathcal{H}}\) in (\ref{def:tilde_script_V}). By assumption, we have \(\mathscr{A}_{\mathcal{H}} \leq L \log(e|\widetilde{\mathscr{V}}_{\mathcal{H}}|)\) for a universal constant \(L > 0\). For ease of notation, let us denote \(\psi = \psi_{\text{adapt}}(p, s, n)\) where \(\psi_{\text{adapt}}\) is given by (\ref{rate:psi_adapt}). Set \(c_\eta := (16L)^{-1/4} \wedge \left(\frac{\eta^2 \log\left(1 + 2\eta^2\right)}{32L}\right)^{1/4}\) and let \(0 < c < c_\eta\). We now define a prior \(\pi\). A draw \(\Theta \sim \pi\) is obtained as follows. First, draw \(v \sim \text{Uniform}(\widetilde{\mathscr{V}}_{\mathcal{H}})\). Set \(s := \min\left\{s' \in [p] : s' \geq \sqrt{p\mathscr{A}_{\mathcal{H}}} \text{ and } \frac{v}{2} < \nu_{\mathcal{H}}(s', \mathscr{A}_{\mathcal{H}}) \leq v\right\}\). Then draw uniformly at random a subset \(S \subset [p]\) of size exactly \(s\). Let \(d_s := 2^{k_s} \in \widetilde{\mathscr{V}}_{\mathcal{H}}\) satisfy \(2^{k_s-1} < \nu_{\mathcal{H}}(s, \mathscr{A}_{\mathcal{H}}) \leq 2^{k_s}\). Draw independently 
    \begin{equation*}
        \theta_{k, j} \sim
        \begin{cases}
            \Uniform\left\{-\sqrt{2}c\rho_s, \sqrt{2}c \rho_s\right\} &\textit{if } j \in S \text{ and } 1 \leq k < \nu_{\mathcal{H}}(s, \mathscr{A}_{\mathcal{H}}), \\
            \delta_0 &\textit{otherwise}. 
        \end{cases}
    \end{equation*}
    Here, \(\rho_s = \sqrt{\frac{\psi^2/s}{\nu_{\mathcal{H}}(s, \mathscr{A}_{\mathcal{H}})}}\). This description concludes the definition of \(\pi\). \newline

    Now we must show that \(\pi\) is indeed supported on \(\bigcup_{s} \mathscr{T}_s\). Note that when we draw \(\Theta \sim \pi\), the associated sparsity level \(s\) always satisfies \(s \geq \sqrt{p\mathscr{A}_{\mathcal{H}}}\). Furthermore, conditional on \(s\) we have 
    \begin{equation*}
        ||\Theta||_F^2 = 2c^2 \frac{\psi^2}{\nu_{\mathcal{H}}(s, \mathscr{A}_{\mathcal{H}})} (\nu_{\mathcal{H}}(s, \mathscr{A}_{\mathcal{H}}) - 1) \geq c^2 \psi^2.
    \end{equation*}
    Furthermore, consider that for \(j \in S\) we have 
    \begin{align*}
        \sum_{\ell = 1}^{\infty} \frac{\theta_{\ell, j}^2}{\mu_\ell} &\leq 2c^2 \rho^2 \sum_{\ell < \nu_{\mathcal{H}}(s, \mathscr{A}_{\mathcal{H}})} \frac{1}{\mu_\ell} \\
        &= 2c^2 \frac{1}{\nu_{\mathcal{H}}(s, \mathscr{A}_{\mathcal{H}})} \sum_{\ell < \nu_{\mathcal{H}}(s, \mathscr{A}_{\mathcal{H}})} \frac{\psi^2/s}{\mu_\ell} \\
        &\leq 2c^2 \frac{1}{\nu_{\mathcal{H}}(s, \mathscr{A}_{\mathcal{H}})} \sum_{\ell < \nu_{\mathcal{H}}(s, \mathscr{A}_{\mathcal{H}})} \frac{\mu_{\nu_{\mathcal{H}}(s, \mathscr{A}_{\mathcal{H}})}}{\mu_{\ell}} \\
        &\leq 2c^2 \\
        &\leq 1
    \end{align*}
    where the last line follows from \(c < c_\eta\). Note we have used the ordering of the eigenvalues as well as \(\Gamma_{\mathcal{H}}(s, \mathscr{A}_{\mathcal{H}}) \leq \mu_{\nu_{\mathcal{H}}(s, \mathscr{A}_{\mathcal{H}}) - 1}\) by Lemma \ref{lemma:exact_Gamma_ada}. Of course, for \(j \not \in S\) we have \(\Theta_j = 0\). Hence, \(\Theta \in \mathscr{T}_s\) and so \(\pi\) is properly supported. 
        
    Writing \(P_\pi = \int P_\Theta \pi(d\Theta)\) for the mixture, we have 
    \begin{align}
        \inf_{\varphi}\left\{ P_0\left\{\varphi = 1\right\} + \max_{s \geq \sqrt{p\mathscr{A}_{\mathcal{H}}}} \sup_{\substack{\Theta \in \mathscr{T}_s, \\ ||\Theta||_F \geq c \psi_{\text{adapt}}(p, s, n)}} P_\Theta \left\{\varphi = 0\right\} \right\} &\geq 1 - \frac{1}{2}\sqrt{\chi^2(P_\pi \,||\, P_0)}. \label{eqn:adapt_lbound_chisquare}
    \end{align}
    For the following calculations, let \(\Theta, \Theta' \overset{iid}{\sim} \pi\). Let \(v, v' \in \widetilde{\mathscr{V}}_{\mathcal{H}}\) be the corresponding quantities. Let \(s\) and \(t\) denote the corresponding sparsities. Denote the corresponding support sets \(S\) and \(T\). Also, let \(\{r_{i\ell}, \tilde{r}_{i\ell}\}_{1 \leq i \leq p, \ell \in \mathbb{N}}\) denote an iid collection of \(\Rademacher\left(1/2\right)\) random variables which is independent of \(s, t, S,\) and \(T\). For ease of notation, let \(\nu_s = \nu_{\mathcal{H}}(s, \mathscr{A}_{\mathcal{H}})\). Likewise, let \(\nu_t = \nu_{\mathcal{H}}(t, \mathscr{A}_{\mathcal{H}})\). By the Ingster-Suslina method (Proposition \ref{prop:Ingster_Suslina}), we have 
    \begin{align*}
        \chi^2\left(P_\pi \,||\, P_0\right) + 1&= E\left(\exp\left( n \langle \Theta, \Theta'\rangle_F \right)\right) \\
        &= E\left(\exp\left(2nc^2\rho_s\rho_t \sum_{i \in S \cap T} \sum_{1 < \ell < \nu_{s} \wedge \nu_t} r_{i\ell}\tilde{r}_{i\ell}\right)\right) \\
        &= E\left( \prod_{i \in S \cap T} \prod_{1 < \ell < \nu_s \wedge \nu_t} \cosh\left( 2nc^2 \rho_s \rho_t \right) \right) \\
        &\leq E\left( \exp\left(2n^2c^4 \rho_s^2\rho_t^2 (\nu_s \wedge \nu_t)|S \cap T| \right) \right) \\
        &\leq E\left( \exp\left(2n^2c^4 \rho_s^2\rho_t^2 (d_s \wedge d_t)|S \cap T| \right) \right). 
    \end{align*}
    Here, we have used the inequality \(\cosh(x) \leq e^{x^2/2}\) for \(x > 0\). Consider that by Lemma \ref{lemma:exact_Gamma_ada}
    \begin{align*}
        2n^2 c^4 \rho_s^2 \rho_t^2 (d_s \wedge d_t) &= 2n^2c^4 \cdot \frac{\Gamma_{\mathcal{H}}(s, \mathscr{A}_{\mathcal{H}})}{\nu_s} \frac{\Gamma_{\mathcal{H}}(t, \mathscr{A}_{\mathcal{H}})}{\nu_t} \cdot (d_s \wedge d_t) \\
        &\leq 2c^4 \sqrt{\log\left(1 + \frac{p \mathscr{A}_{\mathcal{H}}}{s^2}\right) \log\left(1 + \frac{p \mathscr{A}_{\mathcal{H}}}{t^2}\right)} \cdot \frac{d_s\wedge d_t}{\sqrt{\nu_s \nu_t}} \\
        &\leq 4c^4 \sqrt{\log\left(1 + \frac{p \mathscr{A}_{\mathcal{H}}}{s^2}\right) \log\left(1 + \frac{p \mathscr{A}_{\mathcal{H}}}{t^2}\right)} \cdot \frac{d_s\wedge d_t}{\sqrt{d_s d_t}} \\
        &\leq 4c^4 \sqrt{\frac{p\mathscr{A}_{\mathcal{H}}}{s^2} \cdot \frac{p\mathscr{A}_{\mathcal{H}}}{t^2}} \cdot \frac{d_s \wedge d_t}{\sqrt{d_s d_t}} \\
        &\leq 8c^4 \log\left(1 + \frac{p \mathscr{A}_{\mathcal{H}}}{st}\right) \cdot \frac{d_s\wedge d_t}{\sqrt{d_s d_t}}
    \end{align*}
    We have used \(s, t \geq \sqrt{p\mathscr{A}_{\mathcal{H}}}\) along with the inequality \(\frac{u}{2} \leq \log(1+u) \leq u\) for \(0 \leq u \leq 1\). Therefore, 
    \begin{align}
         E\left(\exp\left(n\langle \Theta, \Theta'\rangle_F\right)\right) \leq E\left(\exp\left( 8c^4 \log\left(1 + \frac{p \mathscr{A}_{\mathcal{H}}}{st}\right) \cdot \frac{d_s\wedge d_t}{\sqrt{d_s d_t}} |S \cap T|\right)\right). \label{eqn:adapt_lbound_preparation}
    \end{align}

    Since \(8c^4 \frac{d_s\wedge d_t}{\sqrt{d_sd_t}} \leq 1\), we can use Lemma \ref{lemma:hypgeom} and the inequality \((1 + x)^\delta \leq 1 + \delta x\) for \(x > 0\) and \(\delta \leq 1\) to obtain 
    \begin{align*}
        \chi^2\left(P_\pi \,||\, P_0\right) + 1 &\leq E\left( \exp\left(8c^4 \frac{d_s \wedge d_t}{\sqrt{d_s d_t}} \log\left(1 + \frac{p\mathscr{A}_{\mathcal{H}}}{st} \right)|S \cap T|\right)\right) \\
        &\leq E\left( \left(1 - \frac{s}{p} + \frac{s}{p} \exp\left(8c^4 \log\left(1 + \frac{p\mathscr{A}_{\mathcal{H}}}{st}\right) \cdot \frac{d_s\wedge d_t}{\sqrt{d_s d_t}}\right)\right)^t\right) \\
        &=  E\left( \left(1 - \frac{s}{p} + \frac{s}{p} \left(1 + \frac{p \mathscr{A}_{\mathcal{H}}}{st}\right)^{8c^4 \frac{d_s\wedge d_t}{\sqrt{d_s d_t}}}\right)^t\right) \\
        &\leq E\left( \left(1 + \frac{1}{t} \cdot 8c^4 \frac{d_s\wedge d_t}{\sqrt{d_s d_t}} \mathscr{A}_{\mathcal{H}}\right)^t\right) \\
        &\leq E\left(\exp\left( 8c^4 \frac{d_s\wedge d_t}{\sqrt{d_s d_t}} \mathscr{A}_{\mathcal{H}}\right)\right). 
    \end{align*}
    Recall we write \(d_s = 2^{k_s}\) and \(d_t = 2^{k_t}\). Moreover, recall \(d_s, d_t \in \widetilde{\mathscr{V}}_{\mathcal{H}}\). Now observe 
    \begin{equation*}
        E\left(\exp\left( 8c^4 \frac{d_s\wedge d_t}{\sqrt{d_s d_t}} \mathscr{A}_{\mathcal{H}}\right)\right) = E\left( \exp\left( 8c^4 2^{-\frac{|k_s-k_t|}{2}} \mathscr{A}_{\mathcal{H}} \right)\right).
    \end{equation*}
    Recall we have \(\mathscr{A}_{\mathcal{H}} \leq L \log(e|\widetilde{\mathscr{V}}_{\mathcal{H}}|)\). Define the sets 
    \begin{align*}
        E_1 &:= \left\{ (k_s, k_t) \in \mathbb{N} \times \mathbb{N} : 2^{k_s}, 2^{k_t} \in \widetilde{\mathscr{V}}_{\mathcal{H}} \text{ and } |k_s-k_t| \leq \frac{\eta^2}{2e^{8Lc^4}} \log_2(e|\widetilde{\mathscr{V}}_{\mathcal{H}}|) \right\}, \\
        E_2 &:= \left\{ (k_s, k_t) \in \mathbb{N} \times \mathbb{N} : 2^{k_s}, 2^{k_t} \in \widetilde{\mathscr{V}}_{\mathcal{H}} \text{ and } |k_s-k_t| > \frac{\eta^2}{2e^{8Lc^4}} \log_2(e|\widetilde{\mathscr{V}}_{\mathcal{H}}|) \right\}.
    \end{align*}
    Then 
    \begin{align*}
        &E\left( \exp\left( 8c^4 2^{-\frac{|k_s-k_t|}{2}} \mathscr{A}_{\mathcal{H}} \right)\right) \\
        &= E\left( \exp\left( 8c^4 2^{-\frac{|k_s-k_t|}{2}} \mathscr{A}_{\mathcal{H}} \right) \mathbbm{1}_{\{(k_s, k_t) \in E_1\}}\right) + E\left( \exp\left( 8c^4 2^{-\frac{|k_s-k_t|}{2}} \mathscr{A}_{\mathcal{H}} \right) \mathbbm{1}_{\{(k_s, k_t) \in E_2\}}\right).
    \end{align*}
    Let us examine the second term. Note that for any \(\ell > 0\) we have \(x^{-\ell} \log(x) \leq (e\ell)^{-1}\) for all \(x > 0\). Using this, we obtain
    \begin{align*}
        E\left( \exp\left( 8c^4 2^{-\frac{|k_s-k_t|}{2}} \mathscr{A}_{\mathcal{H}} \right) \mathbbm{1}_{\{(k_s, k_t) \in E_2\}}\right) &\leq \exp\left(8Lc^4 \left(e|\widetilde{\mathscr{V}}_{\mathcal{H}}|\right)^{-\frac{\eta^2}{4e^{8Lc^4}}} \mathscr{A}_{\mathcal{H}}\right) \\
        &\leq \exp\left(8Lc^4 \left(e|\widetilde{\mathscr{V}}_{\mathcal{H}}|\right)^{-\frac{\eta^2}{4e^{8Lc^4}}} \log(e|\widetilde{\mathscr{V}}_{\mathcal{H}}|)\right) \\
        &\leq \exp\left(8Lc^4 \cdot \frac{4e^{8Lc^4-1}}{\eta^2}\right) \\
        &\leq \exp\left(\frac{32Lc^4}{\eta^2}\right) \\
        &\leq 2\eta^2 + 1.
    \end{align*}
    The final line results from \(c^4 < c_\eta^4 \leq \frac{1}{32L} \eta^2 \log\left(1 + 2\eta^2\right)\). Note we have also used \(8Lc^4 - 1 \leq 0\) to obtain the penultimate line. 
    We now examine \(E_1\). Consider that \(|E_1| \leq \frac{\eta^2}{2e^{8Lc^4}} |\widetilde{\mathscr{V}}_{\mathcal{H}}| \log_2(e|\widetilde{\mathscr{V}}_{\mathcal{H}}|)\). Therefore, 
    \begin{align*}
        E\left( \exp\left( 8c^4 2^{-\frac{|k_s-k_t|}{2}} \mathscr{A}_{\mathcal{H}} \right) \mathbbm{1}_{\{(k_s, k_t) \in E_1\}}\right) &= \frac{1}{|\widetilde{\mathscr{V}}_{\mathcal{H}}|^2} \sum_{v, v' \in \widetilde{\mathscr{V}}_{\mathcal{H}}} \exp\left( 8c^4 2^{-\frac{|k_s-k_t|}{2}} \mathscr{A}_{\mathcal{H}}\right) \mathbbm{1}_{\{(2^{k_s}, 2^{k_t}) \in E_1\}} \\
        &\leq \frac{|E_1|}{|\widetilde{\mathscr{V}}_{\mathcal{H}}|^2} \exp\left(8c^4 \mathscr{A}_{\mathcal{H}}\right) \\
        &\leq \frac{\frac{\eta^2}{2e^{8Lc^4}} |\widetilde{\mathscr{V}}_{\mathcal{H}}| \log_2(e|\widetilde{\mathscr{V}}_{\mathcal{H}}|)}{|\widetilde{\mathscr{V}}_{\mathcal{H}}|^2} \cdot \exp\left(8Lc^4 \log(e|\widetilde{\mathscr{V}}_{\mathcal{H}}|)\right) \\
        &\leq \frac{\eta^2}{2} \cdot \frac{\log_2(e|\widetilde{\mathscr{V}}_{\mathcal{H}}|)}{|\widetilde{\mathscr{V}}_{\mathcal{H}}|^{1-8Lc^4}} \\
        &\leq \eta^2 + \frac{\eta^2}{2 \log 2} \cdot \frac{\log(|\widetilde{\mathscr{V}}_{\mathcal{H}}|)}{|\widetilde{\mathscr{V}}_{\mathcal{H}}|^{1-8Lc^4}} \\
        &\leq \eta^2 + \frac{\eta^2}{2 \log 2} \\
        &\leq 2\eta^2.
    \end{align*}

    Note we have used \(\mathscr{A}_{\mathcal{H}} \leq L\log(e|\widetilde{\mathscr{V}}_{\mathcal{H}}|)\). We have also used \(8Lc^4 < 8Lc_\eta^4 \leq \frac{1}{2}\) along with the inequality \(\log(x) \leq \sqrt{x}\) for all \(x > 0\) to obtain the penultimate line. Therefore, we have shown 
    \begin{equation*}
        \chi^2(P_\pi \,||\, P_0) \leq 2\eta^2 + \left(2\eta^2 + 1\right) - 1 = 4\eta^2. 
    \end{equation*}
    Plugging into (\ref{eqn:adapt_lbound_chisquare}) yields 
    \begin{equation*}
        \inf_{\varphi} \left\{ P_0\left\{\varphi = 1\right\} + \max_{s \geq \sqrt{p\mathscr{A}_{\mathcal{H}}}}\sup_{\substack{\Theta \in \mathscr{T}_s, \\ ||\Theta||_F \geq c \psi_{\text{adapt}}(p, s, n)}} P_{\Theta}\left\{ \varphi = 0 \right\}\right\} \geq 1 - \eta.
    \end{equation*}
    The proof is complete. 
\end{proof}

\subsection{Adaptive upper bound}

\begin{proof}[Proof of Theorem \ref{thm:adapt_ubound}]
    Fix \(\eta \in (0, 1)\). In various places of the proof, we will point out that \(C_\eta\) can be taken sufficiently large to obtain desired bounds, so for now let \(C > C_\eta\). Let \(L^*\) denote the universal constant from Lemma \ref{lemma:bulk_alpha}. Let \(D\) denote the maximum of the corresponding universal constants from Lemmas \ref{lemma:null_bulk} and \ref{lemma:null_tail}. Set 
    \begin{align*}
        K_2 &:= 1 \vee \frac{1}{(\log 2)^{1/4}} \vee \left(\frac{2}{c}\right)^{1/4} \vee \left(\frac{L^*}{c}\right)^{1/4}, \\
        K_3 &:= \frac{L^*}{K_2^2}, \\
        K_2' &:= \frac{1}{\sqrt{\log 2}} \vee \left(\frac{2}{c'}\right)^{1/2} \vee (c'K_3^2)^{-1/4}.
    \end{align*}
    Here, \(c := c^* \wedge c^{**}\) with \(c^*\) and \(c^{**}\) being the universal constants in the exponential terms of Lemmas \ref{lemma:null_bulk} and \ref{lemma:bulk_thold_var} respectively. Likewise, \(c' := (c')^* \wedge (c')^{**}\) where \((c')^*\) and \((c')^{**}\) are the universal constants in the exponential terms of Lemmas \ref{lemma:null_tail} and \ref{lemma:tail_thold_var}. Note that these choices of \(K_2, K_3, K_2'\) are almost identical to the choices in the proofs of Propositions \ref{prop:test_sparse_bulk} and \ref{prop:test_sparse_tail}. The only modifications are the terms \((2/c)^{1/4}\) and \((2/c')^{1/4}\), and the utility of this modification will become clear through the course of the proof. 

    For any \(\nu\), define
    \begin{align*}
        \mathcal{S}_{\text{bulk}} &:= \left\{s \in \mathscr{S} : s < \sqrt{p \mathscr{A}_{\mathcal{H}}} \text{ and } \sqrt{\log\left(1 + \frac{p \mathscr{A}_{\mathcal{H}}}{s^2}\right)} \leq K_3 \sqrt{d_\nu}\right\}, \\
        \mathcal{S}_{\text{tail}} &:= \left\{s \in \mathscr{S} : s < \sqrt{p \mathscr{A}_{\mathcal{H}}} \text{ and } \sqrt{\log\left(1 + \frac{p \mathscr{A}_{\mathcal{H}}}{s^2}\right)} > K_3 \sqrt{d_\nu}\right\}. 
    \end{align*}

    We examine the Type I and Type II errors separately. Focusing on the Type I error, union bound yields
    \begin{align}\label{eqn:adapt_typeI_unionbound}
        &P_0\left\{ \max_{\nu \in \mathscr{V}_{\mathcal{H}}} \max_{s \in \mathscr{S}} \varphi_{\nu, s} = 1\right\} \nonumber \\
        &\leq \left( \sum_{\nu \in \mathscr{V}_{\mathcal{H}}} \sum_{s \in \mathcal{S}_{\text{bulk}}} P_0\left\{\varphi_{\nu, s} = 1\right\}\right) + \left(\sum_{\nu \in \mathscr{V}_{\mathcal{H}}}\sum_{s \in \mathcal{S}_{\text{tail}}} P_0\left\{\varphi_{\nu, s} = 1\right\}\right) + \left(\sum_{\nu \in \mathscr{V}_{\mathcal{H}}} P_0\left\{\varphi_{\nu, p} = 1\right\} \right).
    \end{align}

    We bound each term separately. \newline

    \textbf{Type I error: Bulk}

    For \(s \in \mathcal{S}_{\text{bulk}}\), 
    \begin{align*}
        P_0\left\{\varphi_{\nu, s} = 1\right\} &= P_0\left\{ T_{r_{\nu, s}}(d_\nu) > C s\sqrt{\nu\log\left(1 + \frac{p\mathscr{A}_{\mathcal{H}}}{s^2}\right)} \right\}.
    \end{align*}

    Consider that Lemma \ref{lemma:null_bulk} gives us the following. If we select \(x > 0\) satisfying 
    \begin{equation}\label{cond:bulk_adapt}
        C^*\left(\sqrt{pr_{\nu, s}^4 e^{-c^* \frac{r_{\nu, s}^2}{d}} x } + \frac{d_\nu}{r_{\nu, s}^2} x \right) \leq C s\sqrt{\nu\log\left(1 + \frac{p\mathscr{A}_{\mathcal{H}}}{s^2}\right)},
    \end{equation}
    then we have \(P_0\left\{\varphi_{\nu, s} = 1 \right\} \leq e^{-x}\). Let us select 
    \begin{equation*}
        x = \frac{C}{4(C^* \vee (C^*)^2)} \cdot \frac{1}{(2K_2^4\lceil D\rceil) \vee (\sqrt{2\lceil D \rceil}/K_2^2)}\left(\frac{p \mathscr{A}_{\mathcal{H}}^2}{s^2} \wedge s \log\left(1 + \frac{p\mathscr{A}_{\mathcal{H}}}{s^2}\right)\right). 
    \end{equation*}
    To verify (\ref{cond:bulk_adapt}), consider that \(c^* K_2^4 \geq 2\) by our choice of \(K_2\), and so 
    \begin{align*}
        \sqrt{pr_{\nu, s}^4 e^{-c^* \frac{r_{\nu, s}^2}{d}} x } &= \sqrt{p K_2^4 d_\nu \log\left(1 + \frac{p\mathscr{A}_{\mathcal{H}}}{s^2}\right) \frac{s^4}{(s^2 + p\mathscr{A}_{\mathcal{H}})^2} x} \\
        &\leq \frac{\sqrt{C}}{2C^*} \frac{1}{\sqrt{2\lceil D\rceil}} s \sqrt{d_\nu \log\left(1 + \frac{p\mathscr{A}_{\mathcal{H}}}{s^2}\right)} \\
        &\leq \frac{\sqrt{C}}{2C^*} s\sqrt{\nu \log\left(1 + \frac{p\mathscr{A}_{\mathcal{H}}}{s^2}\right)}.
    \end{align*}
    Here, we have used that \(d_\nu = \nu \vee \lceil D \rceil\). We have also used \(C \geq 1\), which holds provided we select \(C_\eta\) large enough (i.e. \(C_\eta \geq 1\)).  
    
    Likewise, consider 
    \begin{align*}
        \frac{d_\nu}{r_{\nu, s}^2} x &= \frac{d_\nu}{K_2^2 \sqrt{d_\nu \log\left(1 + \frac{p\mathscr{A}_{\mathcal{H}}}{s^2}\right)}} x \\
        &\leq \frac{C}{4\sqrt{2\lceil D \rceil }C^*} s \sqrt{d_\nu \log\left(1 + \frac{p\mathscr{A}_{\mathcal{H}}}{s^2}\right)} \\
        &\leq \frac{C}{4C^*} s \sqrt{\nu \log\left(1 + \frac{p\mathscr{A}_{\mathcal{H}}}{s^2}\right)}.
    \end{align*}
    Therefore, (\ref{cond:bulk_adapt}) is satisfied, and so we have 
    \begin{equation*}
        P_0\left\{\varphi_{\nu, s} = 1\right\} \leq \exp\left(-C \kappa \left( \frac{p \mathscr{A}_{\mathcal{H}}^2}{s^2} \wedge s \log\left(1 + \frac{p\mathscr{A}_{\mathcal{H}}}{s^2}\right)\right)\right)
    \end{equation*}
    where \(\kappa\) is the universal constant \(\kappa = \frac{1}{4(C^* \vee (C^*)^2)} \cdot \frac{1}{(2K_2^4\lceil D\rceil) \vee (\sqrt{2}/K_2^2)}\). With this bound in hand, observe 
    \begin{align*}
        &\sum_{\nu \in \mathscr{V}_{\mathcal{H}}} \sum_{s \in \mathcal{S}_{\text{bulk}}} P_0\left\{\varphi_{\nu, s} = 1\right\} \\
        &\leq \sum_{\nu \in \mathscr{V}_{\mathcal{H}}} \sum_{s \in \mathcal{S}_{\text{bulk}}} \exp\left(- C \kappa \left(\frac{p\mathscr{A}_{\mathcal{H}}^2}{s^2} \wedge s \log\left(1 + \frac{p\mathscr{A}_{\mathcal{H}}}{s^2}\right)\right) \right)\\
        &\leq |\mathscr{V}_{\mathcal{H}}| \sum_{\substack{k \in \mathbb{N}\cup\{0\} : \\ 2^k < \sqrt{p\mathscr{A}_{\mathcal{H}}}}} \exp\left(- C \kappa \left(\frac{p\mathscr{A}_{\mathcal{H}}^2}{2^{2k}} \wedge 2^k \log\left(1 + \frac{p\mathscr{A}_{\mathcal{H}}}{ 2^{2k}}\right)\right) \right) \\
        &\leq e^{2\mathscr{A}_{\mathcal{H}}} \sum_{\substack{k \in \mathbb{N}\cup\{0\} : \\ 2^k < \sqrt{p\mathscr{A}_{\mathcal{H}}}}} \exp\left(- C \kappa \frac{p\mathscr{A}_{\mathcal{H}}^2}{2^{2k}}\right) + e^{2\mathscr{A}_{\mathcal{H}}} \sum_{\substack{k \in \mathbb{N}\cup\{0\} : \\ 2^k < \sqrt{p\mathscr{A}_{\mathcal{H}}}}} \exp\left(-C \kappa 2^{k} \log\left(1 + \frac{p\mathscr{A}_{\mathcal{H}}}{2^{2k}}\right)\right).
    \end{align*}
    Here we have used \(\log(e|\mathscr{V}_{\mathcal{H}}|) \leq 2\mathscr{A}_{\mathcal{H}}\) from Lemma \ref{lemma:script_AV_equiv}. We bound each term separately. First, consider 
    \begin{align*}
        e^{2\mathscr{A}_{\mathcal{H}}} \sum_{\substack{k \in \mathbb{N}\cup\{0\} : \\ 2^k < \sqrt{p\mathscr{A}_{\mathcal{H}}}}} \exp\left(- C \kappa \frac{p\mathscr{A}_{\mathcal{H}}^2}{2^{2k}}\right) &= e^{2\mathscr{A}_{\mathcal{H}}} \sum_{\substack{k \in \mathbb{N} : \\ 2^k < \sqrt{p\mathscr{A}_{\mathcal{H}}}}} \exp\left(-C\kappa \mathscr{A}_{\mathcal{H}} \cdot \left(\frac{\sqrt{p\mathscr{A}_{\mathcal{H}}}}{2^k}\right)^2\right) \\
        &\leq \sum_{k=0}^{\infty} \exp\left( - (C\kappa - 2) \mathscr{A}_{\mathcal{H}} 4^k \right) \\
        &\leq \frac{\eta}{12}
    \end{align*}
    provided we take \(C_\eta\) sufficiently large. 

    Likewise, consider 
    \begin{align*}
        e^{2\mathscr{A}_{\mathcal{H}}} \sum_{\substack{k \in \mathbb{N}\cup\{0\} : \\ 2^k  < \sqrt{p\mathscr{A}_{\mathcal{H}}}}} \exp\left(-C \kappa 2^{k} \log\left(1 + \frac{p\mathscr{A}_{\mathcal{H}}}{2^{2k}}\right)\right) &\leq e^{2\mathscr{A}_{\mathcal{H}}} \sum_{\substack{k \in \mathbb{N}\cup\{0\} : \\ 2^k < \sqrt{p\mathscr{A}_{\mathcal{H}}}}} \exp\left(-\frac{C\kappa}{2} 2^{k} \log\left(1 + \frac{p\mathscr{A}_{\mathcal{H}}}{2^{2k}}\right) - \frac{C\kappa}{2} \mathscr{A}_{\mathcal{H}} \right) \\
        &\leq e^{-\left(\frac{C\kappa}{2} - 2\right)\mathscr{A}_{\mathcal{H}}} \sum_{\substack{k \in \mathbb{N}\cup\{0\} : \\ 2^k < \sqrt{p\mathscr{A}_{\mathcal{H}}}}} \exp\left(-\frac{C\kappa}{2} 2^{k} \log\left(1 + \frac{p\mathscr{A}_{\mathcal{H}}}{2^{2k}}\right)\right) \\
        &\leq \frac{\eta}{12}
    \end{align*}
    again, provided we take \(C_\eta\) sufficiently large. We have used that \(\mathscr{A}_{\mathcal{H}}\) exhibits at most logarithmic growth in \(p\), i.e. we use the crude bound \(\mathscr{A}_{\mathcal{H}} \leq \log(ep)\). Therefore, we have established 
    \begin{equation}\label{eqn:adapt_Sbulk_bound}
        \sum_{\nu \in \mathscr{V}_{\mathcal{H}}} \sum_{s \in \mathcal{S}_{\text{bulk}}} P_0\left\{\varphi_{\nu, s} = 1\right\} \leq \frac{\eta}{6}. 
    \end{equation}

    \textbf{Type I error: Tail} 
    
    For \(s \in \mathcal{S}_{\text{tail}}\), we have 
    \begin{equation*}
        P_0\left\{\varphi_{\nu, s} = 1 \right\} = P_0\left\{ T_{r_s'}(d_\nu) > C s \log\left(1 + \frac{p\mathscr{A}_{\mathcal{H}}}{s^2}\right)\right\}. 
    \end{equation*}
    Consider that Lemma \ref{lemma:null_tail} gives us an analogue to (\ref{cond:bulk_adapt}), namely if we select \(x > 0\) satisfying 
    \begin{equation}\label{cond:tail_adapt}
        C^{**}\left(\sqrt{p (r_s')^4 e^{-c^{**} (r_s')^2} x } + x\right) \leq C n \tau^2_{\text{adapt}}(p, s, n)^2,  
    \end{equation}
    then we have \(P_0\left\{\varphi_{\nu, s} = 1\right\} \leq e^{-x}\). Let us select 
    \begin{equation*}
        x = \frac{C}{4} \cdot \frac{1}{2\left((C^{**}) \vee (C^{**})^2\right) \left((K_2')^4 \vee 1\right)}\left( \frac{p\mathscr{A}_{\mathcal{H}}^2}{s^2} \wedge s\log\left(1 + \frac{p\mathscr{A}_{\mathcal{H}}}{s^2}\right)\right).
    \end{equation*}
    To verify (\ref{cond:tail_adapt}), consider that \(c^{**}(K_2')^2 \geq 2\) by our choice of \(K_2'\), and so 
    \begin{align*}
        \sqrt{p (r_s')^4 e^{-c^{**} (r_s')^2} x } &= (K_2')^2 \log\left(1 + \frac{p\mathscr{A}_{\mathcal{H}}}{s^2}\right) \sqrt{p \frac{s^4}{(s^2 + p\mathscr{A}_{\mathcal{H}})^2} x} \\
        &\leq \sqrt{C} \frac{1}{2C^{**}} \log\left(1 + \frac{p\mathscr{A}_{\mathcal{H}}}{s^2}\right) \sqrt{p \frac{s^4}{(s^2 + p\mathscr{A}_{\mathcal{H}})^2} \cdot \frac{p\mathscr{A}_{\mathcal{H}}^2}{s^2}} \\
        &\leq \frac{C}{2C^{**}} s \log\left(1 + \frac{p\mathscr{A}_{\mathcal{H}}}{s^2}\right).
    \end{align*}
    Here, we have used \(C \geq 1\). Likewise, consider that 
    \begin{equation*}
        x \leq \frac{C}{2C^{**}} \cdot s \log\left(1 + \frac{p\mathscr{A}_{\mathcal{H}}}{s^2}\right). 
    \end{equation*}
    Therefore, (\ref{cond:tail_adapt}) is satisfied, and so we have 
    \begin{equation*}
        P_0\left\{\varphi_{\nu, s} = 1\right\} \leq \exp\left(-C \kappa \left( \frac{p \mathscr{A}_{\mathcal{H}}^2}{s^2} \wedge s \log\left(1 + \frac{p\mathscr{A}_{\mathcal{H}}}{s^2}\right)\right)\right)
    \end{equation*}
    where \(\kappa\) is the universal constant \(\kappa = \frac{1}{4} \cdot \frac{1}{2\left((C^{**}) \vee (C^{**})^2\right)}\). With this bound in hand, observe 
    \begin{align*}
        &\sum_{\nu \in \mathscr{V}_{\mathcal{H}}} \sum_{s \in \mathcal{S}_{\text{tail}}} P_0\left\{\varphi_{\nu, s} = 1\right\} \\
        &\leq \sum_{\nu \in \mathscr{V}_{\mathcal{H}}} \sum_{s \in \mathcal{S}_{\text{tail}}} \exp\left(- C \kappa \left(\frac{p\mathscr{A}_{\mathcal{H}}^2}{s^2} \wedge s \log\left(1 + \frac{p\mathscr{A}_{\mathcal{H}}}{s^2}\right)\right) \right)\\
        &\leq |\mathscr{V}_{\mathcal{H}}| \sum_{\substack{k \in \mathbb{N}\cup\{0\} : \\ 2^k < \sqrt{p\mathscr{A}_{\mathcal{H}}}}} \exp\left(- C \kappa \left(\frac{p\mathscr{A}_{\mathcal{H}}^2}{2^{2k}} \wedge 2^{k} \log\left(1 + \frac{p\mathscr{A}_{\mathcal{H}}}{2^{2k}}\right)\right) \right) \\
        &\leq e^{2\mathscr{A}_{\mathcal{H}}}\sum_{\substack{k \in \mathbb{N}\cup\{0\} : \\ 2^k < \sqrt{p\mathscr{A}_{\mathcal{H}}}}} \exp\left(- C \kappa \frac{p\mathscr{A}_{\mathcal{H}}^2}{2^{2k}}\right) + e^{2\mathscr{A}_{\mathcal{H}}} \sum_{\substack{k \in \mathbb{N}\cup\{0\} : \\ 2^k < \sqrt{p\mathscr{A}_{\mathcal{H}}}}} \exp\left(-C \kappa 2^{k} \log\left(1 + \frac{p\mathscr{A}_{\mathcal{H}}}{ 2^{2k}}\right)\right).
    \end{align*}
    Here we have used \(\log(e|\mathscr{V}_{\mathcal{H}}|) \leq 2 \mathscr{A}_{\mathcal{H}}\) from Lemma \ref{lemma:script_AV_equiv}. From here, the same argument from the bulk case can be employed to conclude 
    \begin{equation}\label{eqn:adapt_Stail_bound}
        \sum_{\nu \in \mathscr{V}_{\mathcal{H}}} \sum_{s \in \mathcal{S}_{\text{tail}}} P_0\left\{\varphi_{\nu, s} = 1\right\} \leq \frac{\eta}{6}
    \end{equation}
    provided \(C_\eta\) is taken sufficiently large. 

    \textbf{Type I error: Dense}

    Let us now bound \(\sum_{\nu \in \mathscr{V}_{\mathcal{H}}} P_0\left\{\varphi_{\nu, p} = 1\right\}\). Consider that for any \(\nu \in \mathscr{V}_{\mathcal{H}}\), we have
    \begin{align*}
        P_0\left\{ \varphi_{\nu, p} = 1\right\} &= P_0\left\{n\sum_{j=1}^{p}\sum_{k \leq \nu} X_{k,j}^2 > C \sqrt{\nu p \mathscr{A}_{\mathcal{H}}}\right\} \\
        &= P\left\{ \chi^2_{\nu p} > C \sqrt{\nu p \mathscr{A}_{\mathcal{H}}}\right\} \\
        &\leq  P\left\{ \chi^2_{\nu p} > \frac{C}{2}\left(\sqrt{\nu p \mathscr{A}_{\mathcal{H}}} + \mathscr{A}_{\mathcal{H}} \right)\right\} \\
        &\leq P\left\{ \chi^2_{\nu p} > 2 \sqrt{\nu p \cdot \frac{C}{4} \mathscr{A}_{\mathcal{H}}} + 2 \cdot \frac{C}{4} \mathscr{A}_{\mathcal{H}}\right\} \\
        &\leq e^{-\frac{C}{4} \mathscr{A}_{\mathcal{H}}}. 
    \end{align*}
    To obtain the fourth line, we have used \(\mathscr{A}_{\mathcal{H}} \leq \log(ep)\) which implies \(\sqrt{\nu p \mathscr{A}_{\mathcal{H}}} \geq \mathscr{A}_{\mathcal{H}}\). In the above display, we have also used \(\frac{C}{2} \geq 1\) (which holds provided we take \(C_\eta\) large enough) to obtain the penultimate line and Lemma \ref{lemma:laurent_massart} to obtain the final line. Consequently, 
    \begin{equation}\label{eqn:adapt_Sdense_bound}
        \sum_{\nu \in \mathscr{V}_{\mathcal{H}}} P_0\left\{\varphi_{\nu, p} = 1\right\} \leq |\mathscr{V}_{\mathcal{H}}| e^{-\frac{C}{4} \mathscr{A}_{\mathcal{H}}} \leq e^{-\left(\frac{C}{4} - 2\right)\mathscr{A}_{\mathcal{H}}} \leq \frac{\eta}{6}
    \end{equation}
    since \(\log(e|\mathscr{V}_{\mathcal{H}}|) \leq 2 \mathscr{A}_{\mathcal{H}}\) by Lemma \ref{lemma:script_AV_equiv} and \(\frac{C}{4} - 2 \geq \log(6/\eta)\), which holds provided we take \(C_\eta\) sufficiently large. 

    Putting together (\ref{eqn:adapt_Sbulk_bound}), (\ref{eqn:adapt_Stail_bound}), (\ref{eqn:adapt_Sdense_bound}) into (\ref{eqn:adapt_typeI_unionbound}) yields the Type I error bound
    \begin{equation}\label{eqn:adapt_typeI}
        P_0\left\{\max_{\nu \in \mathscr{V}_{\mathcal{H}}} \max_{s \in \mathscr{S}} \varphi_{\nu, s} = 1\right\} \leq \frac{\eta}{6} + \frac{\eta}{6} + \frac{\eta}{6} = \frac{\eta}{2}. 
    \end{equation}

    \textbf{Type II error:} We now examine the Type II error. Suppose \(s^* \in [p]\). Let \(f \in \mathcal{F}_{s^*}\) with \(||f||_2 \geq C \tau_{\text{adapt}}(p, s^*, n)\). We proceed by considering various cases. \newline
    
    \textbf{Type II error: Dense} 
    
    Suppose \(s^* \geq \sqrt{p\mathscr{A}_{\mathcal{H}}}\). Let \(\tilde{\nu}\) denote the smallest element in \(\mathscr{V}_{\mathcal{H}}\) which is greater than or equal to \(\nu_{\mathcal{H}}(s^*, \mathscr{A}_{\mathcal{H}})\). Then 
    \begin{align*}
        P_f\left\{\max_{\nu \in \mathscr{V}_{\mathcal{H}}} \max_{s \in \mathscr{S}} \varphi_{\nu, s} = 0\right\} &\leq P_f\left\{\varphi_{\tilde{\nu}, p} = 0\right\} = P_f\left\{ n \sum_{j=1}^{p} \sum_{k \leq \tilde{\nu}} X_{k,j}^2 > \tilde{\nu}p + C \sqrt{\tilde{\nu}p\mathscr{A}_{\mathcal{H}}}\right\}.  
    \end{align*}
    Consider that 
    \begin{equation*}
        n \sum_{j=1}^{p} \sum_{k \leq \tilde{\nu}} X_{k, j}^2 \sim \chi^2_{p \tilde{\nu}} \left( n \sum_{j=1}^{p} \sum_{k \leq \tilde{\nu}} \theta_{k, j}^2\right)
    \end{equation*}
    where the collection \(\{\theta_{k,j}\}\) denotes the collection of basis coefficients for \(f\). We will also use the matrix \(\Theta\) to denote the collection of basis coefficients; note that \(\Theta \in \mathscr{T}_{s^*}\). Let \(S^*\) denote the set of \(j\) for which \(\Theta_j\) are nonzero. Observe 
    \begin{align*}
        C^2 \tau_{\text{adapt}}^2(p, s^*, n) &\leq ||f||_2^2 \\
        &= \sum_{j=1}^{p} \sum_{k=1}^{\infty} \theta_{k,j}^2 \\
        &\leq \sum_{j \in S^*} \left( \sum_{k \leq \tilde{\nu}} \theta_{k,j}^2 + \sum_{k >\tilde{\nu}} \theta_{k,j}^2 \right) \\
        &\leq \sum_{j \in S^*} \left(\sum_{k \leq \tilde{\nu}} \theta_{k,j}^2 + \mu_{\tilde{\nu}} \sum_{k > \tilde{\nu}} \frac{\theta_{k,j}^2}{\mu_k} \right) \\
        &\leq \left(\sum_{j \in S^*} \sum_{k \leq \tilde{\nu}} \theta_{k,j}^2 \right) + s^* \mu_{\tilde{\nu}} \\
        &= \left(\sum_{j = 1}^{p} \sum_{k \leq \tilde{\nu}} \theta_{k,j}^2 \right) + s^* \mu_{\tilde{\nu}}. 
    \end{align*}
    Therefore, 
    \begin{align*}
        n\sum_{j = 1}^{p} \sum_{k \leq \tilde{\nu}} \theta_{k,j}^2 &\geq C^2 n\tau_{\text{adapt}}^2(p, s^*, n) - ns^*\mu_{\tilde{\nu}} \\
        &\geq C^2 n\tau_{\text{adapt}}^2(p, s^*, n)  - s^*\sqrt{\tilde{\nu} \log\left(1 + \frac{p \mathscr{A}_{\mathcal{H}}}{(s^*)^2}\right)} \\
        &\geq C^2 n \tau_{\text{adapt}}^2(p, s^*, n) - \sqrt{2} n \tau_{\text{adapt}}^2(p, s^*, n) \\
        &\geq (C^2 - \sqrt{2}) n \tau_{\text{adapt}}^2(p, s^*, n)
    \end{align*}
    We have used that \(\tilde{\nu} \leq 2 \nu_{\mathcal{H}}(s^*, \mathscr{A}_{\mathcal{H}})\) to obtain the third line. Therefore, we have 
    \begin{equation*}
        n\sum_{j = 1}^{p} \sum_{k \leq \tilde{\nu}} \theta_{k,j}^2 \geq \left(C^2 - \sqrt{2}\right) n\tau_{\text{adapt}}^2(p, s^*, n) \geq \frac{C^2 - \sqrt{2}}{\sqrt{2}} \sqrt{\tilde{\nu} p \mathscr{A}_{\mathcal{H}}}. 
    \end{equation*}
    The signal magnitude satisfies the requisite strength needed to successfully detect, which can be seen by following the argument of Proposition \ref{prop:dense_alpha}. Hence, we can achieve Type II error less than or equal to \(\frac{\eta}{2}\) by selecting \(C_\eta\) sufficiently large. We can combine this bound with (\ref{eqn:adapt_typeI_unionbound}) to conclude that the total testing risk is bounded above by \(\eta\), as desired. \newline
    
    \textbf{Type II error: Bulk} 
    
    Suppose \(s^* < \sqrt{p \mathscr{A}_{\mathcal{H}}}\) and \(\sqrt{\log\left(1 + \frac{p\mathscr{A}_{\mathcal{H}}}{(s^*)^2}\right)} \leq K_3 \sqrt{\nu_{\mathcal{H}}(s^*, \mathscr{A}_{\mathcal{H}}) \vee \lceil D \rceil}\). Let \(\tilde{\nu}\) be the smallest element in \(\mathscr{V}_{\mathcal{H}}\) greater than or equal to \(\nu_{\mathcal{H}}(s^*, \mathscr{A}_{\mathcal{H}})\). Let \(\tilde{s}\) be the smallest element in \(\mathscr{S}\) greater than or equal to \(s^*\). By the definitions of these grids, we have \(\tilde{\nu}/2 < \nu_{\mathcal{H}}(s^*, \mathscr{A}_{\mathcal{H}}) \leq \tilde{\nu}\) and \(s^* \leq \tilde{s} \leq 2s^*\). Then 
    \begin{equation*}
        P_f\left\{\max_{\nu \in \mathscr{V}_{\mathcal{H}}} \max_{s \in \mathscr{S}} \varphi_{\nu, s} = 0\right\} \leq P_f\left\{\varphi_{\tilde{\nu}, \tilde{s}} = 0\right\}. 
    \end{equation*}
    With these choices, we have
    \begin{equation*}
        \sqrt{\log\left(1 + \frac{p\mathscr{A}_{\mathcal{H}}}{\tilde{s}^2}\right)} \leq K_3 \sqrt{d_{\tilde{\nu}}}.    
    \end{equation*}
    Note that since \(\tilde{s} \geq s^*\), we have \(f \in \mathcal{F}_{\tilde{s}}\). Following argument similar to those in the proof of Proposition \ref{prop:test_sparse_bulk}, it can be seen that the necessary signal strength to successfully detect is of squared order 
    \begin{equation*}
        \frac{\tilde{s} \sqrt{\tilde{\nu} \log\left(1 + \frac{p\mathscr{A}_{\mathcal{H}}}{\tilde{s}^2}\right)}}{n} \asymp \frac{s^*\sqrt{\nu_{\mathcal{H}}(s^*, \mathscr{A}_{\mathcal{H}}) \log\left(1 + \frac{p\mathscr{A}_{\mathcal{H}}}{(s^*)^2}\right)}}{n}.
    \end{equation*}
    This is precisely the order of \(\tau^2_{\text{adapt}}(p, s^*, n)\), and so we can obtain Type II error less than \(\frac{\eta}{2}\) by choosing \(C_\eta\) sufficiently large. We can combine this bound with (\ref{eqn:adapt_typeI_unionbound}) to conclude that the total testing risk is bounded above by \(\eta\), as desired.
    \newline

    \textbf{Type II error: Tail} 
    
    Suppose \(s^* < \sqrt{p \mathscr{A}_{\mathcal{H}}}\) and \(\sqrt{\log\left(1 + \frac{p\mathscr{A}_{\mathcal{H}}}{(s^*)^2}\right)} > K_3 \sqrt{\nu_{\mathcal{H}}(s^*, \mathscr{A}_{\mathcal{H}}) \vee \lceil D \rceil}\). Let \(\tilde{\nu}\) and \(\tilde{s}\) be as defined in the previous case, i.e. Type II error analysis for the bulk case. As before, we have 
    \begin{equation*}
        P_f\left\{\max_{\nu \in \mathscr{V}_{\mathcal{H}}} \max_{s \in \mathscr{S}} \varphi_{\nu, s} = 0\right\} \leq P_f\left\{\varphi_{\tilde{\nu}, \tilde{s}} = 0\right\}. 
    \end{equation*}
    Now, consider that \(s^* \leq \tilde{s}\) and so \(f \in \mathcal{F}_{\tilde{s}}\).
    
    Suppose we have 
    \begin{equation*}
        \sqrt{\log\left(1 + \frac{p\mathscr{A}_{\mathcal{H}}}{\tilde{s}^2}\right)} > K_3 \sqrt{d_{\tilde{\nu}}}.
    \end{equation*}
    We can follow arguments similar to those in the proof of Proposition \ref{prop:test_sparse_tail} to see that the necessary signal strength to successfully detect is of squared order 
    \begin{equation*}
        \frac{\tilde{s}\log\left(1 + \frac{p\mathscr{A}_{\mathcal{H}}}{\tilde{s}^2}\right)}{n} \asymp \frac{s^* \log\left(1 + \frac{p\mathscr{A}_{\mathcal{H}}}{(s^*)^2}\right)}{n}.
    \end{equation*}
    This is precisely the order of \(\tau^2_{\text{adapt}}(p, s^*, n)\) and so we can obtain Type II error less than \(\frac{\eta}{2}\) by choosing \(C_\eta\) sufficiently large. 

    On the other hand, suppose we have 
    \begin{equation*}
        \sqrt{\log\left(1 + \frac{p\mathscr{A}_{\mathcal{H}}}{\tilde{s}^2}\right)} \leq K_3 \sqrt{d_{\tilde{\nu}}}
    \end{equation*}
    As argued in the previous section about the bulk regime, the necessary signal strength to successfully detect is of squared order
    \begin{equation*}
        \frac{\tilde{s} \sqrt{\tilde{\nu}\log\left(1 + \frac{p\mathscr{A}_{\mathcal{H}}}{\tilde{s}^2}\right)}}{n} \asymp \frac{s^* \sqrt{\nu_{\mathcal{H}}(s^*, \mathscr{A}_{\mathcal{H}}) \log\left(1 + \frac{p\mathscr{A}_{\mathcal{H}}}{(s^*)^2}\right)}}{n}.
    \end{equation*}
    Consider that 
    \begin{align*}
        \tau^2_{\text{adapt}}(p, s^*, n) = \frac{s^*\log\left(1 + \frac{p\mathscr{A}_{\mathcal{H}}}{(s^*)^2}\right)}{n} \gtrsim \frac{s^*\sqrt{\nu_{\mathcal{H}}(s^*, \mathscr{A}_{\mathcal{H}}) \log\left(1 + \frac{p\mathscr{A}_{\mathcal{H}}}{(s^*)^2}\right)}}{n}
    \end{align*}
    and so we can obtain Type II error less than \(\frac{\eta}{2}\) by choosing \(C_\eta\) sufficiently large. We can combine this bound with (\ref{eqn:adapt_typeI_unionbound}) to conclude that the total testing risk is bounded above by \(\eta\), as desired.
\end{proof}

\subsection{Smoothness and sparsity adaptive lower bounds}

\begin{proof}[Proof of Theorem \ref{thm:sobolev_adapt_lbound_dense}]
    Fix \(\eta \in (0, 1)\). Fix any \(s \geq p^{1/2+\delta}\sqrt{\log\log n}\). Through the course of the proof, we will note where we must take \(c_\eta\) suitably small enough, so for now let \(0 < c < c_\eta\). For each \(s\), define the geometric grid 
    \begin{equation*}
        \mathcal{V}_s := \left\{2^k : k \in \mathbb{N} \text{ and } \left(\frac{ns}{\sqrt{p\log\log(np)}}\right)^{\frac{2}{4\alpha_1+1}} \leq 2^k \leq \left(\frac{ns}{\sqrt{p\log\log(np)}}\right)^{\frac{2}{4\alpha_0+1}}\right\}.
    \end{equation*}
    Note \(\log|\mathcal{V}_s| \asymp \log\log(np)\). 
    
    We now define a prior \(\pi\) which is supported on the alternative hypothesis. Let \(\nu \sim \text{Uniform}(\mathcal{V}_s)\). Let \(\alpha = \alpha(\nu, s)\) denote the solution to 
    \begin{equation*}
        \nu = \left(\frac{ns}{\sqrt{p \log\log(np)}}\right)^{\frac{2}{4\alpha+1}}.
    \end{equation*}
    Note that \(\alpha(\nu, s) \in [\alpha_0, \alpha_1]\) for any \(\nu \in \mathcal{V}_s\). Note \(\alpha\) is random since \(\nu\) is random. Draw uniformly at random a subset \(S \subset [p]\) of size \(s\). Draw independently 
    \begin{equation*}
        \theta_{k, j} \sim 
        \begin{cases}
            \Uniform\left\{-c\rho_{\nu}, c\rho_{\nu} \right\} &\textit{if } j \in S \text{ and } k \leq \nu, \\
            \delta_0 &\textit{otherwise}.
        \end{cases}
    \end{equation*}
    Here, \(\rho_{\nu}\) is given by \(\rho_{\nu} := \left(\frac{ns}{\sqrt{p \log\log(np)}}\right)^{-\frac{2\alpha+1}{4\alpha+1}}\). Having defined \(\rho_{\nu}\), the definition of the prior \(\pi\) is complete.  

    Now we must show \(\pi\) is indeed supported on the alternative hypothesis. Consider that for \(\Theta \sim \pi\), we have
    \begin{align*}
        ||\Theta||_F^2 = c^2 s \rho_{\nu}^2 \nu = c^2 s \left(\frac{ns}{\sqrt{p\log\log(np)}}\right)^{-\frac{4\alpha}{4\alpha+1}} = c^2 \tau_{\text{dense}}^2(p, s, n, \alpha). 
    \end{align*}
    Furthermore, consider that for \(i \in S\) we have by definition of \(\alpha = \alpha(\nu, s)\), 
    \begin{equation*}
        \sum_{\ell = 1}^{\infty} \frac{\theta_{i, \ell}^2}{\ell^{-2\alpha}} = c^2 \rho_{\nu}^2 \nu^{2\alpha} \sum_{\ell \leq \nu} \frac{\ell^{2\alpha}}{\nu^{2\alpha}} \leq c^2 \rho_{\nu}^2 \nu^{2\alpha+1}  = c^2 \rho_{\nu}^2 \left(\frac{ns}{\sqrt{p\log\log(np)}}\right)^{\frac{4\alpha+2}{4\alpha+1}} \leq c^2 \leq 1.
    \end{equation*}
    We can take \(c_\eta \leq 1\) to ensure \(c^2 \leq 1\). Of course, for \(i \in S^c\) we have \(\Theta_i = 0\). Hence, we have shown \(\Theta \in \mathscr{T}(s, \alpha)\) with probability one, and so \(\pi\) has the proper support. 

    Writing \(P_\pi = \int P_\Theta \pi(d\Theta)\) to denote the mixture induced by the prior \(\pi\), we have
    \begin{align*}
        \inf_{\varphi}\left\{ P_0\left\{\varphi \neq 0\right\} + \max_{\tilde{s} \geq p^{1/2+\delta}} \sup_{\alpha \in [\alpha_0, \alpha_1]} \sup_{\substack{f \in \mathcal{F}(\tilde{s}, \alpha), \\ ||f||_2 \geq c\tau_{\text{dense}}(p, \tilde{s}, n, \alpha)}} P_f\left\{\varphi \neq 1\right\} \right\} &\geq 1 - \frac{1}{2} \sqrt{\chi^2(P_\pi \,||\, P_0)}.
    \end{align*}
    For the following calculations, let \(\Theta, \Theta' \overset{iid}{\sim} \pi\). Let \(\nu, \nu' \in \mathcal{V}_s\) be the corresponding quantities and let \(\alpha, \alpha'\) denote the corresponding smoothness levels. Further let \(S, S'\) denote the corresponding support sets. Note both are of size \(s\). Let \(\{r_{i\ell},r'_{i\ell}\}_{1 \leq i \leq p, \ell \in \mathbb{N}}\) denote an iid collection of \(\Rademacher(1/2)\) random variables which is independent of all the other random variables. By the Ingster-Suslina method (Proposition \ref{prop:Ingster_Suslina}), we have 
    \begin{align*}
        \chi^2(P_\pi \,||\, P_0) + 1 &= E\left(\exp\left(n \langle \Theta, \Theta'\rangle_F\right)\right) \\
        &= E\left(\exp\left(c^2 n \rho_{\nu}\rho_{\nu'} \sum_{i \in S \cap S'} \sum_{\ell=1}^{\nu \wedge \nu'} r_{i\ell}r_{i\ell}'\right)\right) \\
        &\leq E\left(\exp\left(\frac{c^4}{2} n^2 \rho_{\nu}^2\rho_{\nu'}^2 (\nu \wedge \nu') |S \cap S'|\right)\right) \\
        &\leq E\left(\exp\left(\frac{c^4}{2} n^2 \sqrt{\rho_{\nu}^4\nu \cdot \rho_{\nu'}^4 \nu'} \cdot \frac{\nu \wedge \nu'}{\sqrt{\nu \nu'}} |S \cap S'|\right)\right).
    \end{align*}
    From the definition of \(\alpha(\nu, s)\), we have  
    \begin{align*}
        \sqrt{\rho_\nu^4 \nu \cdot \rho_{\nu'}^4 \nu'} &= \left(\left(\frac{ns}{\sqrt{p \log\log(np)}}\right)^{-\frac{8\alpha+4}{4\alpha+1} + \frac{2}{4\alpha+1}} \left(\frac{ns}{\sqrt{p \log\log(np)}}\right)^{-\frac{8\alpha'+4}{4\alpha'+1} + \frac{2}{4\alpha'+1}} \right)^{1/2} \\
        &= \left(\frac{\sqrt{p \log\log(np)}}{ns}\right)^2 \\
        &\leq \frac{2\log\left(1 + \frac{p\log\log(np)}{s^2}\right)}{n^2}. 
    \end{align*}
    Here, we have used \(s \geq \sqrt{p\log\log(np)}\) since \(s \geq p^{1/2+\delta}\) and \(s \geq \sqrt{p \log\log n}\). We have also used the inequality \(x/2 \leq \log(1+x)\) for \(0 \leq x \leq 1\). With this in hand, it follows that 
    \begin{align*}
        \chi^2(P_\pi \,||\, P_0) + 1 &\leq E\left(\exp\left(c^4 \cdot \log\left(1 + \frac{p \log\log(np)}{s^2}\right) \cdot \frac{\nu \wedge \nu'}{\sqrt{\nu \nu'}} |S \cap S'|\right)\right) \\
        &\leq E\left(\left(1 - \frac{s}{p} + \frac{s}{p} e^{c^4\left(\frac{\nu \wedge \nu'}{\sqrt{\nu \nu'}}\right)\log\left(1 + \frac{p\log\log(np)}{s^2}\right)}\right)^s\right) \\
        &\leq E\left( \exp\left( c^4 \frac{\nu \wedge \nu'}{\sqrt{\nu \nu'}} \log\log(np) \right)\right). 
    \end{align*}
    Noting that we can write \(\nu = 2^{k}\), \(\nu' = 2^{k'}\) and that \(|\mathcal{V}_s| \asymp \log(np)\), we can follow the same steps as in the proof of Proposition 4.2 in \cite{gao_estimation_2020}. Taking \(c_\eta\) suitably small, we obtain \(\chi^2(P_\pi \,||\, P_0) \leq 4\eta^2\) which yields 
    \begin{equation*}
        \inf_{\varphi}\left\{ P_0\left\{\varphi \neq 0\right\} + \max_{\tilde{s} \geq p^{1/2+\delta}} \sup_{\alpha \in [\alpha_0, \alpha_1]} \sup_{\substack{f \in \mathcal{F}(\tilde{s}, \alpha), \\ ||f||_2 \geq c\tau_{\text{adapt}}(p, \tilde{s}, n, \alpha)}} P_f\left\{\varphi \neq 1\right\} \right\} \geq 1 - \eta 
    \end{equation*}
    as desired. 
\end{proof}

\begin{proof}[Proof of Theorem \ref{thm:sobolev_adapt_lbound_sparse}]
    Some simplification is convenient. Note \(\tau_{\text{sparse}}\) can be rewritten as 
    \begin{equation*}
    \tau^2_{\text{sparse}}(p, s, n, \alpha) \asymp 
    \begin{cases}
        s \left(\frac{n}{\sqrt{\log(p\log\log n)}}\right)^{-\frac{4\alpha}{4\alpha+1}} &\textit{if } \log(p\log\log n) \leq n^{\frac{1}{2\alpha+1}}, \\
        \frac{s\log(p \log\log n)}{n} &\textit{if } \log(p\log\log n) > n^{\frac{1}{2\alpha+1}}.
    \end{cases}
    \end{equation*}
    Note that in the case \(\log(p \log\log n) > n^{\frac{1}{2\alpha+1}}\), we have \(\log p \gtrsim n^{\frac{1}{2\alpha+1}}\). Therefore, \(\log(p \log \log n) \asymp \log p\) and so we can further simplify 
    \begin{equation}\label{rate:Sobolev_adapt_sparse_simple}
        \tau^2_{\text{sparse}}(p, s, n, \alpha) \asymp 
        \begin{cases}
            s \left(\frac{n}{\sqrt{\log(p\log\log n)}}\right)^{-\frac{4\alpha}{4\alpha+1}} &\textit{if } \log(p\log\log n) \leq n^{\frac{1}{2\alpha+1}}, \\
            \frac{s\log p}{n} &\textit{if } \log(p\log\log n) > n^{\frac{1}{2\alpha+1}}.
        \end{cases}
    \end{equation}
    Writing \(\tau_{\text{sparse}}\) in the form (\ref{rate:Sobolev_adapt_sparse_simple}) is convenient. The lower bound \(\frac{s \log p}{n}\) is exactly the minimax lower bound and so no new argument is needed. All that needs to be proved is the lower bound \(s \left(n / \sqrt{\log(p\log\log n)}\right)^{-\frac{4\alpha}{4\alpha+1}}\). 

    The proof is very similar to that of the proof of Theorem \ref{thm:sobolev_adapt_lbound_dense}, so we only point out the modifications in the interest of brevity. Fix any \(s < p^{1/2-\delta}\). Let \(\pi\) be the prior from the proof of Theorem \ref{thm:sobolev_adapt_lbound_dense}, but use \(\mathcal{V}_{\text{sparse}}\) given by (\ref{def:V_sobolev_sparse}) and \(\alpha(\nu)\) given by (\ref{def:alpha_nu}), defined below, instead. Define the geometric grid 
    \begin{equation}\label{def:V_sobolev_sparse}
        \mathcal{V}_{\text{sparse}} := \left\{2^k : k \in \mathbb{N} \text{ and } \left(\frac{n}{\sqrt{\log(p\log\log n)}}\right)^{\frac{2}{4\alpha_1+1}} \leq 2^k \leq \left(\frac{n}{\sqrt{\log(p\log\log n)}}\right)^{\frac{2}{4\alpha_0+1}}\right\}
    \end{equation}
    Let \(\alpha(\nu)\) denote the solution to 
    \begin{equation}\label{def:alpha_nu}
        \nu = \left(\frac{n}{\sqrt{\log(p\log\log n)}}\right)^{\frac{2}{4\alpha+1}}.
    \end{equation}
    Note \(\alpha(\nu) \in [\alpha_0, \alpha_1]\) for any \(\nu \in \mathcal{V}_{\text{sparse}}\). Also, use 
    \begin{equation*}
        \rho_{\nu} = \left(\frac{n}{\sqrt{\log(p\log\log n)}}\right)^{-\frac{2\alpha+1}{4\alpha+1}}. 
    \end{equation*}
    It can be checked in the same manner that \(\pi\) with these modifications is properly supported on the alternative hypothesis. We can continue along as in the proof of Theorem \ref{thm:sobolev_adapt_lbound_dense} up until the point we have 
    \begin{equation*}
        \sqrt{\rho_\nu^4 \nu \cdot \rho_{\nu'}^4 \nu'} = \left(\frac{\sqrt{\log(p\log\log n)}}{n}\right)^2.
    \end{equation*}
    From here, we use the fact that \(s \leq p^{\frac{1}{2} - \delta}\) implies \(\log(p \log\log n) \asymp \log\left(1 + \frac{p\log\log n}{s^2}\right)\). In other words, there exists a constant \(\kappa > 0\) such that 
    \begin{equation*}
        \sqrt{\rho_\nu^4 \nu \cdot \rho_{\nu'}^4 \nu'} \leq \kappa \frac{\log\left(1 + \frac{p\log\log n}{s^2}\right)}{n^2}.
    \end{equation*}
    Then by taking \(c_\eta\) sufficiently small, the rest of the proof of Theorem \ref{thm:sobolev_adapt_lbound_dense} can be carried out to obtain the desired result.  
\end{proof}

\subsection{Smoothness and sparsity adaptive upper bounds}

\begin{proof}[Proof of Theorem \ref{thm:sobolev_adapt_ubound_dense}]
    Fix \(\eta \in (0, 1)\). For ease, let us just write \(\mathcal{V} = \mathcal{V}_{\text{test}}\). We will note throughout the proof where we take \(C_\eta\) suitably large, so for now let \(C > C_\eta\). We will also note where we take \(K_\eta\) suitably large. We will also note where we take \(K_\eta\) suitably large. We first examine the Type I error. By union bound and taking \(K_\eta \geq 1\), 
    \begin{align*}
        P_0\left\{\max_{\nu \in \mathcal{V}} \varphi_{\nu} = 1\right\} &\leq \sum_{\nu \in \mathcal{V}} P_0\left\{\varphi_{\nu} = 1\right\} \\
        &= \sum_{\nu \in \mathcal{V}} P\left\{\chi^2_{\nu p} \geq \nu p + K_\eta \left(\sqrt{\nu p \log\log(np)} + \log\log(np)\right)\right\} \\
        &\leq \sum_{\nu \in \mathcal{V}} P\left\{\chi^2_{\nu p} \geq \nu p + 2\left(\sqrt{\nu p \frac{K_\eta}{4}\log\log(np)} + \frac{K_\eta}{4}\log\log(np)\right)\right\} \\
        &\leq |\mathcal{V}| e^{-\frac{K_\eta}{4} \log\log(np)} \\
        &\leq e^{-\left(\frac{K_\eta}{4} - \kappa\right) \log\log(np)}
    \end{align*}
    for some universal positive constant \(\kappa\). Here, we have used \(|\mathcal{V}| \asymp \log(np)\). Taking \(K_\eta\) suitably large, the Type I error is suitably bounded 
    \begin{equation*}
        P_0\left\{\max_{\nu \in \mathcal{V}} \varphi_{\nu} = 1\right\} \leq \frac{\eta}{2}. 
    \end{equation*}
    We now examine the Type II error. Fix any \(s^* \geq p^{1/2+\delta}\), \(\alpha^* \in [\alpha_0, \alpha_1]\), and \(f^* \in \mathcal{F}(s^*, \alpha^*)\) with \(||f^*||_2 \geq C \tau_{\text{dense}}(p,s^*, n, \alpha^*)\). Set 
    \begin{equation*}
        \nu^* = \left(\frac{ns^*}{\sqrt{p \log\log(np)}}\right)^{\frac{2}{4\alpha^*+1}}. 
    \end{equation*}

    Let \(\nu \in \mathcal{V}\) be the smallest element larger than \(\nu^*\). Note \(\nu/2 \leq \nu^* \leq \nu\) by definition of \(\mathcal{V}\). 
    \begin{align*}
        &P_{f^*}\left\{ \max_{\tilde{\nu} \in \mathcal{V}} \varphi_{\tilde{\nu}} = 0 \right\} \\
        &\leq P_{f^*}\left\{ \varphi_{\nu} = 0\right\} \\
        &= P_{f^*}\left\{\chi^2_{\nu p}(n||f^*||_2^2) \leq \sqrt{\nu p} + K_\eta\left(\sqrt{p\nu \log\log(np)} + \log\log(np)\right)\right\}.
    \end{align*}
    Consider that since \(\nu \geq \left(\frac{ns^*}{\sqrt{p \log\log(np)}}\right)^{\frac{2}{4\alpha^*+1}}\) and \(s^* \geq p^{1/2+\delta}\), it immediately follows that \(\nu\) grows polynomially in \(n\), and so \(\nu \gtrsim \log\log n\). Therefore, we have \(\sqrt{\nu p \log\log(np)} \geq \frac{1}{\kappa'} \log\log(np)\) for some universal positive constant \(\kappa'\). 
    Then by Chebyshev's inequality
    \begin{align*}
        &P_{f^*}\left\{\chi^2_{\nu p}(n||f^*||_2^2) \leq \nu p + K_\eta\left(\sqrt{p\nu \log\log(np)} + \log\log(np)\right)\right\} \\
        &\leq P_{f^*}\left\{\chi^2_{\nu p}(n||f^*||_2^2) \leq \nu p + K_\eta(1+\kappa')\left(\sqrt{p\nu \log\log(np)} + \log\log(np)\right)\right\} \\
        &= P_{f^*}\left\{(n||f^*||_2^2 - K_\eta(1+\kappa')\sqrt{p\nu\log\log(np)}) \leq \nu p + n||f^*||_2^2 - \chi^2_{\nu p}(n||f^*||_2^2)\right\} \\
        &\leq \frac{\Var\left(\chi^2_{\nu p}(n||f^*||_2^2)\right)}{\left(n||f^*||_2^2 - K_\eta(1+\kappa')\sqrt{p\nu \log\log(np)} \right)^2} \\
        &\leq \frac{2\nu p}{\left(n||f^*||_2^2 - K_\eta(1+\kappa')\sqrt{p\nu \log\log(np)} \right)^2} + \frac{4n||f^*||_2^2}{\left(n||f^*||_2^2 - K_\eta(1+\kappa')\sqrt{p\nu \log\log(np)} \right)^2}.
    \end{align*}
    Consider that 
    \begin{align*}
        K_\eta(1+\kappa')\sqrt{p\nu \log\log(np)} &\leq K_\eta (1+\kappa') \sqrt{2\nu^*p \log\log(np)}\\
        &\leq K_\eta(1+\kappa')\sqrt{2} ns^* \left(\frac{ns^*}{\sqrt{p \log\log(np)}}\right)^{-\frac{4\alpha^*}{4\alpha^*+1}} \\
        &= \sqrt{2}K_\eta(1+\kappa') n \tau_{\text{dense}}^2(p, s^*, n, \alpha^*) \\
        &\leq \frac{\sqrt{2}K_\eta(1+\kappa')}{C^2} n||f^*||_2^2. 
    \end{align*}
    Therefore, taking \(C_\eta\) sufficiently large, we have
    \begin{align*}
        &P_{f^*}\left\{ \max_{\tilde{\nu} \in \mathcal{V}} \varphi_{\tilde{\nu}} = 0 \right\} \\
        &\leq \frac{2\nu p}{\left(n||f^*||_2^2 - K_\eta(1+\kappa')\sqrt{p\nu \log\log(np)} \right)^2} + \frac{4n||f^*||_2^2}{\left(n||f^*||_2^2 - K_\eta(1+\kappa')\sqrt{p\nu \log\log(np)} \right)^2} \\
        &\leq \frac{2}{\left(\frac{C^2}{\sqrt{2}} - K_\eta(1+\kappa')\right)^2 \log\log(np)} + \frac{4}{(1 - \frac{\sqrt{2}K_\eta(1+\kappa')}{C^2}) n||f^*||_2^2} \\
        &\leq \frac{2}{\left(\frac{C^2}{\sqrt{2}} - K_\eta(1+\kappa')\right)^2 \log\log(np)} + \frac{4}{(1 - \frac{\sqrt{2}K_\eta(1+\kappa')}{C^2}) \frac{C^2}{\sqrt{2}}} \\
        &\leq \frac{\eta}{2}. 
    \end{align*}
    Hence, the Type II error is bounded suitably. Since \(f^*\) was arbitrary, we have shown that 
    \begin{align*}
        P_0\left\{\max_{\nu \in \mathcal{V}} \varphi_{\nu} = 1\right\} + \max_{s \geq p^{1/2+\delta}} \sup_{\alpha \in [\alpha_0, \alpha_1]} \sup_{\substack{f \in \mathcal{F}(s, \alpha), \\ ||f||_2 \geq C \tau_{\text{dense}}(p, s, n, \alpha)}} P_f\left\{ \max_{\nu \in \mathcal{V}} \varphi_\nu = 0 \right\} \leq \eta
    \end{align*}
    as desired. 

\end{proof}

\printbibliography

\appendix

\section{Hard thresholding}
In this section, we collect results about the random variable \((||Z||^2 - \alpha_t(d))\mathbbm{1}_{\{||Z||^2 \geq d + t^2\}}\) for \(Z \sim N(\theta, I_d)\) and \(\alpha_t(d)\) defined in (\ref{def:alpha}). We consider the tail \(t^2 \gtrsim d\) and the bulk \(t^2 \lesssim d\) separately. 

The proof outlines are similar to those employed in \cite{liu_minimax_2021}. However, they only consider \(d = 1\), whereas we need to consider the general \(d\) case. Consequently, a much more careful analysis is required. 

\subsection{Tail}

\begin{lemma}[Moment generating function]\label{lemma:tail_mgf}
    Suppose \(\widetilde{c}\) is a positive universal constant. There exist universal positive constants \(D, \widetilde{C}, C,\) and \(c\) such that if \(d \geq D, t^2 \geq \widetilde{c} d,\) and \(0 < \lambda < \frac{1}{2\widetilde{C}}\), then we have 
    \begin{equation*}
        E\left(e^{\lambda Y}\right) \leq \exp\left(Ct^4 \lambda^2 e^{-ct^2}\right)
    \end{equation*}
    where \(Y = \left(||Z||^2 - \alpha_{t}(d)\right) \mathbbm{1}_{\left\{ ||Z||^2 \geq d + t^2 \right\}}\), \(Z \sim N(0, I_d)\), and where \(\alpha_{t}(d)\) is given by (\ref{def:alpha}).
\end{lemma}
\begin{proof}
    We follow the broad approach of the argument presented in the proof of Lemma 18 in \cite{liu_minimax_2021}. The universal positive constant \(D\) will be chosen later on in the proof, so for now let \(d \geq D\) and \(Z \sim N(0, I_d)\). Note we will also take \(D\) large enough so that Lemma \ref{lemma:bulk_alpha} is applicable. Since \(E(Y) = 0\), we have 
    \begin{equation*}
        E(e^{\lambda Y}) = E(1 + \lambda Y + (e^{\lambda Y} - 1 - \lambda Y)) = 1 + E(e^{\lambda Y} - 1 - \lambda Y).
    \end{equation*}
    Consider that for any \(y \in \R\) we have 
    \begin{equation*}
        e^{y} - 1 - y \leq 
        \begin{cases}
            (-y) \wedge y^2 &\textit{if } y < 0, \\
            y^2 &\textit{if } 0 \leq y \leq 1, \\
            e^y &\textit{if } y > 1. 
        \end{cases}
    \end{equation*}
    Therefore, 
    \begin{equation}\label{eqn:low_dim_mgf}
        E(e^{\lambda Y} - 1 - \lambda Y) \leq \lambda^2 E(Y^2 \mathbbm{1}_{\{Y < 0\}}) + \lambda^2 E(Y^2 \mathbbm{1}_{\{0 \leq Y \leq 1/\lambda\}}) + E(e^{\lambda Y}\mathbbm{1}_{\{Y > 1/\lambda\}}).
    \end{equation}
    Each term will be bounded separately. Considering the first term, note that Lemma \ref{lemma:conditional_2} asserts \(\alpha_{t}(d) \leq d + C^*(t^2 \vee d)\) for a universal constant \(C^*\). Note also that \(\alpha_{t}(d) \geq d + t^2\). Therefore, 
    \begin{align*}
        \lambda^2 E(Y^2\mathbbm{1}_{\{Y < 0\}}) &= \lambda^2 E((||Z||^2 - \alpha_{t}(d))^2 \mathbbm{1}_{\{\alpha_t(d) > ||Z||^2 \geq d + t^2\}}) \\
        &\leq \lambda^2 (\alpha_{t}(d) - d - t^2)^2 P\left\{\alpha_{t}(d) > ||Z||^2 \geq d + t^2\right\} \\
        &\leq \lambda^2(d + C^*(t^2 \vee d) - d - t^2)^2 P\left\{||Z||^2 \geq d + t^2\right\} \\
        &\leq 2C^*\lambda^2(d^2 \vee t^4) \exp\left(-c\min\left(\frac{t^4}{d}, t^2\right)\right) \\
        &= 2C^* \lambda^2 t^4 \exp\left(-ct^2\right)
    \end{align*}
    where \(c\) is a universal positive constant and \(C^*\) remains a universal constant but whose value may change from line to line. Note we have also used Corollary \ref{corollary:chisquare_tail} in the above display as well as \(t^2 \geq \widetilde{c}d\). To summarize, we have shown 
    \begin{equation}\label{eqn:low_dim_mgf_I}
        \lambda^2E(Y^2 \mathbbm{1}_{\{Y < 0\}}) \leq 2C^* \lambda^2 t^4 \exp\left(-ct^2\right).
    \end{equation}
    We now bound the second term in (\ref{eqn:low_dim_mgf}). Let \(f_d\) denote the probability density function of the \(\chi^2_d\) distribution. We have 
    \begin{align*}
        \lambda^2 E(Y^2 \mathbbm{1}_{\{0 < Y < 1/\lambda\}}) &\leq \lambda^2 E(Y^2 \mathbbm{1}_{\{0 < Y\}}) \\
        &= \lambda^2 E((||Z||^2 - \alpha_{t}(d))^2 \mathbbm{1}_{\{||Z||^2 > \alpha_{t}(d)\}}) \\
        &= \lambda^2 \int_{\alpha_t(d)}^{\infty} (z - \alpha_t(d))^2 f_d(z) \, dz \\
        &\leq 2\lambda^2 \int_{\alpha_t(d)}^{\infty} \left(z^2 + \alpha_t(d)^2\right) f_d(z) \, dz.
    \end{align*}
    An application of Lemma \ref{lemma:johnstone_chisquare} yields 
    \begin{align*}
        \int_{\alpha_t(d)}^{\infty} z^2 f_d(z) \, dz &= d(d+2) P\left\{\chi^2_{d+4} \geq \alpha_t(d)\right\} \\
        &\leq 3d^2 P\left\{\chi^2_{d+4} \geq d + 4 + (\alpha_t(d) - d - 4)\right\} \\
        &\leq 3d^2P\left\{\chi^2_{d+4} \geq d + 4 + (t^2 - 4)\right\} \\
        &\leq 6d^2 \exp\left(-c\min\left(\frac{(t^2 - 4)^2}{d+4}, (t^2 - 4)\right)\right) \\
        &\leq 6d^2 \exp\left(-c\min\left(\frac{t^4}{20d}, \frac{t^2}{2} \right)\right) \\
        &\leq 6d^2 \exp\left(-c \min\left(\frac{t^4}{d}, t^2\right)\right) \\
        &= 6d^2 \exp\left(-ct^2\right)
    \end{align*}
    where \(c\) remains a universal constant but has a value which may change from line to line. Note we have used \(\alpha_{t}(d) \geq d + t^2\), Corollary \ref{corollary:chisquare_tail}, and \(t^2 \geq \widetilde{c}d\). Likewise, consider that 
    \begin{align*}
        \int_{\alpha_t(d)}^{\infty} \alpha_t(d)^2 f_d(z) \, dz &= \alpha_{t}(d)^2 P\{\chi^2_d \geq \alpha_t(d)\} \\
        &\leq \alpha_t(d)^2 P\{\chi^2_d \geq d + t^2\} \\
        &\leq 2(d + C^*t^2)^2 \exp\left(-c \min\left(\frac{t^4}{d}, t^2\right)\right) \\
        &\leq C^*t^4 \exp\left(-c\min\left(\frac{t^4}{d}, t^2\right)\right) \\
        &= C^*t^4 \exp\left(-ct^2\right)
    \end{align*}
    where we have used \((a+b)^2 \leq 2a^2 + 2b^2\), \(t^2 \geq \widetilde{c}d\), Lemma \ref{lemma:conditional_2}, and Corollary \ref{corollary:chisquare_tail}. Again, \(C^*\) and \(c\) remain universal constants but have values which may change from line to line. Thus, we have shown 
    \begin{equation}\label{eqn:low_dim_mgf_II}
        \lambda^2 E(Y^2 \mathbbm{1}_{\{Y < 0\}}) \leq C^* \lambda^2 t^4 \exp\left(-ct^2\right).
    \end{equation}
    We now bound the final term in (\ref{eqn:low_dim_mgf}). Consider that for \(0 \leq \lambda < \frac{1}{2}\) we have 
    \begin{equation*}
        E(e^{\lambda Y} \mathbbm{1}_{\{Y > 1/\lambda\}}) = \int_{\alpha_t(d) + 1/\lambda}^{\infty} e^{\lambda(z - \alpha_t(d))} f_d(z) \, dz = e^{-\lambda \alpha_t(d)} \int_{\alpha_t(d)+1/\lambda}^{\infty} \frac{1}{2^{d/2}\Gamma(d/2)} e^{-\left(\frac{1}{2} - \lambda\right) z} z^{\frac{d}{2} - 1} \, dz.
    \end{equation*}
    Let \(u = \left(\frac{1}{2} - \lambda\right) z\). Then \(du = \left(\frac{1}{2} - \lambda\right) dz\) and so 
    \begin{align*}
        e^{-\lambda \alpha_t(d)} \int_{\alpha_t(d)+1/\lambda}^{\infty} \frac{1}{2^{d/2}\Gamma(d/2)} e^{-\left(\frac{1}{2} - \lambda\right) z} z^{\frac{d}{2} - 1} \, dz &= \frac{e^{-\lambda \alpha_t(d)}}{2^{d/2}\Gamma(d/2)} \int_{\left(\alpha_t(d) + 1/\lambda\right)(1/2 - \lambda)}^{\infty} e^{-u} u^{\frac{d}{2} - 1} \left(\frac{1}{2} - \lambda\right)^{-\frac{d}{2}} \, du \\
        &= \frac{e^{-\lambda \alpha_t(d)}}{(1-2\lambda)^{d/2}} \cdot \frac{1}{\Gamma\left(\frac{d}{2}\right)} \int_{(\alpha_t(d) + 1/\lambda)(1/2-\lambda)}^{\infty} e^{-u} u^{\frac{d}{2} - 1} \, du \\
        &= \frac{e^{-\lambda \alpha_t(d)}}{(1-2\lambda)^{d/2}} \cdot \frac{\Gamma\left(\frac{d}{2}, \left(\frac{1}{\lambda} + \alpha_t(d)\right)\left(\frac{1}{2} - \lambda\right)\right)}{\Gamma\left(\frac{d}{2}\right)}
    \end{align*}
    where \(\Gamma(s, x) = \int_{x}^{\infty} e^{-t} t^{s-1} \, dt\). To summarize, we have shown 
    \begin{equation}\label{eqn:low_dim_mgf_III_pt1}
        E(e^{\lambda Y} \mathbbm{1}_{\{Y > 1/\lambda\}}) = \frac{e^{-\lambda \alpha_t(d)}}{(1-2\lambda)^{d/2}} \cdot \frac{\Gamma\left(\frac{d}{2}, \left(\frac{1}{\lambda} + \alpha_t(d)\right)\left(\frac{1}{2} - \lambda\right)\right)}{\Gamma\left(\frac{d}{2}\right)}. 
    \end{equation}
    To continue, we would like to apply Corollary \ref{thm:temme_order_1}, but we must first verify the conditions. Let \(a = \frac{d}{2}\) and \(\eta = \sqrt{2(\mu - \log(1 + \mu))}\) with \(\mu = \frac{\left(\frac{1}{\lambda} + \alpha_t(d)\right)\left(\frac{1}{2} - \lambda\right)}{a} - 1\). Note we can take \(D\) to be a sufficiently large universal constant in order to ensure \(a\) is sufficiently large. Since \(\widetilde{c}d \leq t^2\), it follows from Lemma \ref{lemma:conditional_2} that \(\alpha_t(d) \leq \widetilde{C}t^2\) for some positive universal constant \(\widetilde{C}\). Without loss of generality, we can take \(\widetilde{C} \geq \frac{3}{2}\). Let us restrict our attention to \(\lambda < \frac{1}{2\widetilde{C}}\). Observe that 
    \begin{equation*}
        \mu = \frac{\frac{1}{\lambda} + \alpha_t(d) - 2 - 2\lambda \alpha_t(d) - d}{d} \geq \frac{(2\widetilde{C} - 2) + t^2 - \frac{2}{2\widetilde{C}}(\widetilde{C}t^2)}{d} \geq \frac{2(\widetilde{C} - 1)}{d} = \frac{C^{**}}{d}
    \end{equation*}
    where we have defined \(C^{**} = 2(\widetilde{C} - 1)\). Since we have shown \(\mu > 0\), we can apply Corollary \ref{thm:temme_order_1}. By Corollary \ref{thm:temme_order_1} we have 
    \begin{equation*}
        \frac{\Gamma\left(\frac{d}{2}, \left(\frac{1}{\lambda} + \alpha_t(d)\right)\left(\frac{1}{2} - \lambda\right)\right)}{\Gamma\left(\frac{d}{2}\right)} \leq \left(1 - \Phi\left(\eta \sqrt{a}\right)\right) + \frac{C^*}{\sqrt{d}} \exp\left(-\frac{a\eta^2}{2}\right).
    \end{equation*}
    Here, \(\Phi\) denotes the cumulative distribution function of the standard normal distribution and \(C^*\) is a universal constant. By the Gaussian tail bound \(1 - \Phi(x) \leq \frac{\varphi(x)}{x}\) for \(x > 0\) where \(\varphi = \Phi'\), we have 
    \begin{equation*}
        \frac{\Gamma\left(\frac{d}{2}, \left(\frac{1}{\lambda} + \alpha_t(d)\right)\left(\frac{1}{2}-\lambda\right)\right)}{\Gamma\left(\frac{d}{2}\right)} \leq C^* \exp\left(-\frac{a\eta^2}{2}\right)\left(\frac{1}{\eta\sqrt{a}} + \frac{1}{\sqrt{d}}\right)
    \end{equation*}
    where the value of \(C^*\) has changed from line to line but remains a universal constant. To continue the calculation, we need to bound \(\eta\sqrt{a}\) from below. Note we can take \(D \geq C^{**}\). For \(\lambda < \frac{1}{2\widetilde{C}}\), it follows from \(\frac{C^{**}}{d} \leq 1\) and that 
    \begin{align*}
        \eta\sqrt{a} &= \sqrt{a} \sqrt{2(\mu - \log(1 + \mu))} \\
        &\geq \sqrt{a} \sqrt{2 \left(\frac{C^{**}}{d} - \log\left(1 + \frac{C^{**}}{d} \right)\right)} \\
        &\geq \sqrt{a} \sqrt{2 \cdot \left( \frac{(C^{**})^2}{2d} - \frac{(C^{**})^3}{3d^3} \right)} \\
        &= \sqrt{\frac{(C^{**})^2}{2d} - \frac{(C^{**})^3}{3d^2}} \\
        &= \frac{C^{**}}{\sqrt{d}} \sqrt{\frac{1}{2} - \frac{C^{**}}{3d}} \\
        &\geq \frac{C^{**}}{\sqrt{6d}}. 
    \end{align*}
    Consequently, we have the bound 
    \begin{equation*}
        \frac{\Gamma\left(\frac{d}{2}, \left(\frac{1}{\lambda} + \alpha_t(d)\right)\left(\frac{1}{2} - \lambda\right)\right)}{\Gamma\left(\frac{d}{2}\right)} \leq C^* \sqrt{d} \exp\left(-\frac{a\eta^2}{2}\right)
    \end{equation*}
    where the value of \(C^*\) has changed but remains a universal constant. We now examine the term \(e^{-a\eta^2/2}\). Consider that 
    \begin{align*}
        \exp\left(-\frac{a\eta^2}{2}\right) &= \exp\left(-a\left(\mu - \log\left(1 + \mu\right)\right)\right) \\
        &= \exp\left(-a\mu\right) \left( \frac{\left(\frac{1}{\lambda} + \alpha_t(d)\right)\left(1 - 2\lambda\right)}{d}\right)^{d/2} \\
        &= \exp\left(\frac{d}{2} - \frac{1}{2\lambda} + 1 - \frac{\alpha_t(d)}{2} + \lambda \alpha_t(d)\right) \left( \frac{\left(\frac{1}{\lambda} + \alpha_t(d)\right)\left(1- 2\lambda\right)}{d} \right)^{d/2}.
    \end{align*}
    Therefore, letting the value of \(C^*\) change from line to line but remaining a universal constant, we have from (\ref{eqn:low_dim_mgf_III_pt1})
    \begin{align*}
        E\left(e^{\lambda Y}\mathbbm{1}_{\{Y > 1/\lambda\}}\right) &\leq C^* \sqrt{d} \frac{e^{-\lambda \alpha_t(d)}}{(1-2\lambda)^{d/2}} \exp\left(\frac{d}{2} - \frac{1}{2\lambda} + 1 - \frac{\alpha_t(d)}{2} + \lambda \alpha_t(d)\right) \left(\frac{\left(\frac{1}{\lambda} + \alpha_t(d)\right)\left(1 - 2\lambda\right)}{d}\right)^{d/2} \\
        &\leq C^* \sqrt{d} e^{-\left(\frac{1}{2\lambda} + \frac{\alpha_t(d) - d}{2}\right)} \left(1 + \frac{\frac{1}{\lambda} + \alpha_t(d) - d}{d} \right)^{d/2} \\
        &= C^* \sqrt{d} \exp\left( - \left(\frac{1}{2\lambda} + \frac{\alpha_t(d) - d}{2}\right) + \frac{d}{2}\log\left(1 + \frac{\frac{1}{\lambda} + \alpha_t(d) - d}{d}\right) \right) \\
        &= C^* \sqrt{d} \exp\left( - \frac{d}{2}\left( \left(\frac{\frac{1}{\lambda} + \alpha_t(d) - d}{d} \right) - \log\left(1 + \frac{\frac{1}{\lambda} + \alpha_t(d) - d}{d}\right) \right)\right) \\
        &\leq C^* \sqrt{d} \exp\left(-\frac{d}{2} \cdot \frac{\frac{1}{\lambda} + \alpha_t(d) - d}{c^{**}d} \right) \\
        &\leq C^* \sqrt{d} \exp\left(-\frac{1}{2c^{**}\lambda} - \frac{t^2}{2c^{**}}\right) 
    \end{align*}
    for a universal positive constant \(c^{**}\). We have used that \(u - \log(1+u) \gtrsim u\) for \(u \geq 1\). Note that we can use this since \(\frac{\frac{1}{\lambda} + \alpha_t(d) - d}{d} \geq \frac{t^2}{d} \gtrsim 1\). Since \(e^{-\frac{1}{2c^{**}u}} \lesssim u^2\) for all \(u \in \R\), it follows that 
    \begin{equation}\label{eqn:low_dim_mgf_III}
        E\left(e^{\lambda Y} \mathbbm{1}_{\{Y > 1/\lambda\}}\right) \leq C^* \sqrt{d}\lambda^2 \exp\left(-\frac{t^2}{2c^{**}}\right)
    \end{equation}
    where the value of \(C^*\), again, has changed but remains a universal constant. Putting together our bounds (\ref{eqn:low_dim_mgf_I}), (\ref{eqn:low_dim_mgf_II}), (\ref{eqn:low_dim_mgf_III}) into (\ref{eqn:low_dim_mgf}) yields
    \begin{equation*}
        E(e^{\lambda Y}) \leq 1 + Ct^4\lambda^2 e^{-ct^2} \leq \exp\left(Ct^4\lambda^2e^{-ct^2}\right)
    \end{equation*}
    for \(\lambda < \frac{1}{2\widetilde{C}}\) where \(C, c > 0\) are universal constants. The proof is complete. 
\end{proof}

\begin{lemma}\label{lemma:tail_thold_exp}
    Let \(Z \sim N(\theta, I_d)\). Suppose \(\widetilde{c}\) is a universal positive constant. There exists a universal constant \(C\) such that for every \(t^2 \geq \widetilde{c}d\), we have 
    \begin{equation*}
        E\left\{ (||Z||^2 - \alpha_t(d))\mathbbm{1}_{\{||Z||^2 \geq d + t^2\}} \right\} 
        \begin{cases}
            = 0 &\textit{if } \theta = 0, \\
            \geq 0 &\textit{if } ||\theta||^2 < Ct^2, \\
            \geq ||\theta||^2/2 &\textit{if } ||\theta||^2 \geq Ct^2.
        \end{cases}
    \end{equation*}
    Here, \(\alpha_t(d)\) is given by (\ref{def:alpha}).
\end{lemma}
\begin{proof}
    We will make a choice for \(C\) at the end of the proof. Consider first the case where \(\theta = 0\). Then \(E\left\{(||Z||^2 - \alpha_t(d))\mathbbm{1}_{\{||Z||^2 \geq d + t^2\}}\right\} = 0\) by definition of \(\alpha_t(d)\) and so we have the desired result. The second case follows since the expression inside the expectation is stochastically increasing in \(||\theta||\). Moving on to the final case, suppose \(||\theta||^2 \geq Ct^2\). Since \(\alpha_t(d) \geq d + t^2\), we have 
    \begin{equation*}
        E\left\{ (||Z||^2 - \alpha_t(d))\mathbbm{1}_{\{||Z||^2 < d + t^2\}}\right\} \leq 0
    \end{equation*}
    Consequently, 
    \begin{align*}
        E\left\{(||Z||^2 - \alpha_t(d))\mathbbm{1}_{\{||Z||^2 \geq d + t^2\}}\right\} &= E\left\{(||Z||^2 - \alpha_t(d))(1 - \mathbbm{1}_{\{||Z||^2 < d + t^2\}})\right\} \\
        &= d + ||\theta||^2 - \alpha_t(d) - E\left\{(||Z||^2 - \alpha_t(d))\mathbbm{1}_{\{||Z||^2 < d + t^2\}}\right\} \\
        &\geq d + ||\theta||^2 - \alpha_t(d)
    \end{align*}
    By Lemma \ref{lemma:conditional_2} and \(t^2 \gtrsim d\), there exists a universal positive constant \(C^*\) such that \(\alpha_t(d) \leq d + C^* t^2\). Therefore, 
    \begin{align*}
        E\left\{(||Z||^2 - \alpha_t(d))\mathbbm{1}_{\{||Z||^2 \geq d + t^2\}}\right\} \geq ||\theta||^2 - C^* t^2 \geq ||\theta||^2 \left(1 - \frac{C^*}{C}\right). 
    \end{align*}
    Selecting \(C = \frac{C^*}{2}\) yields the desired result. 
\end{proof}

\begin{lemma}\label{lemma:tail_thold_var}
    Let \(Z \sim N(\theta, I_d)\). If \(t^2 \geq \widetilde{c}d\) for a universal positive constant \(\widetilde{c}\), then 
    \begin{equation*}
        \Var\left( (||Z||^2 - \alpha_t(d))\mathbbm{1}_{\{||Z||^2 \geq d + t^2\}} \right) \lesssim 
        \begin{cases}
            t^4 \exp\left(-c^* \min\left(\frac{t^4}{d}, t^2\right)\right) &\textit{if } \theta = 0, \\
            t^4 &\textit{if } 0 < ||\theta|| < 2t, \\
            ||\theta||^2 &\textit{if } ||\theta|| > 2t. 
        \end{cases}
    \end{equation*}
    Here, \(\alpha_t(d)\) is given by (\ref{def:alpha}) and \(c^*\) is a universal positive constant. 
\end{lemma}
\begin{proof}
    The proof of Lemma \ref{lemma:bulk_thold_var} can be repeated with the modification of invoking Lemma \ref{lemma:conditional_2} instead of Lemma \ref{lemma:bulk_alpha}.
\end{proof}

\begin{lemma}\label{lemma:null_tail}
    Suppose \(\widetilde{c}\) is a universal positive constant. There exist universal positive constants \(D\) and \(C^{*}\) such that if \(d \geq D\), \(t^2 \geq \widetilde{c} d\), and \(x > 0\), then 
    \begin{equation*}
        P\left\{\sum_{i=1}^{n} \left(||Z_i||^2 - \alpha_t(d)\right)\mathbbm{1}_{\{||Z||^2 \geq d + t^2\}} \geq C^* \left(\sqrt{x nt^4 e^{-ct^2}} + x\right) \right\} \leq e^{-x}
    \end{equation*}
    where \(Z_1,...,Z_n \overset{iid}{\sim} N(0, I_d)\) and \(\alpha_t(d)\) is given by (\ref{def:alpha}). Here, \(c\) is a universal positive constant. 
\end{lemma}

\begin{proof}
    Let \(D\) be the constant from Lemma \ref{lemma:tail_mgf}. We use the Chernoff method to obtain the desired bound. Let \(Y\) be as in Lemma \ref{lemma:tail_mgf} and let \(\widetilde{C}\) be the universal constant from Lemma \ref{lemma:tail_mgf}. For any \(u > 0\), we have by Lemma \ref{lemma:tail_mgf}
    \begin{align*}
        P\left\{\sum_{j = 1}^{n} \left(||Z_j||^2 - \alpha_t(d)\right)\mathbbm{1}_{\{||Z_j||^2 \geq d + t^2\}} > u \right\} &\leq \inf_{\lambda < \frac{1}{2\widetilde{C}}} e^{-\lambda u} \left(E(e^{\lambda Y})\right)^n \\
        &\leq \inf_{\lambda < \frac{1}{2\widetilde{C}}} \exp\left(-\lambda u + C nt^4 \lambda^2 e^{-ct^2}\right) \\
        &= \exp\left(-C\left(\frac{u^2}{nt^4 e^{-ct^2}} \wedge u\right)\right)
    \end{align*}
    where \(c > 0\) is a universal constant and \(C > 0\) is a universal constant whose value may change from line to line but remain a universal constant. Selecting \(u = \frac{1}{2}\left(\sqrt{\frac{x nt^4 e^{-ct^2}}{C}} + \frac{x}{C}\right)\) and selecting \(C^*\) suitably yields the desired result. The proof is complete. 
\end{proof}

\subsection{Bulk}

\begin{lemma}\label{lemma:bulk_mgf}
    Let \(L^*\) be the universal positive constant from Lemma \ref{lemma:bulk_alpha}. There exist universal constants \(C^*, C^{**}, C, c > 0\) such that if \(d \geq C^{**}\) and \(1 \leq \beta \leq L^* \sqrt{d}\), then 
    \begin{equation*}
        E(e^{\lambda Y}) \leq \exp\left(Cd\beta^2\lambda^2 e^{-c\beta^2}\right)
    \end{equation*}
    for \(\lambda < \frac{\beta}{2(\beta C^* + \sqrt{d})}\). Here, \(Y = (||Z||^2 - \alpha_t(d))\mathbbm{1}_{\{||Z||^2 \geq d + t^2\}}\) with \(t^2 = \beta \sqrt{d}\) and \(Z \sim N(0, I_d)\). Also, \(\alpha_t(d)\) is given by (\ref{def:alpha}).
\end{lemma}
\begin{proof}
    The argument we give will be similar to the proof of Lemma \ref{lemma:tail_mgf}, namely we will separately bound each term in (\ref{eqn:low_dim_mgf}) and substitute into the equation 
    \begin{equation*}
        E(e^{\lambda Y}) = 1 + E\left(e^{\lambda Y} - 1 - \lambda Y\right). 
    \end{equation*}
    We start with the first term on the right-hand side in (\ref{eqn:low_dim_mgf}). By Lemma \ref{lemma:johnstone_chisquare}, we have \(\alpha_t(d) \leq d + C^* \beta \sqrt{d}\) for a universal positive constant \(C^*\). Note that \(C^* \geq 1\) since we trivially have \(\alpha_t(d) \geq d + t^2 = d + \beta \sqrt{d}\). Consequently, arguing as in the proof of Lemma \ref{lemma:tail_mgf}, we have 
    \begin{align*}
        \lambda^2E(Y^2 \mathbbm{1}_{\{Y < 0\}}) &\leq \lambda^2(d + t - \alpha_t(d))^2 P\{||Z||^2 \geq d + t^2\}\\
        &\leq C \lambda^2 d\beta^2 \exp\left(-c\min\left(\frac{t^4}{d}, t^2\right)\right) \\
        &= C\lambda^2 d\beta^2 e^{-c\beta^2}
    \end{align*}
    where \(C, c > 0\) are universal constants. We have used Corollary \ref{corollary:chisquare_tail} to obtain the final inequality. We now bound the second term in (\ref{eqn:low_dim_mgf}). Letting \(f_d\) denote the probability density function of the \(\chi^2_d\) distribution, we can repeat and then continue the calculation in the proof of Lemma \ref{lemma:tail_mgf} to obtain 
    \begin{align*}
        \lambda^2 E(Y^2 \mathbbm{1}_{\{0 < Y < 1/\lambda\}}) &\leq \lambda^2 \int_{\alpha_t(d)}^{\infty} (z - \alpha_t(d))^2 f_d(z) \, dz \\
        &\leq \lambda^2 \int_{d+t^2}^{\infty} (z - \alpha_t(d))^2 f_d(z) \, dz \\
        &= \lambda^2 \left( \int_{d+t^2}^{\infty} z^2 f_d(z) \, dz - 2\alpha_t(d) \int_{d+t^2}^{\infty} zf_d(z) \, dz + \alpha_t(d)^2 \int_{d + t^2} f_d(z) \, dz \right) \\
        &= \lambda^2 P\{\chi^2_d \geq d + t^2\} \Var\left(||Z||^2 \,|\, ||Z||^2 \geq d + t^2\right) \\
        &\leq C \lambda^2 d \beta^2 \exp\left(-c \min\left(\frac{t^4}{d}, t^2\right)\right) \\
        &= C\lambda^2 d \beta^2 e^{-c\beta^2}
    \end{align*}
    where the values of \(C, c > 0\) have changed but remain universal constants. We have used Lemma \ref{lemma:conditional_var} here. 

    We now bound the final term in (\ref{eqn:low_dim_mgf}). Arguing exactly as in the proof of Lemma \ref{lemma:tail_mgf}, we have 
    \begin{equation*}
        E(e^{\lambda Y} \mathbbm{1}_{\{Y > 1/\lambda\}}) = \frac{e^{-\lambda \alpha_t(d)}}{(1-2\lambda)^{d/2}} \cdot \frac{\Gamma\left(\frac{d}{2}, \left(\frac{1}{\lambda} + \alpha_t(d)\right)\left(\frac{1}{2}-\lambda\right)\right)}{\Gamma\left(\frac{d}{2}\right)}
    \end{equation*}
    for \(0 \leq \lambda < \frac{1}{2}\). Recall that \(\Gamma(s, x) = \int_{x}^{\infty} e^{-t} t^{s-1} \, dt\). 
    
    Let us restrict our attention to \(\lambda < \frac{\beta}{2(\sqrt{d} + C^*\beta)}\). We seek to apply Corollary \ref{thm:temme_order_1}, but we must verify the conditions. Let \(a = \frac{d}{2}\) and \(\eta = \sqrt{2(\mu - \log\left(1 + \mu\right))}\) with \(\mu = \frac{\left(\frac{1}{\lambda} + \alpha_t(d)\right)\left(\frac{1}{2} - \lambda\right)}{a} - 1\). Observe 
    \begin{equation*}
        \mu = \frac{\left(\frac{1}{\lambda} + \alpha_t(d)\right)\left(1 - 2\lambda\right) - d}{d} = \frac{\frac{1-2\lambda}{\lambda} + \alpha_t(d)(1-2\lambda) - d}{d}.
    \end{equation*}
    Consider that \(\lambda < \frac{\beta}{2(\sqrt{d} + C^*\beta)} = \frac{t^2}{2(d+C^*\beta\sqrt{d})} \leq \frac{\alpha_t(d) - d}{2\alpha_t(d)}\). Therefore, \(\alpha_t(d)(1-2\lambda) - d \geq 0\) and so 
    \begin{equation*}
        \mu \geq \frac{1-2\lambda}{\lambda d} \geq \frac{2}{\beta\sqrt{d}}.
    \end{equation*}
    Note we have used \(C^* \geq 1\). Since \(\mu > 0\) and \(d \geq C^{**}\) which is a sufficiently large universal constant, we can apply Corollary \ref{thm:temme_order_1}. By Corollary \ref{thm:temme_order_1}, we have 
    \begin{equation*}
        \frac{\Gamma\left(\frac{d}{2}, \left(\frac{1}{\lambda} + \alpha_t(d)\right)\left(\frac{1}{2} - \lambda\right)\right)}{\Gamma\left(\frac{d}{2}\right)} \leq \left(1 - \Phi\left(\eta\sqrt{a}\right)\right) + \frac{C^{\dag}}{\sqrt{d}}\exp\left(-\frac{a\eta^2}{2}\right).
    \end{equation*}
    Here, \(C^{\dag}\) is a positive universal constant. Recall that \(\Phi\) denotes the cumulative distribution function of the standard normal distribution. By the Gaussian tail bound \(1-\Phi(x) \leq \frac{\varphi(x)}{x}\) for \(x > 0\), we have 
    \begin{equation}\label{eqn:bulk_mgf_gamma}
        \frac{\Gamma\left(\frac{d}{2}, \left(\frac{1}{\lambda} + \alpha_t(d)\right)\left(\frac{1}{2} - \lambda\right)\right)}{\Gamma\left(\frac{d}{2}\right)} \leq C^{\dag} \exp\left(-\frac{a\eta^2}{2}\right)\left(\frac{1}{\eta\sqrt{a}} + \frac{1}{\sqrt{d}}\right)
    \end{equation}
    where the value of \(C^{\dag}\) has changed but remains a universal constant. To continue with the bound, we need to bound \(\eta\sqrt{a}\) from below. Let us take \(C^{**}\) larger than \(4\). Since \(a = \frac{d}{2}\), we have 
    \begin{align*}
        \eta \sqrt{a} &= \sqrt{a}\sqrt{2(\mu - \log(1+\mu))} \\
        &\geq \sqrt{a}\sqrt{2\left(\frac{2}{\beta\sqrt{d}} + \log\left(1 + \frac{2}{\beta\sqrt{d}}\right)\right)} \\
        &\geq \sqrt{a}\sqrt{2\left(\frac{4}{2\beta^2 d} - \frac{8}{3\beta^3d^{3/2}}\right)} \\
        &\geq \frac{c^*}{\beta}
    \end{align*}
    where \(c^* > 0\) is a universal constant. We can conclude from (\ref{eqn:bulk_mgf_gamma}) that 
    \begin{equation*}
        \frac{\Gamma\left(\frac{d}{2}, \left(\frac{1}{\lambda} + \alpha_t(d)\right)\left(\frac{1}{2} - \lambda\right)\right)}{\Gamma\left(\frac{d}{2}\right)} \leq C^{\dag} \beta \exp\left(-\frac{a\eta^2}{2}\right)
    \end{equation*}
    where the value of \(C^\dag\) has changed but remains a universal constant. We now examine the term \(e^{-\frac{a\eta^2}{2}}\). Arguing exactly as in the proof of Lemma \ref{lemma:tail_mgf}, we obtain 
    \begin{equation*}
        \exp\left(-\frac{a\eta^2}{2}\right) = \exp\left(\frac{d}{2} - \frac{1}{2\lambda} + 1 - \frac{\alpha_t(d)}{2} + \lambda \alpha_t(d)\right) \left( \frac{\left(\frac{1}{\lambda} + \alpha_t(d)\right)\left(1 - 2\lambda\right)}{d} \right)^{d/2}
    \end{equation*}
    which, as argued in the proof of Lemma \ref{lemma:tail_mgf}, leads us to the bound 
    \begin{equation*}
        E(e^{\lambda Y}\mathbbm{1}_{Y > 1/\lambda}) \leq C^* \beta \exp\left( C\lambda^2 d\beta^2 e^{-c\beta^2} \right)
    \end{equation*}
    for \(\lambda < \frac{\beta}{2(\sqrt{d} + C^*\beta)}\) as desired. 
\end{proof}

\begin{lemma}\label{lemma:bulk_thold_exp}
    Let \(Z \sim N(\theta, I_d)\). Suppose \(1 \leq \beta \leq L^* \sqrt{d}\) where \(L^*\) is the universal constant from Lemma \ref{lemma:bulk_alpha}. There exists a universal positive constant \(C\) such that if \(t^2 = \beta \sqrt{d}\), then we have 
    \begin{equation*}
        E\left\{(||Z||^2 - \alpha_t(d))\mathbbm{1}_{\{||Z||^2 \geq d + t^2\}}\right\} 
        \begin{cases}
            = 0 &\textit{if } \theta = 0, \\
            \geq 0 &\textit{if } ||\theta||^2 < C t^2, \\
            \geq ||\theta||^2/2 &\textit{if } ||\theta||^2 \geq C t^2
        \end{cases}
    \end{equation*}
    where \(\alpha_t(d)\) is given by (\ref{def:alpha}).
\end{lemma}
\begin{proof}
    We will make a choice for \(C\) later on in the proof. The proof for the cases \(\theta = 0\) and \(||\theta||^2 \leq C t^2\) follows exactly as in the proof of Lemma \ref{lemma:tail_thold_exp}. Here, we focus on the final case in which \(||\theta||^2 \geq C t^2\). Since \(\alpha_t(d) \geq d + t^2\), we have 
    \begin{equation*}
        E\left\{(||Z||^2 - \alpha_t(d))\mathbbm{1}_{\{||Z||^2 < d + t^2\}}\right\} \leq 0.
    \end{equation*}
    Since \(t^2 = \beta \sqrt{d}\), we can apply Lemma \ref{lemma:bulk_alpha} to obtain 
    \begin{align*}
        E\left\{(||Z||^2 - \alpha_t(d))\mathbbm{1}_{\{||Z||^2 \geq d + t^2\}}\right\} &= E\left\{(||Z||^2 - \alpha_t(d))(1 - \mathbbm{1}_{\{||Z||^2 < d + t^2\}})\right\} \\
        &\geq d + ||\theta||^2 - \alpha_t(d) - E\left\{(||Z||^2 - \alpha_t(d))\mathbbm{1}_{\{||Z||^2 < d + t^2\}}\right\} \\
        &\geq ||\theta||^2 - C^* \beta \sqrt{d} \\
        &= ||\theta||^2 \left(1 - \frac{C^*\beta \sqrt{d}}{||\theta||^2}\right) \\
        &= ||\theta||^2 \left(1 - \frac{C^*t^2}{||\theta||^2}\right) \\
        &\geq ||\theta||^2 \left(1 - \frac{C^*}{C}\right). 
    \end{align*}
    where \(C^*\) is a universal positive constant. Taking \(C = \frac{C^*}{2}\) completes the proof.
\end{proof}

\begin{lemma}\label{lemma:bulk_thold_var}
    Let \(Z \sim N(\theta, I_d)\). Suppose \(1 \leq \beta \leq L^* \sqrt{d}\) where \(L^*\) is the universal constant from Lemma \ref{lemma:bulk_alpha}. Then there exists a universal positive constant \(c^*\) such that if \(t^2 = \beta \sqrt{d}\), then 
    \begin{equation*}
        \Var\left((||Z||^2 - \alpha_t(d))\mathbbm{1}_{\{||Z||^2 \geq d + t^2\}}\right) \lesssim 
        \begin{cases}
            t^4 \exp\left(-\frac{c^*t^4}{d}\right) &\textit{if } \theta = 0, \\
            d + ||\theta||^2 &\textit{if } ||\theta|| \geq 2t, \\
            t^4 &\textit{if } 0 < ||\theta|| < 2t.
        \end{cases}
    \end{equation*}
    Here, \(\alpha_t(d)\) is given by (\ref{def:alpha}).
\end{lemma}
\begin{proof}
    Let \(f_d\) and \(F_d\) respectively denote the probability density function and cumulative distribution function of the \(\chi^2_d\) distribution. 

    \textbf{Case 1:} Consider the first case in which \(\theta = 0\). Then 
    \begin{align*}
        \Var((||Z||^2 - \alpha_t(d))\mathbbm{1}_{\{||Z||^2 \geq d + t^2\}}) &\leq E\left((||Z||^2 - \alpha_t(d))^2 \mathbbm{1}_{\{||Z||^2 \geq d + t^2\}}\right) \\
        &= P\left\{||Z||^2 \geq d + t^2\right\} \cdot E\left( (||Z||^2 - \alpha_t(d))^2 \,|\, ||Z||^2 \geq d + t^2\right) \\
        &= \left(1 - F_d(d+t^2)\right)\Var\left(||Z||^2 \,|\, ||Z||^2 \geq d + t^2\right) \\
        &\leq C^* \left(1 - F_d(d+t^2)\right) d \beta^2
    \end{align*}
    where we have applied the definition of \(\alpha_t(d)\) and we have applied Lemma \ref{lemma:conditional_var}. Here, \(C^*\) is a universal positive constant. An application of Corollary \ref{corollary:chisquare_tail} gives the desired result for this case. 

    We now move to the other two cases. Suppose \(\theta \neq 0\). For ease of notation, let \(Y = (||Z||^2 - \alpha_t(d))\mathbbm{1}_{\{||Z||^2 \geq d + t^2\}}\). Observe 
    \begin{align}
        &\Var(Y) \nonumber\\
        &= E\left(\Var(Y\,|\, \mathbbm{1}_{\{||Z||^2 \geq d + t^2\}})\right) + \Var\left(E(Y \,|\, \mathbbm{1}_{\{||Z||^2 \geq d + t^2\}})\right) \nonumber \\
        &\leq P\left\{||Z||^2 \geq d + t^2\right\} \Var\left(||Z||^2 \,|\, ||Z||^2 \geq d + t^2\right) + P\left\{||Z||^2 \geq d + t^2\right\} \left(E(||Z||^2 - \alpha_t(d) \,|\, ||Z||^2 \geq d + t^2)\right)^2.\label{eqn:bulk_var}
    \end{align}
    Note that \(0 \leq E(||Z||^2 - \alpha_t(d) \,|\, ||Z||^2 \geq d + t^2)\) since \(\theta \neq 0\) (by an appeal to stochastic ordering), and so we can find upper bounds for the square of this conditional expectation by first finding upper bounds on the conditional expectation. We examine each term separately in the above display. First, consider 
    \begin{align}
        P\left\{||Z||^2 \geq d + t^2\right\}\Var\left(||Z||^2 \,|\, ||Z||^2 \geq d + t^2\right) &\leq E\left(\Var\left(||Z||^2 \,|\, \mathbbm{1}_{\{||Z||^2 \geq d + t^2\}}\right)\right) \nonumber \\
        &\leq \Var(||Z||^2) \nonumber\\
        &= 2d + 4||\theta||^2. \label{eqn:bulk_var_I}
    \end{align}
    We now examine the second term. Note we can write \(Z = g + \theta\) where \(g \sim N(0, I_d)\). Therefore, 
    \begin{align}
        E(||Z||^2 - \alpha_t(d) \,|\, ||Z||^2 \geq d + t^2) &= E\left(||g+\theta||^2 - \alpha_t(d) \,|\, ||g + \theta||^2 \geq d + t^2\right) \nonumber\\
        &\leq E(||g||^2 + 2\langle \theta, g\rangle + ||\theta||^2 - d - t^2 \,|\, ||g+\theta||^2 \geq d + t^2) \nonumber \\
        &= E(||g||^2 - d \,|\, ||g + \theta||^2 \geq d + t^2) + ||\theta||^2 - t^2 \label{eqn:bulk_var_II}
    \end{align}
    where we have used that \(\langle \theta, g \rangle \mathbbm{1}_{\{||g + \theta||^2 \geq d + t^2\}} \overset{d}{=} \langle \theta, - g\rangle \mathbbm{1}_{\{||-g+\theta||^2 \geq d + t^2\}}\). With this in hand, we have \(E(\langle \theta, g \rangle \mathbbm{1}_{\{||g + \theta||^2 \geq d + t^2\}}) = 0\), which further implies \(E(\langle \theta, g \rangle \,|\, ||g + \theta||^2 \geq d + t^2) = 0\). Note we have also used \(\alpha_t(d) \geq d + t^2\) to obtain the second line in the above display. With the above display in hand, we now examine the remaining two cases. 

    \textbf{Case 2:} Consider the case \(||\theta|| \geq 2t\). Observe that 
    \begin{equation*}
        E(||g||^2 - d \,|\, ||g + \theta||^2 \geq d + t^2) = \frac{E\left((||g||^2 - d)\mathbbm{1}_{\{||g + \theta||^2 \geq d + t^2\}}\right)}{P\left\{||g+\theta||^2 \geq d + t^2\right\}}.
    \end{equation*}
    Examining the denominator, since \(||g + \theta||^2 \sim \chi^2_d(||\theta||^2) \overset{d}{=} \chi^2_{d-1} + \chi^2_1(||\theta||^2)\) where the two \(\chi^2\)-variates on the right hand side are independent, we have 
    \begin{equation*}
        P\{||g+\theta||^2 \geq d + t^2\} \geq P\{\chi^2_{d-1} \geq d - 1\} P\{\chi^2_1(||\theta||^2) \geq 1 + t^2\} \geq c P\{\chi^2_1(||\theta||^2) \geq 1 + t^2\} 
    \end{equation*}
    where \(c\) is a universal positive constant. Examining the numerator, consider that by Lemma \ref{lemma:johnstone_chisquare} we have 
    \begin{align*}
        E\left((||g||^2 - d)\mathbbm{1}_{\{||g+\theta||^2 \geq d + t^2\}}\right) &\leq E((||g||^2 - d)\mathbbm{1}_{\{||g||^2 \geq d\}}) \\
        &= \int_{d}^{\infty} (z-d)f_d(z) \, dz \\
        &= d\int_{d}^{\infty} (f_{d+2}(z) - f_d(z))\, dz \\
        &= -2d\int_{d}^{\infty} f_{d+2}'(z) \, dz \\
        &= 2df_{d+2}(d).
    \end{align*}
    Hence, from (\ref{eqn:bulk_var_II}), we have shown 
    \begin{equation*}
        E(||Z||^2 - \alpha_t(d) \,|\, ||Z||^2 \geq d + t^2) \leq \frac{2df_{d+2}(d)}{cP\{\chi^2_1(||\theta||^2) \geq 1 + t^2\}} + ||\theta||^2 - t^2.
    \end{equation*}
    Thus, we have the bound 
    \begin{align*}
        \Var\left(Y\right) \leq 2d + 4||\theta||^2 + P\{||Z||^2 \geq d + t^2\} \left( \frac{2df_{d+2}(d)}{cP\{\chi^2_1(||\theta||^2) \geq 1 + t^2\}} + ||\theta||^2 - t^2 \right)^2.
    \end{align*}
    Consider that
    \begin{equation*}
        \frac{2df_{d+2}(d)}{c} = \frac{d}{c\Gamma\left(\frac{d}{2}+1\right)} \left(\frac{d}{2e}\right)^{d/2} \leq c\sqrt{d}
    \end{equation*}
    where the value of \(c\) can change in each expression but remains a universal positive constant. The final inequality follows from Stirling's formula. Moreover, since \(||\theta|| \geq 2t\), there exists a positive universal constant \(c'\) such that \(P\{\chi^2_1(||\theta||^2) \geq 1 + t^2\} \geq 1/c'\). Furthermore, it follows by Chebyshev's inequality that 
    \begin{equation*}
        P\{||Z||^2 < d + t^2\} = P\{||\theta||^2 - t^2 < d + ||\theta||^2 - ||Z||^2\} \leq \frac{2d + 4||\theta||^2}{(||\theta||^2 - t^2)^2}.
    \end{equation*}
    Consequently, 
    \begin{align*}
        \Var(Y) &\leq 2d + 4||\theta||^2 + \frac{2d + 4||\theta||^2}{(||\theta||^2 - t^2)^2} (cc'\sqrt{d} + ||\theta||^2 - t^2)^2 \\
        &\asymp d + ||\theta||^2 + \frac{d + ||\theta||^2}{(||\theta||^2 - t^2)^2} (\sqrt{d} + ||\theta||^2)^2 \\
        &\asymp d + ||\theta||^2 + \frac{d + ||\theta||^2}{||\theta||^4} ||\theta||^4 \\
        &\asymp d + ||\theta||^2.
    \end{align*}
    The proof for this case is complete. 

    \textbf{Case 3:} Consider the case \(0 < ||\theta|| < 2t\). Then 
    \begin{align*}
        E(||Z||^2 - \alpha_t(d) \,|\, ||Z||^2 \geq d + t^2) &= E(||g||^2 - d \,|\, ||g+\theta||^2 \geq d + t^2) + ||\theta||^2 - t^2 \\
        &\leq E(||g||^2 - d \,|\, ||g + \theta||^2 \geq d + t^2) + 3t^2 \\
        &\leq E(||g||^2 - d \,|\, ||g||^2 \geq d + t^2) + 3t^2 \\
        &= \alpha_t(d) - d + 3t^2 \\
        &\lesssim \beta \sqrt{d} + t^2 \\
        &\asymp t^2
    \end{align*}
    where we have used Lemma \ref{lemma:bulk_alpha}. Therefore, using the above bound with (\ref{eqn:bulk_var_I}) and plugging into (\ref{eqn:bulk_var}), we obtain 
    \begin{equation*}
        \Var(Y) \lesssim d + ||\theta||^2 + t^4 \asymp t^4.
    \end{equation*}
    The proof is complete.
\end{proof}

\begin{lemma}\label{lemma:null_bulk}
    Let \(L^*\) be the universal positive constant from Lemma \ref{lemma:bulk_alpha}. There exists universal positive constants \(D\) and \(C^{*}\) such that if \(d \geq D\), then for any \(1 \leq \beta \leq L^* \sqrt{d}\) and \(x > 0\), we have 
    \begin{equation*}
        P\left\{\sum_{j = 1}^{n} \left(||Z_j||^2 - \alpha_t(d)\right) \mathbbm{1}_{\{||Z_j||^2 \geq d + t^2\}} \geq C^* \left(\sqrt{x n t^4 e^{-\frac{ct^4}{d}}} + \frac{d}{t^2}x\right) \right\} \leq e^{-x}
    \end{equation*}
    where \(Z_1,...,Z_n \overset{iid}{\sim} N(0, I_d)\), \(t^2 = \beta \sqrt{d}\), and \(\alpha_{t}(d)\) is given by (\ref{def:alpha}). Here, \(c > 0\) is a universal constant. 
\end{lemma}
\begin{proof}
    The proof of Lemma \ref{lemma:null_tail} can be repeated with the modification of invoking Lemma \ref{lemma:bulk_mgf} instead of Lemma \ref{lemma:tail_mgf} and taking infimum over \(\lambda < \frac{\beta}{2\left(\beta C^* + \sqrt{d}\right)}\). 
\end{proof}

\section{\texorpdfstring{Properties of the \(\chi^2_d\) distribution}{ Properties of the chi-squared distribution}}

\begin{theorem}[Bernstein's inequality - Theorem 2.8.1 \cite{vershynin_high-dimensional_2018}]
    Let \(Y_1,...,Y_d\) be independent mean-zero subexponential random variables. Then, for every \(u \geq 0\), we have 
    \begin{equation*}
        P\left\{ \left| \sum_{i=1}^{d} Y_i \right| \geq u \right\} \leq 2 \exp\left(-c\min\left(\frac{u^2}{\sum_{i=1}^{d} ||X_i||^2_{\psi_1}}, \frac{u}{\max_i ||X_i||_{\psi_1}} \right) \right)
    \end{equation*}
    where \(c > 0\) is a universal constant and \(\psi_2, \psi_1\) denote the subgaussian and subexponential norms respectively (see (2.13) and (2.21) of \cite{vershynin_high-dimensional_2018})
\end{theorem}

\begin{corollary}\label{corollary:chisquare_tail}
    Suppose \(Z \sim N(0, I_d)\). If \(u \geq 0\), then 
    \begin{equation*}
        P\left\{||Z||^2 \geq d + u\right\} \leq 2 \exp\left(-c\min\left(\frac{u^2}{d}, u\right)\right)
    \end{equation*}
    where \(c > 0\) is a universal constant.
\end{corollary}

\begin{lemma}[Lemma 1 \cite{laurent_adaptive_2000}]\label{lemma:laurent_massart}
    Let \(Z_1,...,Z_d \overset{iid}{\sim} N(0, 1)\). If \(\lambda_1 \geq ... \geq \lambda_d > 0\) and \(x > 0\), then 
    \begin{equation*}
        P\left\{ \sum_{j=1}^{d} \lambda_j Z_j^2 \geq \sum_{j=1}^{d} \lambda_j + 2\sqrt{x \sum_{j=1}^{d} \lambda_j^2} + 2\lambda_1 x \right\} \leq e^{-x}.
    \end{equation*}
\end{lemma}

\begin{lemma}[Corollary 3 \cite{zhang_non-asymptotic_2020}]\label{lemma:anru_chisquare}
    Suppose \(Z \sim N(0, I_d)\). There exist universal constants \(C, c > 0\) such that
    \begin{equation*}
        P\left\{||Z||^2 \geq d + u\right\} \geq C \exp\left(-c \min\left(\frac{u^2}{d}, u\right)\right)
    \end{equation*}
    for all \(u > 0\). 
\end{lemma}

\begin{lemma}[\cite{johnstone_chi-square_2001}]\label{lemma:johnstone_chisquare}
    Let \(f_d\) and \(F_d\) respectively denote the probability density and cumulative distribution functions of the \(\chi^2_d\) distribution. Then the following relations hold, 
    \begin{enumerate}[label=(\roman*)]
        \item \(tf_d(t) = d f_{d+2}(t)\),
        \item \(t^2 f_d(t) = d(d+2)f_{d+4}(t)\),
        \item \(f_{d+2}'(t) = \frac{f_d(t) - f_{d+2}(t)}{2}\),
        \item \((1 - F_{d+2}(t)) - (1 - F_d(t)) = \frac{e^{-t/2}t^{d/2}}{2^{d/2}\Gamma(d/2+1)}\).
    \end{enumerate}
\end{lemma}

\begin{lemma}
    Let \(f_d\) and \(F_d\) respectively denote the probability density and cumulative distribution functions of the \(\chi^2_d\) distribution. Suppose \(d > 2\). If \(t \geq d\), then \(2f_d(t) \leq 1 - F_d(t)\). 
\end{lemma}
\begin{proof}
    By Mean Value Theorem, we have for any \(r \geq t\), 
    \begin{align*}
        \inf_{x \geq d} \frac{\frac{\partial}{\partial x} (1 - F_d(x))}{2f_d'(x)} &\leq \frac{(1-F_d(t)) - (1-F_d(r))}{2f_d(t) - 2f_d(r)} \\
        &= \frac{\frac{1-F_d(t)}{2f_d(t)} - \frac{1-F_d(r)}{2f_d(t)}}{1 - \frac{f_d(r)}{f_d(t)}}.
    \end{align*}
    Consider that \(\lim_{r \to \infty} 1 - F_d(r) = \lim_{r \to \infty} f_d(r) = 0\). So taking \(r \to \infty\) yields 
    \begin{equation*}
        \inf_{x \geq d} \frac{\frac{\partial}{\partial x}(1 - F_d(x))}{2f_d'(x)} \leq \frac{1-F_d(t)}{2f_d(t)}.
    \end{equation*}
    We now evaluate the infimum on the left-hand side. For \(x \geq d\), consider that an application of Lemma \ref{lemma:johnstone_chisquare} gives 
    \begin{align*}
        \frac{\frac{\partial}{\partial x}(1-F_d(x))}{2f_d'(x)} &= -\frac{f_d(x)}{f_{d-2}(x) - f_d(x)} \\
        &= -\frac{1}{\frac{f_{d-2}(x)}{f_d(x)}-1} \\
        &= -\frac{1}{\frac{xf_{d-2}(x)}{xf_d(x)}-1}\\
        &= -\frac{1}{\frac{d-2}{x} - 1} \\
        &= \frac{1}{1-\frac{d-2}{x}}.
    \end{align*}
    Since \(x \geq d > d-2\), it follows that 
    \begin{equation*}
        \inf_{x \geq d} \frac{\frac{\partial}{\partial x}(1-F_d(x))}{2f_d'(x)} = 1.
    \end{equation*}
    Thus we can immediately conclude \(2f_d(t) \leq 1 - F_d(t)\) as desired.
\end{proof}

\begin{lemma}\label{lemma:chisquare_cdf}
    Let \(F_d\) denote the cumulative distribution function of the \(\chi^2_d\) distribution. If \(x \geq 0\), then 
    \begin{equation*}
        1 - F_d(x) = \frac{1}{\Gamma\left(\frac{d}{2}\right)} \int_{x/2}^{\infty} t^{d/2 - 1} e^{-t} \, dt = Q\left(\frac{d}{2}, \frac{x}{2}\right)
    \end{equation*}
    where \(Q\) is the upper incomplete gamma function defined in Theorem \ref{thm:temme_expansion}.
\end{lemma}
\begin{proof}
    The result follows directly from a change of variables when integrating the probability density function. 
\end{proof}

\begin{lemma}\label{lemma:bulk_alpha}
    Suppose \(d\) is larger than a sufficiently large universal constant. There exist universal positive constants \(L^*\) and \(C^*\) such that the following holds. If \(1 \leq \beta \leq L^* \sqrt{d}\) and \(t^2 = \beta \sqrt{d}\), then 
    \begin{equation*}
        \alpha_t(d) \leq d + C^* \beta \sqrt{d}
    \end{equation*}
    where \(\alpha_{t}(d)\) is given by (\ref{def:alpha}). 
\end{lemma}
\begin{proof}
    For ease of notation, let \(r^2 = d + t^2\). Let \(f_d\) and \(F_d\) denote the probability density and cumulative distribution functions of the \(\chi^2_d\) distribution. We will select the universal constant \(L^*\) later on in the proof. By Lemma \ref{lemma:johnstone_chisquare}, we have 
    \begin{align*}
        \alpha_{t}(d) &= \frac{\int_{r^2}^{\infty} zf_d(z) \, dz}{1 - F_d(r^2)} \\
        &= \frac{\int_{r^2}^{\infty} df_{d+2}(z) \,dz}{1 - F_d(r^2)} \\
        &= d \left(1 + \frac{\left(1 - F_{d+2}(r^2)\right) - \left(1 - F_d(r^2)\right)}{1-F_d(r^2)} \right) \\
        &= d \left(1 + \frac{r^d e^{-r^2/2}}{2^{d/2} \Gamma\left(\frac{d}{2}+1\right)} \cdot \frac{1}{1-F_d(r^2)}\right) \\
        &= d \left(1 + \left(\frac{r^2}{2}\right)^{d/2} \frac{e^{-r^2/2}}{\Gamma\left(\frac{d}{2}+1\right)} \cdot \frac{1}{1-F_d(r^2)}\right)
    \end{align*}
    Rearranging terms and invoking Stirling's approximation (which states \(\Gamma(x+1) \sim \sqrt{2\pi x}\left(\frac{x}{e}\right)^x\) as \(x \to \infty\)) yields 
    \begin{equation*}
        \left(\frac{r^2}{2}\right)^{d/2} \frac{e^{-r^2/2}}{\Gamma\left(\frac{d}{2}+1\right)} \leq \frac{1+c}{\sqrt{\pi d}} \exp\left(\frac{d}{2} \log\left(\frac{r^2}{2}\right) - \frac{r^2}{2} - \frac{d}{2}\log\left(\frac{d}{2}\right) + \frac{d}{2}\right)
    \end{equation*}
    for a universal constant \(c > 0\) since \(d\) is larger than a sufficiently large universal constant. Applying Lemma \ref{lemma:chisquare_cdf} and Corollary \ref{thm:temme_order_1} yields 
    \begin{equation*}
        1 - F_d(r^2) = Q\left(\frac{d}{2}, \frac{r^2}{2}\right) \geq \left(1 - \Phi(\eta\sqrt{a})\right) - \frac{e^{-a\eta^2/2}}{\sqrt{2\pi}} \cdot \frac{c^{**}}{\sqrt{d}}
    \end{equation*}
    where \(c^{**}\) is a universal constant. Observe that \(\eta \sqrt{a}\) is larger than a sufficiently large universal constant since \(d\) is larger than a sufficiently large universal constant. Using the fact that \(1 - \Phi(x) = \frac{1}{x\sqrt{2\pi}} e^{-x^2/2} \left(1 + o(1)\right)\) as \(x \to \infty\), we have 
    \begin{equation*}
        1 - F_d(r^2) \geq \frac{1}{\sqrt{2\pi}} e^{-\frac{a\eta^2}{2}} \left(\frac{c^*}{\eta\sqrt{a}} - \frac{c^{**}}{\sqrt{d}}\right)
    \end{equation*}
    for a universal positive constant \(c^*\). Consider that 
    \begin{equation*}
        \frac{\eta^2}{2} = \mu - \log(1 + \mu) = \frac{\frac{r^2}{2}}{\frac{d}{2}} - 1 - \log\left(\frac{r^2}{2}\right) + \log\left(\frac{d}{2}\right).
    \end{equation*}
    Consequently, 
    \begin{equation*}
        -\frac{\eta^2 a}{2} = - \left( \frac{r^2}{2} - \frac{d}{2} - \frac{d}{2}\log\left(\frac{r^2}{2}\right) + \frac{d}{2}\log\left(\frac{d}{2}\right) \right) = -\frac{r^2}{2} + \frac{d}{2} + \frac{d}{2} \log\left(\frac{r^2}{2}\right) - \frac{d}{2}\log\left(\frac{d}{2}\right). 
    \end{equation*}
    Therefore, we have the bound 
    \begin{equation*}
        1 - F_d(r^2) \geq \frac{1}{\sqrt{2\pi}} \exp\left( -\frac{r^2}{2} + \frac{d}{2} + \frac{d}{2} \log\left(\frac{r^2}{2}\right) - \frac{d}{2}\log\left(\frac{d}{2}\right) \right) \cdot \left( \frac{c^{*}}{\eta \sqrt{a}} - \frac{c^{**}}{\sqrt{d}} \right).
    \end{equation*}
    Consider further that the inequality \(x - \log(1 + x) \leq \frac{x^2}{2} \leq x^2\) gives us 
    \begin{align*}
        \eta\sqrt{a} = \sqrt{\frac{d}{2}} \cdot \sqrt{2 \left(\frac{\beta}{\sqrt{d}} - \log\left(1 + \frac{\beta}{\sqrt{d}}\right)\right)} = \sqrt{d\left(\frac{\beta}{\sqrt{d}} - \log\left(1 + \frac{\beta}{\sqrt{d}}\right)\right)} \leq \beta.
    \end{align*}
    Hence, we have 
    \begin{equation*}
        1 - F_d(r^2) \geq \frac{1}{\sqrt{2\pi}} \exp\left(-\frac{r^2}{2} + \frac{d}{2} + \frac{d}{2}\log\left(\frac{r^2}{2}\right) - \frac{d}{2}\log\left(\frac{d}{2}\right) \right) \left(\frac{c^*}{\beta} - \frac{c^{**}}{\sqrt{d}}\right).
    \end{equation*}
    Let us take \(L^* := \frac{c^*}{2c^{**}}\). With this choice, we have \(\beta \leq \sqrt{d} \cdot \frac{c^*}{2c^{**}}\) and so \(\left(\frac{c^*}{\beta} - \frac{c^{**}}{\sqrt{d}}\right) \geq \frac{c^*}{2\beta}\). Therefore, 
    \begin{equation*}
        \left(\frac{r^2}{2}\right)^{d/2} \frac{e^{-r^2/2}}{\Gamma\left(\frac{d}{2}+1\right)} \cdot \frac{1}{1-F_d(r^2)} \leq \frac{1+c}{\sqrt{\pi d}} \cdot \frac{2\beta \sqrt{2\pi}}{c^*} \cdot \frac{\exp\left(\frac{d}{2} \log\left(\frac{r^2}{2}\right) - \frac{r^2}{2} - \frac{d}{2}\log\left(\frac{d}{2}\right) + \frac{d}{2}\right)}{\exp\left(-\frac{r^2}{2} + \frac{d}{2} + \frac{d}{2}\log\left(\frac{r^2}{2}\right) - \frac{d}{2}\log\left(\frac{d}{2}\right) \right)} \leq C^* \frac{\beta}{\sqrt{d}}
    \end{equation*}
    for a universal positive constant \(C^*\). Thus we have 
    \begin{equation*}
        \alpha_{t}(d) \leq d\left(1 + \frac{C^*\beta}{\sqrt{d}}\right) = d + C^* \beta \sqrt{d}
    \end{equation*}
    as desired. 
\end{proof}

\begin{theorem}[Uniform expansion of the incomplete gamma function \cite{temme_uniform_1975}]\label{thm:temme_expansion}
    For \(a > 0\) and \(x \geq 0\) real numbers, define the upper incomplete gamma function 
    \begin{equation*}
        Q(a, x) := \frac{1}{\Gamma(a)} \int_{x}^{\infty} t^{a-1} e^{-t} \, dt. 
    \end{equation*}
    Further define \(\lambda = \frac{x}{a}, \mu = \lambda - 1\), and \(\eta = \sqrt{2(\mu-\log(1+\mu))}\). Then \(Q\) admits an asymptotic series expansion in \(a\) which is uniform in \(\eta \in \R\). In other words, for any integer \(N \geq 0\), we have 
    \begin{equation*}
        Q(a, x) = \left(1 - \Phi\left(\eta\sqrt{a}\right)\right) + \frac{e^{-\frac{a\eta^2}{2}}}{\sqrt{2\pi a}} \sum_{k=0}^{N} c_k(\eta) a^{-k} + \Rem_{N}(a, \eta)
    \end{equation*}
    where the remainder term satisfies
    \begin{equation*}
        \lim_{a \to \infty} \sup_{\eta \in \R} \left| \frac{\Rem_N(a, \eta)}{\frac{e^{-\frac{a\eta^2}{2}}}{\sqrt{2\pi a}} c_N(\eta) a^{-N}} \right| = 0.
    \end{equation*}
    Here, \(\Phi\) denotes the cumulative distribution function of the standard normal distribution. 
\end{theorem}

\begin{theorem}[Theorem 1 in \cite{temme_asymptotic_1979}]\label{thm:temme_coef}
    Consider the setting of Theorem \ref{thm:temme_expansion}. The coefficient \(c_0\) is given by \(c_0(\eta) = \frac{1}{\mu} - \frac{1}{\eta}\). 
\end{theorem}

\begin{corollary}\label{thm:temme_order_1}
    Consider the setting of Theorem \ref{thm:temme_expansion}. For any \(a > 0\) and \(x \geq 0\), we have 
    \begin{equation*}
        Q(a, x) = \left(1 - \Phi\left(\eta\sqrt{a}\right)\right) + \frac{e^{-\frac{a\eta^2}{2}}}{\sqrt{2\pi a}} \left(\frac{1}{\mu} - \frac{1}{\eta}\right) + \Rem_0(a, \eta)
    \end{equation*}
    where 
    \begin{equation*}
        \lim_{a \to \infty} \sup_{\eta \in \R} \left| \frac{\Rem_0(a, \eta)}{\frac{e^{-\frac{a\eta^2}{2}}}{\sqrt{2\pi a}} \left(\frac{1}{\mu} - \frac{1}{\eta}\right)} \right| = 0.
    \end{equation*}
    Consequently, if \(a\) is larger than some universal positive constant and \(\mu > 0\), we have 
    \begin{equation*}
        \left|Q(a, x) - \left(1 - \Phi\left(\eta\sqrt{a}\right)\right)\right| \leq \frac{Ce^{-\frac{a\eta^2}{2}}}{\sqrt{2\pi a}} 
    \end{equation*}
    where \(C\) is some universal positive constant. 
\end{corollary}
\begin{proof}
    The first two displays follow exactly from Theorems \ref{thm:temme_expansion} and \ref{thm:temme_coef}. To show the final display, we must show that \(\left|\frac{1}{\mu} - \frac{1}{\eta}\right| \lesssim 1\) whenever \(\mu > 0\). First, consider the Taylor expansion
    \begin{equation*}
        \log(1+\mu) = \mu - \frac{\mu^2}{2(1+\xi)^2} 
    \end{equation*}
    where \(\xi\) is some point between \(0\) and \(\mu\). Therefore, 
    \begin{equation*}
        \sqrt{2(\mu - \log(1+\mu))} = \frac{\mu}{1+\xi}.
    \end{equation*}
    Thus 
    \begin{align*}
        \left| \frac{1}{\mu} - \frac{1}{\eta}\right| = \left|\frac{1}{\mu} - \frac{1+\xi}{\mu}\right| = \frac{\xi}{\mu} \leq 1
    \end{align*}
    since \(\xi\) is between \(0\) and \(\mu\). Therefore, the final display in the statement of the Corollary follows by taking \(a\) to be larger than some universal constant and taking \(C\) to be a large enough universal constant. 
\end{proof}

\begin{lemma}\label{lemma:conditional_2}
    Let \(Z \sim N(0, I_d)\). If \(t \geq 0\), then 
    \begin{equation*}
        E(||Z||^2 \,|\, ||Z||^2 \geq d + t^2) \leq d + C^*(t^2 \vee d)
    \end{equation*}
    where \(C^*\) is a positive universal constant. 
\end{lemma}
\begin{proof}
    We will choose a universal constant \(L \geq 1\) at the end of the proof, so for now we leave it as undetermined. Let \(f_d\) denote the probability density function of the \(\chi^2_d\) distribution. Observe that 
    \begin{align*}
        E(||Z||^2 \,|\, ||Z||^2 \geq d + t^2) &= E(||Z||^2 \mathbbm{1}_{\{||Z||^2 \leq d + Lt^2\}} \,|\, ||Z||^2 \geq d + t^2) + E(||Z||^2 \mathbbm{1}_{\{||Z||^2 > d + Lt^2\}} \,|\, ||Z||^2 \geq d + t^2) \\
        &\leq d + Lt^2 + \sqrt{E(||Z||^4 \,|\, ||Z||^2 \geq d + t^2)} \sqrt{P\{||Z||^2 > d + Lt^2\}}.
    \end{align*}
    By Corollary \ref{corollary:chisquare_tail}, there exists a universal constant \(c_1\) such that 
    \begin{equation*}
        \sqrt{P\{||Z||^2 > d + Lt^2\}} \leq \sqrt{2} \exp\left(-Lc_1 \min\left(\frac{t^4}{d}, t^2\right)\right).
    \end{equation*}
    Here we have used \(L \geq 1\). Further consider that 
    \begin{align*}
        E(||Z||^4\,|\, ||Z||^2 \geq d + t^2) &= \frac{\int_{d+t^2}^{\infty} z^2 f_d(z) \, dz}{P\{||Z||^2 \geq d + t^2\}} \\
        &\leq \frac{E(||Z||^4)}{P\{||Z||^2 \geq d + t^2\}} \\
        &= \frac{d(d+2)}{P\{||Z||^2 \geq d + t^2\}}.
    \end{align*}
    Therefore, by Lemma \ref{lemma:anru_chisquare} we have 
    \begin{equation*}
        E(||Z||^2 \,|\, ||Z||^2 \geq d + t^2) \leq d + Lt^2 + d C \cdot \frac{\exp\left( - Lc_1 \min\left(\frac{t^4}{d}, t^2\right) \right)}{\exp\left(-c_2 \min\left(\frac{t^4}{d}, t^2\right)\right)}
    \end{equation*}
    where \(C\) and \(c_2\) are universal positive constants. Taking \(L := \frac{c_2}{c_1} \vee 1\) completes the proof. 
\end{proof}

\begin{lemma}\label{lemma:conditional_4}
    Let \(Z \sim N(0, I_d)\). If \(t \geq 0\), then 
    \begin{equation*}
        E(||Z||^4 \,|\, ||Z||^2 \geq d + t^2) \leq d^2 + C^*(t^4 \vee d^2)
    \end{equation*}
    where \(C^*\) is a positive universal constant. 
\end{lemma}
\begin{proof}
    We will choose a universal constant \(L \geq 1\) at the end of the proof, so for now we leave it as undetermined. Let \(f_d\) denote the probability density function of the \(\chi^2_d\) distribution. Consider 
    \begin{align*}
        &E(||Z||^4 \,|\, ||Z||^2 \geq d + t^2) \\
        &= E(||Z||^4 \mathbbm{1}_{\{||Z||^2 \leq d + Lt^2\}} \,|\, ||Z||^2 \geq d + t^2) + E(||Z||^4 \mathbbm{1}_{\{||Z||^2 \leq d + Lt^2\}} \,|\, ||Z||^2 \geq d + t^2) \\
        &\leq d^2 + 2Lt^2 d + L^2 t^4 + \sqrt{E(||Z||^8 \,|\, ||Z||^2 \geq d + t^2)} \sqrt{P\{||Z||^2 > d + Lt^2\}}.
    \end{align*}
    By Corollary \ref{corollary:chisquare_tail}, there exists a universal positive constant \(c_1\) such that 
    \begin{equation*}
        \sqrt{P\{||Z||^2 > d + Lt^2\}} \leq \sqrt{2}\exp\left(-Lc_1 \min\left(\frac{t^4}{d}, t^2\right)\right). 
    \end{equation*}
    Here we have used \(L \geq 1\). Further consider 
    \begin{align*}
        E(||Z||^8 \,|\, ||Z||^2 \geq d + t^2) &= \frac{\int_{d+t^2}^{\infty} z^4 f_d(z) \, dz}{P\{||Z||^2 \geq d + t^2\}} \\
        &\leq \frac{E(||Z||^8)}{P\{||Z||^2 \geq d + t^2\}} \\
        &= \frac{d(d+2)(d+4)(d+6)}{P\{||Z||^2 \geq d + t^2\}}. 
    \end{align*}
    Hence, by Lemma \ref{lemma:anru_chisquare} we have 
    \begin{equation*}
        E(||Z||^4 \,|\, ||Z||^2 \geq d + t^2) \leq d^2 + 2Lt^2d + L^2t^4 + d^2 C \frac{\exp\left(-Lc_1 \min\left(\frac{t^4}{d}, t^2\right)\right)}{\exp\left(-c_2\min\left(\frac{t^4}{d}, t^2\right)\right)}
    \end{equation*}
    where \(C > 0\) is a universal constant. Taking \(L > \frac{c_2}{c_1} \vee 1\) completes the proof.
\end{proof}

\begin{lemma}\label{lemma:conditional_var}
    Let \(Z \sim N(0, I_d)\). If \(t \geq 0\), then 
    \begin{equation*}
        \Var(||Z||^2 \,|\, ||Z||^2 \geq d + t^2) \lesssim t^4 + de^{-C\min\left(\frac{t^4}{d}, t^2\right)}
    \end{equation*}
    for some universal positive constant \(C\). 
\end{lemma}
\begin{proof}
    We will use a universal constant \(L \geq 1\) in our proof; a choice for it will be made at the end. Let \(f_d\) denote the probability density of the \(\chi^2_d\) distribution. Observe that 
    \begin{align*}
        &\Var(||Z||^2 \,|\, ||Z||^2 \geq d + t^2) \\
        &= \Var(||Z||^2 - d \,|\, ||Z||^2 \geq d + t^2) \\
        &\leq E((||Z||^2 - d)^2 \,|\, ||Z||^2 \geq d + t^2) \\
        &= E((||Z||^2 - d)^2 \mathbbm{1}_{\{||Z||^2 \leq d + Lt^2\}} \,|\, ||Z||^2 \geq d + t^2) + E((||Z||^2 - d)^2 \mathbbm{1}_{\{||Z||^2 > d + Lt^2\}} \,|\, ||Z||^2 \geq d + t^2) \\
        &\leq L^2 t^4 + \sqrt{E((||Z||^2 - d)^4 \,|\, ||Z||^2 \geq d + t^2)} \sqrt{P\{||Z||^2 \geq d + Lt^2\}}. 
    \end{align*}
    By Corollary \ref{corollary:chisquare_tail}, there exists a universal positive constant \(c_1\) such that 
    \begin{equation*}
        \sqrt{P\{|Z||^2 \geq d + Lt^2\}} \leq \sqrt{2} \exp\left(-c_1 L\min\left(\frac{t^4}{d}, t^2\right)\right).
    \end{equation*}
    Here, we have used \(L \geq 1\). With this term under control, now observe 
    \begin{align*}
        E((||Z||^2 - d)^4 \,|\, ||Z||^2 \geq d + t^2) &= \frac{\int_{d+t^2}^{\infty} (z - d)^2 f_d(z) \, dz}{P\{||Z||^2 \geq d + t^2\}} \\
        &\leq \frac{E((||Z||^2 - d)^4)}{P\{||Z||^2 \geq d + t^2\}} \\
        &= \frac{12d(d+4)}{P\{||Z||^2 \geq d + t^2\}}.
    \end{align*}
    Hence, by Lemma \ref{lemma:anru_chisquare} we have 
    \begin{equation*}
        \sqrt{E((||Z||^2 - d)^4 \,|\, ||Z||^2 \geq d + t^4)} \leq \frac{C^* d}{\exp\left(-c_2 \min\left(\frac{t^4}{d}, t^2\right)\right)}.
    \end{equation*}
    Taking \(L > \frac{c_2}{c_1} \vee 1\), we have 
    \begin{align*}
        \Var(||Z||^2 \,|\, ||Z||^2 \geq d + t^2) &\leq L^2t^4 + \frac{C^* d}{\exp\left(-c_2 \min\left(\frac{t^4}{d}, t^2\right)\right)} \cdot \sqrt{2} \exp\left(-c_1 L\min\left(\frac{t^4}{d}, t^2\right)\right) \\
        &\leq L^2 t^4 + d \cdot C^* \sqrt{2} \exp\left( -(c_1L - c_2) \min\left(\frac{t^4}{d}, t^2\right)\right). 
    \end{align*}
    which is precisely the desired result. 
\end{proof}

\section{Miscellaneous}
\subsection{Minimax}
\begin{lemma}\label{lemma:exact_Gamma}
    We have 
    \begin{equation*}
        \Gamma_{\mathcal{H}} = \mu_{\nu_{\mathcal{H}}} \vee \frac{\sqrt{(\nu_{\mathcal{H}}-1)\log\left(1 + \frac{p}{s^2}\right)}}{n}.
    \end{equation*}
    Consequently, \(\Gamma_{\mathcal{H}} \leq \mu_{\nu_{\mathcal{H}}-1}\).
\end{lemma}
\begin{proof}
    Consider that 
    \begin{align*}
        \Gamma_{\mathcal{H}} &= \max_{\nu < \nu_{\mathcal{H}}} \left\{ \mu_{\nu} \wedge \frac{\sqrt{\nu\log\left(1 + \frac{p}{s^2}\right)}}{n}\right\} \vee \max_{\nu \geq \nu_{\mathcal{H}}} \left\{\mu_{\nu} \wedge \frac{\sqrt{\nu \log\left(1 + \frac{p}{s^2}\right)}}{n}\right\} \\
        &= \max_{\nu < \nu_{\mathcal{H}}} \left\{ \frac{\sqrt{\nu\log\left(1 + \frac{p}{s^2}\right)}}{n}\right\} \vee \max_{\nu \geq \nu_{\mathcal{H}}} \left\{\mu_{\nu} \wedge \frac{\sqrt{\nu \log\left(1 + \frac{p}{s^2}\right)}}{n}\right\} \\
        &= \frac{\sqrt{(\nu_{\mathcal{H}} - 1)\log\left(1 + \frac{p}{s^2}\right)}}{n} \vee \mu_{\nu_{\mathcal{H}}}.
    \end{align*}
    Here, we have used the ordering of the eigenvalues and the fact that \(\mu_{\nu_{\mathcal{H}}} \leq \frac{\sqrt{\nu_{\mathcal{H}} \log\left(1+ \frac{p}{s^2}\right)}}{n}\) to obtain the second term. The bound \(\Gamma_{\mathcal{H}} \leq \mu_{\nu_{\mathcal{H}}-1}\) follows immediately from \(\mu_{\nu_{\mathcal{H}}} \leq \mu_{\nu_{\mathcal{H}}-1}\) and \(\frac{\sqrt{(\nu_{\mathcal{H}}-1)\log\left(1 + \frac{p}{s^2}\right)}}{n} \leq \mu_{\nu_{\mathcal{H}}-1}\) by definition of \(\nu_{\mathcal{H}}\). 
\end{proof}

\begin{proof}[Proof of Lemma \ref{lemma:fixed_point_equiv}]
    The result is a special case of Lemma \ref{lemma:Gamma_nu_equiv_adapt}.
\end{proof}

\subsection{Adaptation}

\begin{lemma}\label{lemma:exact_Gamma_ada}
    Suppose \(a > 0\). We have 
    \begin{equation*}
        \Gamma_{\mathcal{H}}(s, a) = \mu_{\nu_{\mathcal{H}}(s, a)} \vee \frac{\sqrt{(\nu_{\mathcal{H}}(s, a) - 1)\log\left(1 + \frac{pa}{s^2}\right)}}{n}. 
    \end{equation*}
    Consequently, \(\Gamma_{\mathcal{H}}(s, a) \leq \mu_{\nu_{\mathcal{H}}(s, a) - 1}\). 
\end{lemma}
\begin{proof}
    Consider that 
    \begin{align*}
        \Gamma_{\mathcal{H}}(s, a) &= \max_{\nu < \nu_{\mathcal{H}}(s,a)} \left\{ \mu_{\nu} \wedge \frac{\sqrt{\nu \log\left(1 + \frac{pa}{s^2}\right)}}{n} \right\} \vee \max_{\nu \geq \nu_{\mathcal{H}}(s,a)} \left\{ \mu_{\nu} \wedge \frac{\sqrt{\nu \log\left(1 + \frac{pa}{s^2}\right)}}{n}  \right\} \\
        &= \max_{\nu < \nu_{\mathcal{H}}(s,a)} \left\{ \frac{\sqrt{\nu \log\left(1 + \frac{pa}{s^2}\right)}}{n} \right\} \vee \max_{\nu \geq \nu_{\mathcal{H}}(s,a)} \left\{ \mu_{\nu} \wedge \frac{\sqrt{\nu\log\left(1 + \frac{pa}{s^2}\right)}}{n} \right\} \\
        &= \frac{\sqrt{(\nu_{\mathcal{H}}(s,a) - 1) \log\left(1 + \frac{pa}{s^2}\right)}}{n} \vee \mu_{\nu_{\mathcal{H}}(s,a)}. 
    \end{align*}
    Here, we have used the ordering of the eigenvalues and the fact that \(\mu_{\nu_{\mathcal{H}}(s,a)} \leq \frac{\sqrt{\nu_{\mathcal{H}}(s,a) \log\left(1 + \frac{pa}{s^2}\right)}}{n}\) to obtain the second term. To bound \(\Gamma_{\mathcal{H}} \leq \mu_{\nu_{\mathcal{H}}(s,a) - 1}\) follows immediately from \(\mu_{\nu_{\mathcal{H}}} \leq \mu_{\nu_{\mathcal{H}}(s,a) - 1}\) and \(\frac{\sqrt{(\nu_{\mathcal{H}}(s,a) - 1)\log\left(1 + \frac{pa}{s^2}\right)}}{n} \leq \mu_{\nu_{\mathcal{H}}(s,a) - 1}\) by definition of \(\nu_{\mathcal{H}}(s,a)\). 
\end{proof}

\begin{proof}[Proof of Lemma \ref{lemma:Gamma_nu_equiv_adapt}]
    We first prove the second inequality. Since \(\log\left(1 + \frac{pa}{s^2}\right) \leq \frac{n}{2}\) and \(\mu_1 = 1\), it immediately follows that \(\nu_{\mathcal{H}}(s,a) \geq 2\). Therefore, 
    \begin{align*}
        \frac{\sqrt{\nu_{\mathcal{H}}(s,a)\log\left(1 + \frac{pa}{s^2}\right)}}{n} &\leq \sqrt{2} \frac{\sqrt{(\nu_{\mathcal{H}}(s,a) - 1) \log\left(1 + \frac{pa}{s^2}\right)}}{n} \\
        &\leq \sqrt{2}\left( \mu_{\nu_{\mathcal{H}}(s,a)} \vee \frac{\sqrt{(\nu_{\mathcal{H}}(s,a) - 1)\log\left(1 + \frac{pa}{s^2}\right)}}{n} \right) \\ 
        &\leq \sqrt{2} \Gamma_{\mathcal{H}}(s, a)
    \end{align*}
    as desired. 

    Now we prove the first inequality. For any \(\nu \geq \nu_{\mathcal{H}}(s,a)\), we have by the ordering of the eigenvalues 
    \begin{equation*}
        \mu_{\nu} \wedge \frac{\sqrt{\nu \log\left(1 + \frac{pa}{s^2}\right)}}{n} \leq \mu_{\nu} \leq \mu_{\nu_{\mathcal{H}}}(s, a) \leq \frac{\sqrt{\nu_{\mathcal{H}}(s,a) \log\left(1 + \frac{pa}{s^2}\right)}}{n}.
    \end{equation*}
    For any \(\nu < \nu_{\mathcal{H}}(s,a)\), by definition of \(\nu_{\mathcal{H}}(s,a)\) it follows that 
    \begin{equation*}
        \mu_{\nu} \wedge \frac{\sqrt{\nu\log\left(1 + \frac{pa}{s^2}\right)}}{n} = \frac{\sqrt{\nu\log\left(1 + \frac{pa}{s^2}\right)}}{n} \leq \frac{\sqrt{\nu_{\mathcal{H}}(s,a)\log\left(1 + \frac{pa}{s^2}\right)}}{n}.
    \end{equation*}
    Since for all \(\nu\) we have shown 
    \begin{equation*}
        \mu_\nu \wedge \frac{\sqrt{\nu \log\left(1 + \frac{pa}{s^2}\right)}}{n} \leq \frac{\sqrt{\nu_{\mathcal{H}}(s,a) \log\left(1 + \frac{pa}{s^2}\right)}}{n},
    \end{equation*}
    taking max over \(\nu\) yields 
    \begin{equation*}
        \Gamma_{\mathcal{H}}(s, a) \leq \frac{\sqrt{\nu_{\mathcal{H}}(s,a)\log\left(1 + \frac{pa}{s^2}\right)}}{n}
    \end{equation*}
    as desired. 
\end{proof}

\begin{proof}[Proof of Lemma \ref{lemma:script_AV_equiv}]
    For ease of notation, let us write \(a = \mathscr{A}_{\mathcal{H}}\). By the assumption \(\log\left(1 + pa\right) \leq \frac{n}{2}\) and \(\mu_1 = 1\) we have \(\nu_{\mathcal{H}}(s,a) \geq 2\) for all \(s\). By definition of \(\nu_{\mathcal{H}}(s,a)\), we have 
    \begin{equation*}
        \mu_{\nu_{\mathcal{H}}(s,a)} \leq \frac{\sqrt{\nu_{\mathcal{H}}(s,a)\log\left(1 + \frac{pa}{s^2}\right)}}{n}
    \end{equation*}
    and
    \begin{equation*}
        \mu_{\nu_{\mathcal{H}}(s,a) - 1} > \frac{\sqrt{(\nu_{\mathcal{H}}(s,a)-1)\log\left(1 + \frac{pa}{s^2}\right)}}{n}
    \end{equation*}
    Define 
    \begin{equation*}
        \delta_s := \inf\left\{ \delta > 0 : \mu_{\nu(s, a) - 1} \leq \frac{\sqrt{(\nu_{\mathcal{H}}(s,a) - 1)\log\left(1 + \frac{p(a+\delta)}{s^2}\right)}}{n}\right\}. 
    \end{equation*}
    It is clear \(\delta_s > 0\) for each \(s\) since the inequality in the penultimate display is strict. Moreover, observe that \(\nu_{\mathcal{H}}(s, a+\delta) > \nu_{\mathcal{H}}(s,a) - 1\) for all \(\delta < \delta_s\) by definition of \(\nu_{\mathcal{H}}\). Moreover, since \(\nu_{\mathcal{H}}\) is decreasing in its second argument, it thus follows that \(\nu_{\mathcal{H}}(s, a+\delta) = \nu_{\mathcal{H}}(s,a)\) for all \(\delta < \delta_s\). Letting \(\delta^* := 1 \wedge \min_{s \in [p]} \delta_s\), it immediately follows that \(\mathscr{V}_{a} = \mathscr{V}_{a + \delta}\) for all \(\delta < \delta^*\). By definition of \(\mathscr{A}_{\mathcal{H}}\), it follows that for any \(\delta < \delta^*\)
    \begin{equation*}
        \mathscr{A}_{\mathcal{H}} = a \leq \log(e|\mathscr{V}_{a}|)= \log(e|\mathscr{V}_{a + \delta}|) < a + \delta \leq 2\mathscr{A}_{\mathcal{H}} 
    \end{equation*}
    where we have used \(\mathscr{A}_{\mathcal{H}} \geq 1\) and \(\delta \leq 1\) to obtain the final inequality. The proof is complete.
\end{proof}

\subsection{Technical}

\begin{proposition}[Ingster-Suslina method \cite{ingster_nonparametric_2003}]\label{prop:Ingster_Suslina}
    Suppose \(\Sigma \in \R^{d \times d}\) is a positive definite matrix and \(\Theta \subset \R^d\) is a parameter space. If \(\pi\) is a probability distribution supported on \(\Theta\), then 
    \begin{equation*}
        \chi^2\left(\left.\left. \int N(\theta, \Sigma) \, \pi(d\theta) \right|\right| N(0, \Sigma)\right) = E\left(\exp\left(\left\langle \theta, \Sigma^{-1} \widetilde{\theta} \right\rangle\right)\right) - 1
    \end{equation*}
    where \(\theta, \widetilde{\theta} \overset{iid}{\sim} \pi\) and \(\chi^2(\cdot || \cdot)\) denotes the \(\chi^2\)-divergence defined in Section \ref{section:notation}.
\end{proposition}

\begin{lemma}[\cite{tsybakov_introduction_2009}]
    If \(P, Q\) are probability measures on a measurable space \((\mathcal{X}, \mathcal{A})\) with \(P \ll Q\), then \(d_{TV}(P, Q) \leq \frac{1}{2}\sqrt{\chi^2(P||Q)}\). 
\end{lemma}

\begin{lemma}\label{lemma:hypgeom}
    If \(Y\) is distributed according to the hypergeometric distribution with probability mass function \(P\{Y = k\} = \frac{\binom{s}{k} \binom{n-s}{t-k}}{\binom{n}{t}}\) for \(0 \leq k \leq s \wedge t\), then \(E(\exp(\lambda Y)) \leq \left(1 - \frac{s}{n} + \frac{s}{n} e^{\lambda} \right)^t\) for \(\lambda > 0\). 
\end{lemma}

\begin{corollary}[\cite{collier_minimax_2017}]\label{corollary:hypergeometric}
    If \(Y\) is distributed according to the hypergeometric distribution with probability mass function \(P\{Y = k\} = \frac{\binom{s}{k}}{\binom{n-s}{s-k}}{\binom{n}{s}}\) for \(0 \leq k \leq s \leq n\), then \(E(Y) = \frac{s^2}{n}\) and \(E(\exp(\lambda Y)) \leq \left(1 - \frac{s}{n} + \frac{s}{n} e^{\lambda}\right)^s\) for \(\lambda > 0\). 
\end{corollary}

\end{document}